\documentclass[11pt,reqno]{amsart} 
\usepackage[utf8]{inputenc} 
\usepackage{bbm}

\usepackage[margin=1in]{geometry} 
\usepackage{graphicx}
 \usepackage[parfill]{parskip}
 
\usepackage{booktabs} 
\usepackage{array} 
\usepackage{paralist} 
\usepackage{verbatim} 
\usepackage{subfig} 
\usepackage{mathrsfs}
\usepackage{amssymb}
\usepackage{amsthm}
\usepackage{amsmath,amsfonts,amssymb,esint}
\usepackage{graphics,color}
\usepackage{enumerate}
\usepackage{mathtools,centernot}
\usepackage{cases}

\usepackage{xcolor}
\usepackage{xfrac}

\newtheorem{theorem}{Theorem}[section]
\newtheorem{lemma}[theorem]{Lemma}
\newtheorem{corollary}[theorem]{Corollary}
\newtheorem{definition}[theorem]{Definition}
\newtheorem{proposition}[theorem]{Proposition}
\newtheorem{remark}[theorem]{Remark}

\newtheorem*{remark 1}{Remark}

\numberwithin{equation}{section}

\newcommand{\norm}[1]{\left\|#1\right\|}
\newcommand{\abs}[1]{\left|#1\right|}

\DeclarePairedDelimiter{\ceil}{\lceil}{\rceil}

\newcommand{\T}{\ensuremath{\mathbb{T}}}
\newcommand*{\R}{\ensuremath{\mathbb{R}}}

\newcommand*{\N}{\ensuremath{\mathbb{N}}}
\newcommand*{\Z}{\ensuremath{\mathbb{Z}}}
\newcommand*{\C}{\ensuremath{\mathbb{C}}}

\newcommand*{\tr}{\ensuremath{\mathrm{tr\,}}}

\newcommand*{\Id}{\ensuremath{\mathrm{Id}}}

\newcommand*{\pal}{\ensuremath{
\partial^{\boldsymbol{\alpha}}}}
\newcommand*{\pab}{\ensuremath{
\partial^{\boldsymbol{\beta}}}}

\newcommand*{\RR}{\ensuremath{\mathcal{R}}}

\newcommand*{\pat}{\ensuremath{\partial_t}}
\newcommand*{\pas}{\ensuremath{\partial_s}}

\usepackage{color, graphicx}
\usepackage{mathrsfs, dsfont}

\usepackage{hyperref}

\def\div{\mathop{\rm div}\nolimits}    
\def\supp{\mathop{\rm supp}\nolimits}    
\def\curl{\mathop{\rm curl}\nolimits}    

\def\Xint#1{\mathchoice
{\XXint\displaystyle\textstyle{#1}}%
{\XXint\textstyle\scriptstyle{#1}}%
{\XXint\scriptstyle\scriptscriptstyle{#1}}%
{\XXint\scriptscriptstyle\scriptscriptstyle{#1}}%
\!\int}
\def\XXint#1#2#3{{\setbox0=\hbox{$#1{#2#3}{\int}$ }
\vcenter{\hbox{$#2#3$ }}\kern-.6\wd0}}

\def\dashint{\Xint-}

\DeclareFontEncoding{FMS}{}{}
\DeclareFontSubstitution{FMS}{futm}{m}{n}
\DeclareFontEncoding{FMX}{}{}
\DeclareFontSubstitution{FMX}{futm}{m}{n}
\DeclareSymbolFont{fouriersymbols}{FMS}{futm}{m}{n}
\DeclareSymbolFont{fourierlargesymbols}{FMX}{futm}{m}{n}
\DeclareMathDelimiter{\VE}{\mathord}{fouriersymbols}{152}{fourierlargesymbols}{147}

\title{Anomalous dissipation and Euler flows}

\author[J. Burczak]{Jan Burczak}
\address{Mathematisches Institut, Leipzig University, Augustusplatz 10, 04109 Leipzig, Germany and Max Planck Institute for Mathematics in the Sciences, Inselstraße 22-26, 04103 Leipzig, Germany}
\email{burczak@math.uni-leipzig.de}

\author[L. Sz{\'{e}}kelyhidi]{ L\'{a}szl\'{o} Sz{\'{e}}kelyhidi, Jr.}
\thanks{LSz gratefully acknowledges the support of the  Deutsche Forschungsgemeinschaft (DFG, German Research Foundation) through GZ SZ 325/2-1.}
\address{Max Planck Institute for Mathematics in the Sciences, Inselstraße 22-26, 04103 Leipzig, Germany}
\email{szekelyhidi@mis.mpg.de}

\author[B. Wu]{Bian Wu}
\address{Max Planck Institute for Mathematics in the Sciences, Inselstraße 22-26, 04103 Leipzig, Germany}
\email{bian.wu@mis.mpg.de}

\date{\today}
\thanks{The authors would like to thank Theodore Drivas and Massimo Sorella for  their helpful remarks on an earlier version of this paper.}

\begin{document}

\begin{abstract}
We show anomalous dissipation of scalars advected by weak solutions to the incompressible Euler equations with  $C^{(\sfrac{1}{3})^-}$ regularity, for an arbitrary initial datum in $\dot H^1 (\T^3)$. This is the first rigorous derivation of zeroth law of scalar turbulence, where the scalar is advected by solution to an equation of hydrodynamics (unforced and deterministic). As a byproduct of our method, we provide a typicality statement for the drift, and recover certain desired properties of turbulence, including a lower bound on scalar variance commensurate with the Richardson pair dispersion hypothesis.
\end{abstract}

\maketitle


\section{Introduction}

We consider the Cauchy problem for the linear advection-diffusion equation
\begin{equation}\label{e:advectiondiffusion}
\begin{split}
\partial_t\rho_{\kappa}+u\cdot\nabla \rho_\kappa-\kappa\Delta\rho_{\kappa}&=0,\\
\rho_{\kappa}|_{t=0}&=\rho_{in},
\end{split}
\end{equation}
on the $3$-dimensional flat torus $\T^3$. The parameter  $\kappa>0$ represents molecular diffusivity, and the quantity $\rho_{\kappa}: \T^3\times [0,T] \to \R$ is a passive scalar which is advected by a given divergence-free vector field $u: \T^3\times [0,T] \to \R^3$ for some time $T>0$. The energy equality reads 
\begin{equation*}
\frac12\|\rho_{\kappa}(T)\|_{L^2}^2=\frac12\|\rho_{in}\|_{L^2}^2-\kappa\int_0^T\|\nabla\rho_\kappa(t)\|^2_{L^2}\,dt.	
\end{equation*}
Even though the drift $u$ does not feature in this identity, it is well known that the nature of $u$ crucially affects the decay rate of $\frac12\|\rho_{\kappa}(T)\|_{L^2}^2$. Indeed, one typically expects advection by the flow to generate small scales which are then much more efficiently damped by the diffusion. In the context of scalar turbulence, where $u$ is the velocity field of a turbulent flow, the expectation is that the effect of small-scale generation via advection balances diffusion, so that the rate of dissipation of scalar variance
\begin{equation}\label{e:ad}
    \kappa\int_0^T\|\nabla\rho_\kappa(t)\|_{L^2}^2\,dt
\end{equation}
is bounded from below by a positive constant, independently of $\kappa>0$. This is referred to as \emph{anomalous dissipation}. The purpose of this work is to establish anomalous dissipation for a large class of drifts, which are themselves weak solutions of the incompressible Euler equations
\begin{equation}\label{e:Euler}
\begin{cases}
\partial_t u+u\cdot \nabla u+\nabla p=0,\\
\div u=0.
\end{cases}
\end{equation}
Our main result, in simplified form, reads as follows.
\begin{theorem}\label{t:main}
Let $0<T_0<\infty$ and $0<\beta<1/3$. There exists a weak solution $u\in C^{\beta}(\T^3\times [0,T_0])$ to \eqref{e:Euler} such that, for every non-zero initial datum $\rho_{in}\in H^1(\T^3)$ with zero mean, the family of unique solutions $\{\rho_{\kappa}\}_{\kappa>0}$ to \eqref{e:advectiondiffusion} satisfies for any $T\le T_0$
\begin{equation}\label{e:anomalousdissipation}
\limsup_{\kappa\to 0} \; \kappa\int_0^T\|\nabla\rho_\kappa\|_{L^2}^2\,dt\geq c_0 \|\rho_{in}\|_{L^2}^2,
\end{equation}
where $c_0>0$ depends only on $\beta$, $\frac{\|\rho_{in}\|_{L^2}}{\|\nabla \rho_{in}\|_{L^2}}$ and $T$. 
\end{theorem}
In order to prove this result, we develop an approach that is able to simultaneously handle two important but so far unrelated recent developments in the field: the construction of weak solutions of the incompressible Euler equations via convex integration, leading to a resolution of the celebrated Onsager's conjecture \cite{DSz13, Isett2018, BDSV}; and the work of S.~Armstrong and V.~Vicol on implementing iterated homogenization to construct incompressible time-dependent vector fields in 2D, exhibiting anomalous dissipation of scalar variance \cite{ArmstrongVicol}. Apart from highlighting the strong similarities at the technical level, this combination leads naturally to a more general statement than Theorem \ref{t:main}, which highlights two key features of the construction. The general statement is deferred to Section \ref{s:main_proof} (see Theorem \ref{t:main_general}), but for the purposes of this introduction let us briefly comment on the main features. 

\subsection{h-principle}\label{ss:hprinciple} First of all, our approach yields not just a single vector field, but a large class of such fields. Indeed, it is well known that there is a close analogy between weak solutions constructed via convex integration and Gromov's h-principle, a principle referring to a high level of non-uniqueness and flexibility. We refer the interested reader to one of the many surveys on the subject \cite{DeSz2012-survey,Sz-2014survey,DeSz2017-survey,BuckVicol2019-survey}. In particular, within the context of weak solutions of the Euler equations, one obtains certain density statements. In the setting of this paper this means:
\begin{itemize}
\item[{\bf Strong density:}] For any smooth Euler solution $\bar u$ there is sequence $\{u^{(n)}\}_{n \in \N}$ of Euler solutions with properties as in Theorem \ref{t:main}, such that $\|u^{(n)} - \bar u \|_{C^0 (\T^3\times [0,T])} \to 0$ as $n \to 0$. 
\item[{\bf Weak density:}] For any smooth solenoidal $\bar u$ there is sequence $\{u^{(n)}\}_{n \in \N}$ of Euler solutions with properties as in Theorem \ref{t:main}, such that $\|u^{(n)} - \bar u \|_{ C([0,T]; H^{-1} (\T^3))} \to 0$ as $n \to 0$.
\end{itemize}
The constant $c_0$ in Theorem \ref{t:main} depends in both of cases above additionally on $n$.

We remark that obtaining analogous statements within other constructions of vector fields inducing anomalous dissipation (see Section \ref{s:literature}) seems challenging.

\subsection{Continuous-in-time dissipation and Richardson hypothesis}\label{ss:dissipation} Secondly, as in the work of Armstrong-Vicol \cite{ArmstrongVicol}, in our construction anomalous dissipation occurs continuously in time, in agreement with the postulate of statistical stationarity in homogeneous isotropic turbulence. In fact, the constant $c_0$ of \eqref{e:anomalousdissipation} is computed explicitly in Theorem \ref{t:main_general} in terms of the data. In particular, the lower bound \eqref{e:main_const} in Theorem \ref{t:main_general} implies the following: for any $0<\beta<1/3$ there is $\varepsilon_\beta$ such that for any $0<\varepsilon<\varepsilon_\beta$ the statement of Theorem \ref{t:main} holds with constant $c_0$ satisfying
\begin{equation}\label{e:c_0explicit}
    c_0  \ge c_\varepsilon 
        \min \big\{ \ell_0^{-2p_\varepsilon}\ell_{in}^{2p_\varepsilon-2}T^{p_\varepsilon},
        \ell_{in}^{2\varepsilon} \big\},
\end{equation}
where $p_\varepsilon=\frac{2+\varepsilon}{1-\varepsilon\frac{2+\varepsilon}{1+\varepsilon} - (1+\varepsilon) \beta}$, $\ell_0 := \sup_{f \in \dot H^1 (\T^3)} \frac{\|f\|_{L^2}} {\|\nabla f\|_{L^2}}$, and 
$ \ell_{in}:=\frac{\|\rho_{in}\|_{L^2}} {\|\nabla \rho_{in}\|_{L^2}}$.
Observe that $p_\varepsilon\to \frac{2}{1-\beta}$ as $\varepsilon\to 0$, leading to the asymptotic rate $T^\frac{2}{1-\beta}$ in \eqref{e:c_0explicit}. This is in agreement with the Richardson hypothesis on pair dispersion in turbulence\footnote{We are grateful to Theodore Drivas for pointing out this connection to us.}. In modern language this hypothesis reads as follows. Denote by $(X^\kappa_t)_t$ the stochastic trajectories of the advection-diffusion \eqref{e:advectiondiffusion}. Then $\rho_{\kappa}(t) = \mathbb{E} (\rho_{in} \circ X^\kappa_t)$, and one can write $\|\rho_{in}\|_{L^2}^2=  \mathbb{E} \|\rho_{in} \circ X^\kappa_t\|_{L^2}^2$, thus
\begin{equation}
\begin{split}
\kappa\int_0^t\|\nabla\rho_\kappa\|_{L^2}^2\,ds &= \frac12 \left( \|\rho_{in}\|_{L^2}^2 -  \|\rho_{\kappa}(t) \|_{L^2}^2  \right)  = \frac12 \int_{\T^3} {\rm{Var}} (\rho_{in} \circ X^\kappa_t).
\end{split}
\end{equation}
In particular, when $\rho_{in}$ has uniform random space gradient with isotropic statistics $G \,\Id$ then, see  \cite{DriEyi17} formula (2.17), 
\[
 \int_{\T^3} {\rm{Var}} (\rho_{in} \circ X^\kappa_t)  = \frac14 G^2 \int_{\T^3} \mathbb{E}^{(1),(2)} |X^{\kappa, (1)}_t -X^{\kappa, (2)}_t|^2 =: \delta_t^2,
 \]
where superscripts ${(1),(2)}$ indicate two independent Brownian ensembles. Thus 
\[
\begin{split}
\dot\delta_t &\lesssim \left( \int_{\T^3} \mathbb{E}^{(1),(2)} |\dot X^{\kappa, (1)}_t -\dot X^{\kappa, (2)}_t|^2 \right)^{\frac12} = \left( \int_{\T^3} \mathbb{E}^{(1),(2)} |u \circ X^{\kappa, (1)}_t -u \circ X^{\kappa, (2)}_t|^2 \right)^{\frac12} \\
&\sim  \left( \int_{\T^3} \mathbb{E}^{(1),(2)} |X^{\kappa, (1)}_t -X^{\kappa, (2)}_t|^{2\beta} \right)^{\frac12} \lesssim \delta^\beta_t,
\end{split}
 \]
where the second line uses $C^{\beta}$ regularity of the drift. Assuming that in the above the inequalities $\lesssim$ can be replaced by $\sim$, we obtain the ODE $\dot\delta_t \sim \delta^\beta_t$, and deduce $\delta^2_t \sim T^\frac{2}{1-\beta}$. 
Observe that, in particular, the limit in the allowed regularity of the drift: $T^\frac{2}{1-\beta} \overset{\beta \to \sfrac13}{\to} T^3$ recovers the classical Richardson (or Richardson-Obukhov) scaling.

Let us now discuss in more detail the physical and mathematical context of our work.

\subsection{Turbulence}\label{ssec:tur}

In the 1949 paper \cite{Obu49} titled 'Structure of the Temperature Field in Turbulent Flow' A.~N.~Obukhov  writes
\begin{center}
    \textit{
In other words, the turbulent motion inside a thermally heterogeneous medium with gradients which are initially weak can contribute to the local gradients of temperature, which are subsequently smoothed out by the action of molecular heat conductivity.}
\end{center}

This statement is precisely a prediction that dissipation caused by little ('molecular') $\kappa$ present in \eqref{e:advectiondiffusion} can be amplified through advection $u \cdot \nabla$ by a turbulent velocity field $u$. 
The responsible mechanism  should be transfer of modes towards higher frequencies/smaller scales, where eventually diffusivity at the `molecular' length-scale steps in. This prediction of Obukhov has been further corroborated by phenomenological, experimental and numerical arguments, see the classical Yaglom \cite{Yag49}, Corrsin \cite{Cor51}, Batchelor \cite{Bat59}, 
and more recent \cite{SreSch10}, \cite{DonSreYou05}, \cite{ShrSig00}. 
Furthermore, the expected amplification of dissipation is supposed to be sufficient to keep the quantity \eqref{e:ad} bounded away from zero, thus causing ``anomalous dissipation''. 
Thus, it is strongly believed that scalar anomalous dissipation is a key feature and manifestation of turbulence.

On the other hand, it is argued that turbulent flows are irregular. More precisely, L.~Onsager conjectured  in his famous 1949 note on statistical hydrodynamics \cite{Ons49} that the threshold regularity for the validity of energy conservation in the class of H\"older continuous weak solutions of the Euler equations is the exponent $1/3$. Onsager’s interest in this issue came from an effort to explain the primary mechanism of energy dissipation in turbulence, one that persists even in the absence of viscosity. Thus, the implicit suggestion was that H\"older continuous weak solutions of the Euler equations may be an appropriate mathematical description of turbulent flows in the inviscid limit, corresponding to a form of \emph{ideal turbulence}. 

We recall that the central analytical model of passive scalar turbulence, the Kraichnan 1968 model \cite{Kra68}, is based on Gaussian random velocity fields consisting of spatially non-smooth velocities with a K41-type scaling of increments and white-noise correlation in time. This model is able to recover various scaling laws, including the Obukhov-Corrsin $k^{-5/3}$ spectrum in the inertial-convective regime (Schmidt number $=O(1)$) and the Batchelor $k^{-1}$ spectrum in the viscous-convective regime (Schmidt number $\gg 1$). In particular, Kraichnan's model is known to exhibit anomalous dissipation; for more details we refer to \cite{BGK1998,FGV2001,Sreeni2019, DriEyi17,Row23}. Lifting these results to a rigorous PDE setting has remained a central challenge in the field. 

In a recent breakthrough work \cite{BedBluPS22a,BedBluPS22b} J.~Bedrossian, A.~Blumenthal and S.~Punshon-Smith were able to show almost sure exponential mixing of Lagrangian trajectories and consequently verify Batchelor’s law on the cumulative power spectrum for velocity fields given by solutions of randomly forced 2D Navier-Stokes and hyperviscous 3D Navier-Stokes at finite Reynolds number. However, in these examples realizations of the velocity field are almost surely uniformly bounded in $C^1$ so that anomalous dissipation in the sense of \eqref{e:anomalousdissipation} is not possible (see below). In contrast, our focus in this work is on the complementary Obukhov-Corrsin range, where, in analogy with Onsager's conjecture and its relation to Kolmogorov's K41 theory, we study Obukhov's prediction  of anomalous dissipation in the context of ideal turbulence. 

\subsection{Anomalous dissipation in PDEs}\label{s:literature}
For any fixed initial datum, the unique solution $\rho_{\kappa}$ of \eqref{e:advectiondiffusion} converges weakly along subsequences $\kappa_j\to 0$ to a distributional solution of the transport equation $\partial_t\rho+u\cdot\nabla \rho = 0$. If the advecting velocity field $u$ is sufficiently regular (classically $u \in L^1 (W^{1, \infty})$, but $u \in L^1 (W^{1, 1})$ or even  $u \in L^1 (BV)$ suffices, see, respectively, \cite{DiPernaLions} and \cite{Amb04}), then the transport equation admits a uniquely defined measure-preserving Lagrangian flow map, which can be used to define a unique renormalized solution. In turn, the renormalized solution will be the unique inviscid limit \cite{BoCiCr2022} - however, we mention that other, non-renormalized solutions of the transport equation may nevertheless exist \cite{MoSz2018,MoSz2019}. Thus, for sufficiently regular velocity fields anomalous dissipation is excluded, since measure preservation implies conservation of the $L^2$ norm. 

On the other hand, for less regular velocity fields $u$ anomalous dissipation as well as non-uniqueness in the inviscid limit may occur, as has been shown in a number of examples, beginning with the work of T.~Drivas, T.~Elgindi, G.~Iyer and I.-J.~Jeong \cite{DriElgIyeJeo22}. In \cite{DriElgIyeJeo22} the authors constructed, for any initial datum $\rho_{in}\in H^2$ and any $\alpha<1$, a divergence-free velocity field $u\in L^1 ([0,T_0]; C^\alpha (\T^2))$ for which one has anomalous dissipation in the sense that \eqref{e:anomalousdissipation} holds for $T=T_0$. The construction of the advecting velocity $u$ is based on alternating shear flows with a carefully chosen sequence of length- and time-scales concentrating at the critical final time $t=T$. In particular, there is no anomalous dissipation for any time $T<T_0$. This construction was later substantially improved in \cite{ColCriSor23}, on the one hand by complementing with uniform estimates on the solution of the type $\rho\in L^{\infty}([0,T_0]; C^\beta (\T^2))$ with $\alpha+2\beta<1$ (in agreement with the power-law prediction of Obukhov-Corrsin) and on the other hand by showing non-uniqueness of the inviscid limit. Recently in \cite{ElgLis23}, Elgindi and Liss have constructed a divergence-free vector field which satisfies $u \in C^\infty ([0,T]; C^\alpha (\T^2))$ for every $\alpha<1$ and which gives rise to anomalous dissipation in equation \eqref{e:advectiondiffusion} for every smooth initial datum; in fact $u$ is merely logarithmically below Lipschitz regularity and the initial datum is only required to be in $\rho_{in}\in H^{(\sfrac75)^+}\cap W^{1,\infty}$. All these constructions have a common feature: the velocity fields are based on alternating shear flows, where one has very good control over mixing properties of the Lagrangian flow-map (inspired by \cite{DeP03, Pie94, Aiz78}), and the dissipation anomaly takes place at a single ``critical'' time $t=T_0$. 

Further, in the recent papers \cite{BruDeL23}, \cite{Bru_etals22}, partially drawing from \cite{JeoYon21, JeoYon22}, it was shown that the alternating shear-flow constructions above can also be seen as velocity fields which are solutions of the \emph{forced} Navier-Stokes equations exhibiting anomalous dissipation of kinetic energy, thereby giving rigorous verification of the `zeroth law of turbulence', at the cost of introducing a forcing term. In a different direction, in \cite{Hofmanova2023} the authors consider the Navier-Stokes equations with large stochastic forcing and friction, and showed for such velocity fields \emph{total} dissipation at the critical time, i.e. $\lim_{t\to T_0}\|\rho_\kappa(t)\|_{L^2}=0$ for any $\kappa>0$, by letting the strength of the forcing and of the friction blow up as $t\to T_0$. Finally, we mention the recent interesting work \cite{huysmans2023nonuniqueness}, where the authors construct a velocity field which, although does not exhibit anomalous dissipation in the sense of \eqref{e:anomalousdissipation}, leads to a (non-unique) solution in the inviscid limit $\rho_{\kappa}\rightharpoonup \rho$ as $\kappa\to 0$ with energy that jumps down and then up again.

\subsubsection{Homogenization, the Armstrong-Vicol approach, and our result} 

The classical concept of eddy diffusivity is a prediction that molecular dissipation is amplified by presence of small-scale fluid eddies (see U. Frisch \cite{Frisch95}, Section 9.6). This concept assumes the existence of two well-separated scales,  the macroscopic integral scale and the microscopic scale of eddies, and has been successfully formalised within the classical homogenization theory, see
\cite{PapPir81,McLaughlin1985,Fri89,Fannjiang1994,MajdaKramer1999}. However, fully developed turbulence has neither merely two scales, nor are the scales well-separated. Even though the homogenization methods has been developed to iterated quantitative homogenization, which includes multiple scales and non-infinite separation (see e.g.~\cite{NiuShenXu2020}), the lack of any apparent separation of scales in turbulence has been perceived as a major obstacle to applying homogenization to modern turbulence theory (see e.g.~the discussion in Section 3 of \cite{MajdaKramer1999}). Despite these reservations, S.~Armstrong and V.~Vicol  \cite{ArmstrongVicol} have recently successfully combined homogenization methods with an infinite iterative scheme to construct a divergence-free vector field with infinitely many active scales, which induces anomalous dissipation. This crucial development has substantially inspired our work. On a heuristic level, the lack of scale separation in fully developed turbulent flow is overcome by the choice of slightly super-exponential growth of active wavenumbers/scales, which does lead to an increasing separation of scales along the iteration - a technical trick which to our knowledge has been introduced into the field in the paper \cite{DSz13} and has been used since ubiquitously.

The methodology of \cite{ArmstrongVicol} has several advantages over the previously mentioned ways of constructing velocities that give rise to anomalous diffusion:
First of all, it is not based on the analysis of the inviscid limit and on mixing properties of the Lagrangian flow map, consequently it becomes compatible with the existence of an inertial range in turbulence. Secondly, in works based on alternating shear flows focussing towards a singular time, the anomalous dissipation only occurs at only one instant of time, a major drawback in the context of turbulent flows in light of the statistical time-invariance, whereas in \cite{ArmstrongVicol} the dissipation anomaly has fractal structure in time. Thirdly, as in classical homogenization, the method of \cite{ArmstrongVicol} yields anomalous dissipation for arbitrary initial data, in contrast with previous constructions based on convex integration \cite{MoSz2018,MoSz2019,BruCoDe2021,MoSa2020}, where the density is constructed in parallel to the velocity field scale-by-scale.  

It is worth mentioning that, although the framework of iterative homogenization comes with several advantages mentioned above, it has one major drawback in comparison with the tools developed based on mixing and alternating shear flows summarized in Section \ref{s:literature} above: the technique seems to be restricted to velocity fields with spatial regularity $C^{(\sfrac{1}{3})^-}$ - although this seems natural in light of Onsager's conjecture and weak solutions of \eqref{e:Euler} and the scaling analysis of effective diffusivities (both in \cite{ArmstrongVicol} and in our work) gives compelling evidence that the bound $\beta<1/3$ in Theorem \ref{t:main} is not merely technical and appears even if $u$ is not required to solve any PDE.

The main step forward that we make compared with \cite{ArmstrongVicol} is that our velocity field solves incompressible Euler equation. To our knowledge this is the first result on anomalous dissipation, where the advecting vector field is itself the solution of an unforced, deterministic PDE. Consequently we provide the first deterministic link on the PDE level between the Obukhov-Corrsin theory for scalar transport and the Onsager theory of ideal turbulence. On the methodological level, we interpret iterative quantitative homogenization in parallel with convex integration, with the basic principle in mind that whilst convex integration is a form of ``inverse renormalization" strategy, iterative quantitative homogenization can be seen as ``forward renormalization". On a more technical level we need to overcome a number of challenges: 
\begin{itemize}
\item First of all, a key issue in applying homogenization techniques is to obtain good control of the corrector. In our setting, which is not based on shear flows but on Mikado flows, obtaining the required control from elliptic estimates seems challenging (indeed, Appendix A in \cite{ArmstrongVicol} points to possible obstruction). Instead, we modify the construction of Mikado flows in such a way that allows us to \emph{explicitly} construct the corrector with a Mikado-like structure.
\item The construction of weak solutions to the Euler equations in \cite{BDSV} relies on the gluing technique introduced by P.~Isett in \cite{Isett2018}, which requires introducing an alternating arrangement of time intervals over which the small-scale fluctuations are turned on and off - and both states have to persist over time-scales which are of the order of the eddy turnover time. In order to have sufficient control of the dissipation rate,  this requires a careful fine-tuning of the parameters and sharper bounds in the convex integration scheme of \cite{BDSV}. 
\item Due to the gluing technique in the convex integration construction, the effective diffusivity coefficient obtained from the homogenization step is oscillating in time, and an additional temporal averaging step is needed to obtain a global effective diffusivity. Our approach to deal with this issue is again inspired by \cite{ArmstrongVicol}, but due to the requirements given by the gluing scheme we need sharper bounds. Indeed, in this paper we have identified and isolated this as an issue of independent interest, and therefore present the temporal averaging in Section \ref{s:time} as a standalone result that can be applied independently. 
\end{itemize}

\subsection{Related topics: enhanced dissipation, mixing, and beyond}
Within or above Lipschitz regularity for $u$, divergence-free drift can of course also assist dissipation, but in a weaker sense than anomalous dissipation. This is usually referred to as enhanced dissipation. The enhanced dissipation by a single shear flow, or another simple ``lower dimensional'' flow is well studied, with roots in the computations by Kelvin and Kolmogorov \cite{Kol34}, compare \cite{BedCZ17,CZElgWid20,ColCZWid21,AlbBeeNov22,CZGal} and references therein. In such cases one needs to restrict the initial datum to a class not orthogonal to the shearing direction, but the setting usually allows to extract precise decay rates. For a more general related approach see \cite{Vuk21}. The general (spectral) condition on arbitrary $u$ for dissipation enhancement in the Lipschitz case has been provided in the seminal \cite{ConKisRyzZla08} (for an example of a related optimal rate, see \cite{ElgLissMatt2023}), whose results has been recently strenghtened and reproved by a different method in \cite{Wei21}. 

Let us point out, without getting into details, natural connections between enhanced dissipation and mixing (which is the phenomenon of inviscid-case migration towards smaller scales caused by a vector field), see \cite{CZDelElg20} and the recent survey \cite{CZsurvey} with its references. Mixing phenomena are present also naturally in the kinetic setting (e.g. Landau damping) cf. \cite{MouVil08}, and in suppression of chemotactic blowup, including the recent beautiful \cite{HuKisYao23}. 

There are also interesting developments in the stochastic setting. In this context we mention the recent work \cite{Otto2023}, inspired by \cite{BriKup91}, \cite{SznZei06}, where optimal quantitative enhanced dissipation estimates are obtained for a random drift given by the Gaussian free field - the strategy there, based on a renormalization strategy of computing effective diffusivities scale-by-scale, is very much reminiscent of our point of view.

\subsection{Notation}
We use mostly standard notation, the less standard is explained where it appears. To follow several vector calculus identitites, one should keep in mind that the vectors are column vectors. We will often write $a \lesssim b$, which means there is an (implicit) constant $C$, such that $a\leq C b$. Importantly, such implicit constants are uniform over iterative steps we will perform.

\subsection{Paper outline}
Section \ref{s:blocks} presents fundamental building blocks for our construction. The technical results needed in Section \ref{s:blocks}: the homogenization Proposition \ref{p:hom_trho} (and its corollaries), the time averaging Proposition \ref{p:t_avg} and the convex integration Proposition \ref{p:Onsager} are proved, respectively, in Section \ref{s:homogenization}, Section \ref{s:time}, Section \ref{s:Onsager}. Section \ref{s:main_proof} contains the proofs of our main results. Section \ref{s:energy}  gathers mostly standard energy estimates needed in the paper. We have tried to present the results of Sections \ref{s:energy}  - \ref{s:Onsager} in a self-contained manner, as we believe they may be of independent interest.

\section{Building blocks for proof of Theorem \ref{t:main}}\label{s:blocks}
In this section we provide the key ingredients for our proofs, and postpone technical details to subsequent sections. As mentioned in the introduction, our strategy involves a "backward renormalization" in constructing the velocity field in Section \ref{s:vectorfield} (in the sense of an iteration from large-scale to small-scale, $u_q\mapsto u_{q+1}\mapsto\dots$, based on  convex-integration ideas), and a "forward renormalization" in constructing the passive tracer in Section \ref{s:enhanceddis} (in the sense of an iteration from small-scale to large-scale, $\rho_{q+1}\mapsto\rho_q\mapsto\dots$, based on homogenization-related arguments). Both of these sections have an inductive flavour. In order to initialise the induction, we will need also an h-principle argument, presented in Section \ref{s:hprinciple}.

\subsection{Construction of the vector field - Convex integration}
\label{s:vectorfield}

The general scheme for producing H\"older continuous weak solutions of the Euler equations \eqref{e:Euler} is by now well understood. Our approach below closely follows the presentation in \cite{BDSV}. One proceeds via an inductive process on a sequence of approximate solutions $u_q$ with associated Reynolds defect $\mathring{R}_q$ and pressure $p_q$, for $q=0,1,2,\dots$, which satisfy the \emph{Euler-Reynolds system}
\begin{equation}\label{e:EulerReynolds}
\begin{split}
\partial_tu_q+\div(u_q\otimes u_q)+\nabla p_q&=\div \mathring{R}_q\,,\\	
\div u_q&=0\,,
\end{split}
\end{equation}
with constraints 
\begin{equation}\label{e:ERconstraints}
\tr \mathring{R}_q(x,t)=0,\quad \int_{\T^3}u_q(x,t)\,dx=0,\quad \int_{\T^3}p_q(x,t)\,dx=0.	
\end{equation}
We note in passing that these normalizations determine the pressure $p_q$ uniquely from $(u_q,\mathring{R}_q)$ so that one may speak of the pair $(u_q,\mathring{R}_q)$ being a solution of \eqref{e:EulerReynolds}.

\subsubsection{Inductive assumptions}

The induction process involves a set of \emph{inductive estimates}. These estimates are in terms of a frequency parameter $\lambda_q$ and amplitude $\delta_q$, which are given by
\begin{equation}\label{e:lambdadelta}
\lambda_q:= 2\pi \ceil{a^{(b^q)}},\quad 
\delta_q:=\lambda_q^{-2\beta}
\end{equation}
 where $\ceil{x}$ denotes the smallest integer $n\geq x$, $a\gg 1$ is a  large parameter, $b>1$ is close to $1$ and $0<\beta<\sfrac13$ is the exponent of Theorem \ref{t:main}. The parameters $a$ and $b$ are then related to $\beta$. 
With these parameters the inductive estimates take the form\footnote{In \cite{BDSV} an additional estimate is added for $\|u_q\|_{C^0}$, but it turns out this can be avoided. The only place where a bound on $\|u_q\|_{C^0}$ was needed in \cite{BDSV} is in \cite[Proposition 5.9]{BDSV}, but as we show below, the (much worse) bound induced by \eqref{e:u_q_inductive_est} for $n=1$ suffices to control the relevant terms - see Proposition \ref{p:est_mollification} and Lemma \ref{l:boundsonbc1}.}
\begin{align}
\norm{\mathring R_q}_{C^0}&\leq  \delta_{q+1}\lambda_q^{-\gamma_R},\label{e:R_q_inductive_est}\\
\norm{u_q}_{C^n}&\leq M \delta_q^{\sfrac12}\lambda_q^n\quad\textrm{ for }n=1,2,\dots,\bar{N},\label{e:u_q_inductive_est}\\
\left|e(t)-\int_{\T^3}\abs{u_q}^2\,dx-\bar{e}\delta_{q+1}\right|&\leq \delta_{q+1}\lambda_q^{-\gamma_E},\label{e:energy_inductive_assumption}
\end{align}
where $\gamma_R,\gamma_E>0$ are small parameters, $\bar{N}\in\N$ is a large parameter to be chosen suitably (depending on $\beta>0$ and $b>1$), and $\bar{e}>0$, $M\geq 1$ are universal constants, which will be fixed throughout the iteration and depend on the particular geometric form of the perturbing building blocks - they will be specified in Definition \ref{d:defebar} and \ref{d:defM}. We remark that in \cite{BDSV} only the case $\bar{N}=1$ is required in \eqref{e:u_q_inductive_est} and \eqref{e:energy_inductive_assumption} is slightly weaker. Moreover, in \cite{BDSV} a generic small parameter $\alpha>0$ is used in place of $\gamma_R,\gamma_E$, but for our purposes we need to choose these small corrections more carefully.

For notational convenience we introduce 
\begin{equation*}
\mathring{\delta}_{q+1}:=\delta_{q+1}\lambda_q^{-\gamma_R}=\lambda_q^{-2b\beta-\gamma_R},	
\end{equation*}
so that \eqref{e:R_q_inductive_est} can be written as $\norm{\mathring{R}_q}_{C^0}\leq \mathring{\delta}_{q+1}$.

\subsubsection{Parameter choices}

The inductive construction for passing from $u_q$ to $u_{q+1}$ involves three steps: \emph{mollification}, \emph{gluing} and \emph{perturbation}. In these steps two more scales are introduced, an adjusted length-scale $\ell_q$ and an adjusted time-scale $\tau_q$. In our case these will be defined as
\begin{equation}\label{e:elltau}
\ell_q:=\lambda_q^{-1-\gamma_L},\quad \tau_q:=\lambda_q^{-1+\beta-\gamma_T},	
\end{equation}
where $\gamma_L,\gamma_T$ are additional small parameters to be chosen suitably in dependence of $\beta,b$. It is worth pointing out that, if one were to set $\gamma_L=\gamma_T=0$, then $\ell_q,\tau_q$ would be the natural (dimensionally consistent) length- and time-scales induced by the velocity field $u_q$ (cf.~\eqref{e:lambdadelta} and \eqref{e:u_q_inductive_est}). Indeed, we can think of $\delta_q^{\sfrac12}$ having physical dimension of velocity, i.e. $LT^{-1}$, and $\lambda_q^{-1}$ having physical dimension of length $L$.

Moreover, setting $\gamma_R=0$ would be the consistent estimate in light of the basic principle in convex integration, that the error $R_q$ is cancelled by the new average stress $\langle (u_{q+1}-u_q)\otimes (u_{q+1}-u_q)\rangle$. 

In addition to these small parameters we will also use $\alpha>0$, as is also done in \cite{BDSV}, to take care of the lack of Schauder estimates in $C^0, C^1,\dots$ spaces. Thus, in summary, we will use the following set of additional parameters, which will all be chosen depending on $\beta,b$:
\begin{itemize}
\item $\gamma_R\in(0,1)$: smallness of $\|\mathring{R}_q\|_{C^0}$ with respect to $\|R_q\|_{C^0}$,
\item $\gamma_L\in(0,1)$: smallness of $\ell_q$ with respect to $\lambda_q^{-1}$,
\item $\gamma_T\in(0,1)$: smallness of $\tau_q$ with respect to $(\delta_q^{\sfrac12}\lambda_q)^{-1}$,
\item $\gamma_E\in(0,1)$: smallness of energy gap with respect to $\|R_q\|_{C^0}$,
\item $\alpha\in(0,1)$: Schauder exponent,
\item $\bar{N}\in\N$: number of derivatives in the induction.
\end{itemize}
With a suitable choice of these parameters we have the following analogue of \cite[Proposition 2.1]{BDSV}:
\begin{proposition}\label{p:Onsager}
There exist universal constants $M\geq 1$, $\bar{e}>0$ with the following property. Assume $0<\beta<\frac13$ and
\begin{equation}\label{e:b_beta_rel}
1<b<\frac{1-\beta}{2\beta}\,.
\end{equation}
Further, assume that $\gamma_T,\gamma_R,\gamma_E>0$ satisfy
\begin{equation}\label{e:Onsager_Conditions}
	\max\{\gamma_T+b\gamma_R,\gamma_E\}<(b-1)\bigl(1-(2b+1)\beta\bigr).
\end{equation}
{For any sufficiently small $\gamma_L>0$ there exists $\bar{N}\in\N$, depending on $\beta, b, \gamma_T,\gamma_R,\gamma_E$, such that for any sufficiently small $\alpha>0$ and any strictly positive smooth function $e:[0,T]\to\R$ there exists $a_0\gg 1$ such that the following holds:}

  Let $(u_q,\mathring R_q)$ be a smooth solution of \eqref{e:EulerReynolds} satisfying the estimates \eqref{e:R_q_inductive_est}--\eqref{e:energy_inductive_assumption} with $\lambda_q,\delta_q$ given by \eqref{e:lambdadelta} for any fixed $a\geq a_0$. Then there exists another solution  $(u_{q+1}, \mathring R_{q+1})$ to \eqref{e:EulerReynolds} satisfying \eqref{e:R_q_inductive_est}--\eqref{e:energy_inductive_assumption} with $q$ replaced by $q+1$, and we have 
\begin{equation}
 \norm{u_{q+1}-u_q}_{C^0}+\frac{1}{\lambda_{q+1}}\norm{u_{q+1}-u_q}_{C^1} 
\leq M\delta_{q+1}^{\sfrac12},\label{e:v_diff_prop_est}
\end{equation}
\end{proposition}
The new velocity field $u_{q+1}$ is obtained as
\begin{equation}\label{e:iteration}
u_{q+1}:=\bar{u}_q+w_{q+1},	
\end{equation}
where $\bar{u}_q$ is constructed from $u_q$ via Isett's gluing procedure \cite{Isett2018}, and $w_{q+1}$ is the new perturbation consisting of a deformed family of Mikado flows. In Sections \ref{s:gluingsketch}-\ref{s:Mikado} below we now sketch the proof of Proposition \ref{p:Onsager}. The detailed proof, which is very much based on \cite{BDSV}, is given in Section \ref{s:Onsager}.

\subsubsection{Gluing procedure}\label{s:gluingsketch}

The gluing procedure amounts to the following statement:
\begin{proposition}\label{p:gluing}
Within the setting of Proposition \ref{p:Onsager} we have the following statement. Let $(u_q,\mathring R_q)$ be a smooth solution of \eqref{e:EulerReynolds} satisfying the estimates \eqref{e:R_q_inductive_est}--\eqref{e:energy_inductive_assumption}. Then there exists another solution $(\bar{u}_q,\mathring{\bar{R}}_q)$ to \eqref{e:EulerReynolds} such that
\begin{equation}
\supp\mathring{\bar{R}}_q\subset \T^3\times \bigcup_{i\in\N}(i\tau_q+\tfrac{1}{3}\tau_q,i\tau_q+\tfrac{2}{3}\tau_q)
\end{equation}
and the following estimates hold for any $N\geq 0$:
\begin{align}
\|\bar{u}_q\|_{C^{N+1}}&\lesssim \delta_q^{\sfrac12}\lambda_q\ell_q^{-N}\,,\label{e:gluingu}\\
\|\mathring{\bar{R}}_{q}\|_{C^{N+\alpha}}&\lesssim \mathring{\delta}_{q+1}\ell_q^{-N-2\alpha}\,,\label{e:gluingR}\\
\|(\partial_t+\bar{u}_q\cdot\nabla)\mathring{\bar{R}}_{q}\|_{C^{N+\alpha}}&\lesssim \tau_q^{-1}\mathring{\delta}_{q+1}\ell_q^{-N-2\alpha}\,,\label{e:gluingDR}\\
\left|\int_{\T^3}|\bar{u}_q|^2-|u_q|^2\,dx\right|&\lesssim\mathring{\delta}_{q+1}\,.\label{e:gluingE}
\end{align}
Moreover, the vector potentials of $u_q$ and $\bar{u}_q$ satisfy
\begin{equation}\label{e:gluingz}
\|z_q-\bar{z}_{q}\|_{C^\alpha}\lesssim \tau_q\mathring{\delta}_{q+1}\ell_q^{-\alpha}\,.
\end{equation}
\end{proposition}

Here and in the sequel, the vector potential $z$ of a divergence-free velocity field $u$ is given by the Biot-Savart operator on $\T^3$, defined as $\mathcal{B}=(-\Delta)^{-1}\curl$, so that $z=\mathcal{B}u$ is the unique solution of 
\begin{equation}\label{e:Biot_Savart}
\div z=0\qquad\textrm{ and }\qquad\curl z=u  \,,
\end{equation}
 recalling that we assume zero spatial average of $u$.

Proposition \ref{p:gluing} is restated below as Corollary \ref{c:gluing} and is a direct consequence of Corollary \ref{c:mollification} and Proposition \ref{p:p_gluing}.

\subsubsection{The perturbation}\label{s:Mikado}
The formula for the new perturbation $w_{q+1}$ uses Mikado flows, which we recall here briefly.
Mikado flows were introduced originally in \cite{DaSz2017} and widely used since in applications of convex integration to fluid dynamics. 

We fix a finite set $\Lambda\subset\R^3$ consisting of nonzero vectors $\vec{k}\in\R^3$ with rational coordinates, and for each $\vec{k}\in\Lambda$ let $\varphi_{\vec{k}}$ be a periodic function with the properties
\begin{itemize}
    \item  $\vec{k}\cdot\nabla\varphi_{\vec{k}}=0$,
    \item For any $\vec{k}\neq \vec{k}'\in\Lambda$ we have $\supp\varphi_{\vec{k}}\cap \supp\varphi_{\vec{k}'}=\emptyset$,
    \item $\dashint_{\T^3}\varphi_{\vec{k}}(\xi)\,d\xi=0$ and $\dashint_{\T^3}|\Delta\varphi_{\vec{k}}(\xi)|^2\,d\xi=\dashint_{\T^3}|\nabla\varphi_{\vec{k}}(\xi)|^2\,d\xi=1$.
\end{itemize}
Next, let $\psi_{\vec{k}}=\Delta\varphi_{\vec{k}}$ and define
\begin{equation}\label{e:defUW}
W_{\vec{k}}(\xi)=\psi_{\vec{k}}(\xi)\vec{k},\quad  U_{\vec{k}}(\xi)=\vec{k}\times\nabla\varphi_{\vec{k}} (\xi),
\end{equation}
so that\footnote{Here we use the vector calculus identity $\curl(F\times G)=F\div G-G\div F+(G\cdot\nabla)F-(F\cdot\nabla)G$.} $\curl U_{\vec{k}}=W_{\vec{k}}$  and, for any $a_{\vec{k}}$, $\vec{k}\in \Lambda$, the vector field $W=\sum_{\vec{k}\in\Lambda}a_{\vec{k}}W_{\vec{k}}$ satisfies 
$$
\dashint W\,d\xi=0,\quad \dashint W\otimes W\,d\xi=\sum_{\vec{k}\in\Lambda}a_{\vec{k}}^2\vec{k}\otimes\vec{k}. 
$$
Further, we define $H_{\vec{k}}=H_{\vec{k}}(\xi)$ to be the antisymmetric zero-mean matrix with the property that for any $v \in \R^3$ we have $H_{\vec{k}} v = - U_{\vec{k}} \times v$. 

The following lemma (originating in the work of Nash \cite{Nash54}) is crucial:
\begin{lemma}\label{l:Mikado}
For any compact set $\mathcal{N}\subset \mathcal{S}^{3\times 3}_{+}$ there exists a finite $\Lambda\subset\mathbb{Q}^3$ such that there exists smooth functions $a_{\vec{k}}:\mathcal{N}\to\R_+$ with 
\begin{equation}\label{e:MikadoProperty}
\sum_{\vec{k}\in\Lambda}a_{\vec{k}}^2(R)\vec{k}\otimes\vec{k}=R\quad\textrm{ for any }R\in\mathcal{N}.
\end{equation}
\end{lemma}
The corresponding vector field $W=W(R,\xi)=\sum_{\vec{k}\in\Lambda}a_{\vec{k}}(R)W_{\vec{k}}(\xi)$ is called a Mikado flow. In the following we will fix $\mathcal{N}:=B_{1/2}(\Id)$, the metric ball of radius $1/2$ around the identity matrix in $\mathcal{S}^{3\times 3}_{+}$ and denote the corresponding Mikado vector field by $W=W(R,\xi)$. 

Following \cite{BDSV} the new perturbation is then defined as
\begin{equation}\label{e:defwq+1}
	w_{q+1}=\frac{1}{\lambda_{q+1}}\curl\left[\sum_{i\in\N}\sum_{\vec{k}\in\Lambda}\eta_i\sigma_{q}^{\sfrac12}a_{\vec{k}}(\tilde R_{q,i})\nabla\Phi_i^TU_{\vec{k}}(\lambda_{q+1}\Phi_i)\right]\,,
\end{equation}
where 
\begin{itemize}
\item $\eta_i=\eta_i(x,t)$, $i\in\N$, are smooth nonnegative cutoff functions with pairwise disjoint supports such that
\begin{equation}\label{e:propertyofetai}
\|\partial_t^m\nabla^n_x\eta_i\|_{L^\infty}\leq C_{n,m}\tau_q^{-m},\quad \sum_i\eta_i^2(x,t)=\bar{\eta}^2(x,\tau_q^{-1}t)\,,	
\end{equation}
where $\bar{\eta}=\bar{\eta}(x,t)$ is a (universal) function, $1$-periodic in $t$, such that
\begin{equation*}
	\dashint_{\T^3}\bar{\eta}^2(x,t)\,dx=c_0,\quad \dashint_0^1\bar{\eta}^2(x,t)\,dt=c_1
\end{equation*}
for some universal constants $c_0,c_1>0$.
\item $\sigma_q=\sigma_q(t)$ is a positive scalar function with the property
\begin{equation}\label{e:propertyofsigmaq}
	|\sigma_q(t)-c_1^{-1}\delta_{q+1}|
 \leq C \delta_{q+1}(\lambda_q^{-\gamma_E}+\lambda_q^{-\gamma_R}+\lambda_q^{-(b-1)\beta}), \quad |\pat \sigma_q(t) | \le C \delta_{q+1} \tau_q^{-1}. 
\end{equation}
\item The maps $\Phi_i=\Phi_i(x,t)$ are the volume-preserving diffeomorphisms defined as the inverse flow map of the velocity field $\bar{u}_q$, which satisfy for $(x,t)\in\supp\eta_i$:
\begin{equation}\label{e:propertyofPhii}
\begin{split}
&\|\nabla\Phi_i-\Id\|_{L^\infty}\leq C\lambda_q^{-\gamma_T}, \\
&\|\nabla\Phi_i(x,t)\|_{C^n}+\|(\nabla\Phi_i)^{-1}(x,t)\|_{C^n}\leq C_n\ell_q^{-n},\quad  \|\nabla^n D_t \nabla\Phi_i\|_{L^\infty} \leq C_n \delta_q^{1/2}\lambda_q\ell_q^{-n},	
    \end{split}
\end{equation}
where $D_t=\partial_t+\bar{u}_q\cdot\nabla$.
\item $\tilde R_{q,i}=\tilde R_{q,i}(x,t)$, $i\in\N$ is given by
\begin{equation}\label{e:propertyoftildeRqi}
	\tilde R_{q,i}=\nabla\Phi_i(\Id-\sigma_q^{-1}\mathring{\bar{R}}_q)\nabla\Phi_i^T
\end{equation}
and satisfies
\begin{equation}\label{e:propertyoftildeRqi2}
 \|\tilde R_{q,i}\|_{C^n}\leq C_n\ell_q^{-n}, \quad \|\nabla^n D_t \tilde R_{q,i}\|_{L^\infty}\leq C_n \tau_q^{-1} \ell_q^{-n}.
\end{equation}
\end{itemize}

These properties will be obtained in Section \ref{s:perturbation}, where we will conclude the proof of Proposition \ref{p:Onsager}.

\subsection{Anomalous dissipation via iterative stages}\label{s:enhanceddis}
In this section we present the building block for our construction of the advected scalar.
We define an additional $q$-dependent parameter which plays a key role in this paper:
\begin{equation}\label{e:defkappaq}
\kappa_q:=\lambda_q^{-\theta},\quad \theta=\frac{2b}{b+1}(1+\beta)\,.
\end{equation}
Using \eqref{e:lambdadelta} it is easy to verify the recursive identity
\begin{equation}\label{e:kappaqkappaq+1}
	\kappa_q=\frac{\delta_{q+1}}{\lambda_{q+1}^2\kappa_{q+1}}.
\end{equation}
In turn, this identity indicates that we can think of $\kappa_q$ as having physical dimension of diffusion coefficient $L^2T^{-1}$.

We consider the sequence of advection-diffusion equations on $\T^3\times[0,T]$:
\begin{equation}\label{e:equationq}
\begin{split}
\partial_t\rho_{q}+u_{q}\cdot\nabla\rho_{q}&=\kappa_{q}\Delta\rho_{q}\,,\\	
\rho_{q}|_{t=0}&=\rho_{in}\,,
\end{split}
\end{equation}
where $u_q$ is the sequence of velocity fields obtained in Section \ref{s:vectorfield} via Proposition \ref{p:Onsager}. 
We are interested in comparing the cumulative dissipation for subsequent values of $q$, given by
\begin{equation}\label{e:defDq}
	D_q:=\kappa_q\int_0^T\|\nabla\rho_q\|_{L^2}^2\,dt=\frac{1}{2}(\|\rho_{in}\|_{L^2}^2-\|\rho_q(T)\|_{L^2}^2).
\end{equation}

The  main result of this section is

\begin{proposition}\label{p:main}
Within the setting of Proposition \ref{p:Onsager} let us assume in addition that $\beta,b$ and $\gamma_L, \gamma_T,\gamma_R$ satisfy

\begin{equation}\label{e:Dissipation}
\gamma_T<	\frac{b-1}{b+1}(1-(2b+1)\beta)<\gamma_R+\gamma_T,
\end{equation}
\begin{equation}\label{e:Dissipation2}
2\gamma_L<	\frac{b-1}{b+1} (1+\beta).
\end{equation}
Then there exists $\tilde{N}\in\N$ and $\gamma>0$ such that for any sufficiently small $\alpha>0$ there exists $a_0\gg 1$ with the following property:

For any initial datum $\rho_{in}\in L^2(\T^3)$ with $\int_{\T^3}\rho_{in}\,dx=0$ such that 
\begin{equation}\label{e:boundoninitialdatum}
    \|\rho_{in}\|_{H^n} \leq \lambda_q^n \min \{ D_q^{\sfrac12}, D_{q+1}^{\sfrac12} \}
    \qquad\textrm{ for }1\leq n\leq \tilde N,
\end{equation}
 we have
 \begin{equation}\label{e:mainstability}
 	|D_{q+1}-D_q|\leq \tfrac{1}{2}\lambda_q^{-\gamma}D_q
 \end{equation}
 and
 \begin{equation}\label{e:mainstability2}
  \sup_{t \le T} \|\rho_{q+1}(t) - \rho_{q}(t)\|_{L^2}^2 \le \lambda_q^{-2\gamma}D_{q}.
\end{equation}
\end{proposition}

\begin{proof}
We start by fixing $\gamma,\alpha>0$. The exponent $\gamma>0$ is determined by the inequalities
\begin{equation}   \label{e:global_gamma}
\begin{split}
    \gamma = \tfrac{1}{4} \min \Bigl\{\gamma_L,\gamma_T,\gamma_R,\gamma_E,(b-1)\beta,(b-1)\theta,\gamma_T+\gamma_R+(2b-1)\beta+1-\theta, \\
    \frac{b-1}{b+1}(1-(2b+1)\beta) - \gamma_T \Bigr\}.
\end{split}
\end{equation}
We remark that the last condition in the first line above can be satisfied because of the right inequality in \eqref{e:Dissipation} and the definition \eqref{e:defkappaq} of $\theta$. In turn, $\alpha$ is then chosen sufficiently small so that 
\begin{equation}\label{e:conditiononalpha1}
	\alpha(1+\gamma_L)+2\gamma+\theta<\gamma_T+\gamma_R+(2b-1)\beta+1.
\end{equation}
We also fix 
\begin{equation}   \label{e:global_N}
\begin{split}
    \tilde N \geq \max \Bigl\{2 \left( \frac{b-1}{b+1} \big( 1 - (2b+1)\beta \big) - \gamma_T \right)^{-1}, \,
    1 + \frac{2b(b-1)(1+\beta)}{\gamma_T(b+1)} \Bigr\},
\end{split}
\end{equation}
where the first lower bound is given by $N_h$ in \eqref{e:hom_para_assump:Nh} and the second lower bound is given by $N$ to validate (A1) in Section \ref{s:time}. Now we will show the stability estimate 
\eqref{e:mainstability} in several steps, via consecutive stability estimates.

\noindent\emph{Step 1: Preparation ($\rho_{q+1} \leadsto \tilde \rho_{q+1}$)}

We start with the equation
\begin{equation}\label{e:equationq+1}
\begin{split}
\partial_t\rho_{q+1}+u_{q+1}\cdot\nabla\rho_{q+1}&=\kappa_{q+1}\Delta\rho_{q+1}\,,\\	
\rho_{q+1}|_{t=0}&=\rho_{in},
\end{split}
\end{equation}
and recall the form $u_{q+1}=\bar{u}_q+w_{q+1}$ of the advecting field $u_{q+1}$, where $w_{q+1}$ is given in \eqref{e:defwq+1}. Then we make use of the linear algebra identity $(Au)\times(Av)=\textrm{cof }A(u\times v)=(\det A) A^{-T}(u\times v)$ for any $3\times 3$ matrix $A$ and vectors $u,v\in\R^3$. From this we deduce that if $H$ is the antisymmetric matrix such that $Hv=-u\times v$ and $\det A=1$, then $\tilde H=A^{-1}HA^{-T}$ is the antisymmetric matrix such that $\tilde Hv=-(A^Tu)\times v$. In particular, using the antisymmetric matrix-valued functions $H_{\vec{k}}(\xi)$ introduced above in Section \ref{s:Mikado}, we deduce
\begin{equation*}
(\nabla\Phi^{-1}_i H_{\vec{k}}\nabla\Phi^{-T}_i)v=-(\nabla\Phi_i^TU_{\vec{k}})\times v.
\end{equation*}
Therefore, using the identity $\div(z\times\nabla\rho)=(\curl z)\cdot\nabla\rho$, we can write equation \eqref{e:equationq+1} equivalently as
\begin{equation}\label{e:equationAq+1}
\begin{split}
\partial_t\rho_{q+1}+\bar{u}_{q}\cdot\nabla\rho_{q+1}&=\div A_{q+1}\nabla\rho_{q+1}\,,\\	
\rho_{q+1}|_{t=0}&=\rho_{in}\,,
\end{split}	
\end{equation}
where the elliptic matrix $A_{q+1}=A_{q+1}(x,t)$ is defined as
\begin{equation}\label{e:defAq+1}
	A_{q+1} (x,t)=\kappa_{q+1}\Id + \sum_i \nabla\Phi_i^{-1} (x,t) \frac{\eta_i (x,t)}{\lambda_{q+1}}H^{(i)}(x,t,\lambda_{q+1}\Phi_i)\nabla\Phi_i^{-T} (x,t) \,,
\end{equation}
and
\begin{align*}
H^{(i)}(x,t,\xi)&=\sum_{\vec{k}\in\Lambda}\sigma_q^{\sfrac12}(t)a_{\vec{k}}(\tilde R_{q,i}(x,t))H_{\vec{k}}(\lambda_{q+1}\xi)\,.	
\end{align*}
Let $(\tilde\eta_i)_i$ be a partition of unity such that $\tilde\eta_i\eta_i=\eta_i$,  $\tilde\eta_i\eta_j=0$ for $j\neq i$ and satisfying estimates of the same type as \eqref{e:propertyofetai}: 
\begin{equation} \label{e:propertyoftildeetai}
    \|\partial_t^m\nabla^n_x\tilde\eta_i\|_{L^\infty}\leq \tilde C_{n,m}\tau_q^{-m}.
\end{equation}
Define the elliptic matrix
\begin{equation}\label{e:defAq+1tilde}
	\tilde A_{q+1}(x,t)=\sum_i\tilde\eta_i(x,t)\nabla\Phi_i^{-1}(x,t)\left[\kappa_{q+1}\Id+\frac{\eta_i(x,t)}{\lambda_{q+1}}H^{(i)}(x,t,\lambda_{q+1}\Phi_i(x,t))\right]\nabla\Phi_i^{-T}(x,t)\,.
\end{equation}
The estimate \eqref{e:propertyofPhii} implies, since $\tilde\eta_i$ is a partition of unity
\begin{equation}\label{e:AtildeA}
\|A_{q+1}-\tilde A_{q+1}\|_{L^\infty}\leq \kappa_{q+1}\sum_{i}\tilde\eta_i\|\nabla\Phi_i^{-1}\nabla\Phi_i^{-T}-\Id\|_{L^\infty}\leq C\kappa_{q+1}\lambda_{q}^{-\gamma_T}.	
\end{equation}
Therefore we can compare the cumulative dissipation in \eqref{e:equationAq+1} with that in
\begin{equation}\label{e:equationAq+1tilde}
\begin{split}
\partial_t\tilde\rho_{q+1}+\bar{u}_{q}\cdot\nabla\tilde\rho_{q+1}&=\div \tilde A_{q+1}\nabla\tilde \rho_{q+1}\,,\\	
\tilde \rho_{q+1}|_{t=0}&=\rho_{in}\,.
\end{split}	
\end{equation}
More precisely, assuming that $a_0$ is sufficiently large we can ensure $C\lambda_q^{-\gamma_T}\le \tfrac{1}{10}\lambda_q^{-\sfrac{\gamma_T}{2}}$, and then apply
Proposition \ref{p:stability_in_ellipticity} with $\varepsilon=\tfrac{1}{10}\lambda_q^{-\sfrac{\gamma_T}{4}}\le\tfrac12$ to conclude
\begin{align*}
 \sup_{t \le T} \|\rho_{q+1}(t) - \tilde\rho_{q+1}(t)\|_{L^2}^2 \le& \lambda_q^{-\gamma_T}\left|\|\rho_{in}\|^2_{L^2}-\|\tilde\rho_{q+1}(T)\|_{L^2}^2\right|, \\
   \left|\|\rho_{q+1}(T)\|^2_{L^2}-\|\tilde\rho_{q+1}(T)\|_{L^2}^2\right|\le& \lambda_q^{-\sfrac{\gamma_T}{2}}\left|\|\rho_{in}\|^2_{L^2}-\|\tilde\rho_{q+1}(T)\|_{L^2}^2\right|.
\end{align*}
We may also write this as
\begin{equation}\label{e:step1}
\left|D_{q+1}-\tilde D_{q+1}\right|\le \lambda_q^{-\sfrac{\gamma_T}{2}}\tilde D_{q+1} \le \lambda_q^{-{2\gamma}}\tilde D_{q+1},
\end{equation}
where $\tilde D_{q+1}=\tfrac12\left|\|\rho_{in}\|^2_{L^2}-\|\tilde\rho_{q+1}(T)\|_{L^2}^2\right|$.

\smallskip

\noindent\emph{Step 2: Spatial homogenization ($ \tilde \rho_{q+1}  \leadsto \bar\rho_{q}^{(1)}$)}

In the second step, we use estimates of quantitative homogenization and an explicit formula for the corrector, to replace \eqref{e:equationAq+1tilde} by the homogenized equation
\begin{equation}\label{e:equationAqbar}
\begin{split}
\partial_t\bar{\rho}_{q}^{(1)}+\bar{u}_{q}\cdot\nabla \bar{\rho}^{(1)}_{q}&=\div \bar{A}_{q}\nabla\bar{\rho}^{(1)}_{q}\,,\\	
\bar{\rho}_{q}^{(1)}|_{t=0}&=\rho_{in}\,.
\end{split}	
\end{equation}
The homogenized elliptic coefficient is given, like in classical homogenization theory, by
\begin{equation*} 
	\bar A_{q}(x,t) = \dashint \tilde A_{q+1}(x,t,\xi) \Big( \Id + \sum_i \eta_i(x,t) \nabla\Phi_i^T(x,t) \nabla_\xi\chi_i^T(x,t,\xi) \Big) \,d\xi\,
\end{equation*}
with corrector $\chi_i: \T^3 \times [0,T] \times \T^3 \rightarrow \R^3$ chosen in such a way, that in the homogenization Ansatz in the presence of a macroscopic drift (cf.~\cite{McLaughlin1985,HouYangWang2004})
\begin{equation*}
    \rho(x,t) = \bar{\rho}(x,t) + \frac{1}{\lambda} \sum_i \eta_i(x,t) \chi_i \big( x,t,\lambda \Phi_i(x,t) \big)\cdot\nabla\bar\rho(x,t)+\tilde\rho(x,t)
\end{equation*}
$\tilde\rho$ will be lower order. We will show below in Section \ref{s:homogenization} that, because of the special structure of the oscillating Mikado vector field $w_{q+1}$, $\chi$ can be defined \emph{explicitly} by
\begin{equation}\label{e:explicitchii}
	\chi_i(x,t,\xi) = \frac{\sigma_q^{1/2}(t)}{\kappa_{q+1} \lambda_{q+1}} \nabla \Phi_i^{-1}(x,t) \sum_{\vec{k}} a_{\vec{k}} \big( \tilde R_{q,i}(x,t) \big) \varphi_{\vec{k}}(\xi) \vec{k}\,. 
\end{equation}
Using the properties of Mikado flows in Section \ref{s:Mikado}, this formula allows us to obtain an explicit expression for $\bar{A}_{q+1}$. Indeed, since
$H^{(i)}$ 
satisfies $\dashint H^{(i)}\,d\xi=0$, we have
\begin{equation*}
    \bar{A}_q=\kappa_{q+1}\sum_i\tilde\eta_i\nabla\Phi_i^{-1}\nabla\Phi_i^{-T}+\frac{1}{\lambda_{q+1}}\sum_i\eta_i^2\nabla\Phi_i^{-1}\dashint H^{(i)}(\nabla_\xi\chi_i)^T\,d\xi.
\end{equation*}
On the other hand, using \eqref{e:explicitchii} and the definition of $H^{(i)}(x,t,\xi)$, 
\begin{align*}
	H^{(i)}(\nabla_\xi\chi_i)^T&=\frac{\sigma_q}{\lambda_{q+1}\kappa_{q+1}}\sum_{\vec{k} \in \Lambda}a_{\vec{k}}^2(\tilde R_{q,i})H_{\vec{k}}(\nabla\varphi_{\vec{k}}\otimes (\nabla\Phi_i^{-1}\vec{k}))\\
	&=\frac{\sigma_q}{\lambda_{q+1}\kappa_{q+1}}\sum_{\vec{k}  \in \Lambda}a_{\vec{k}}^2(\tilde R_{q,i})|\nabla\varphi_{\vec{k}}|^2(\vec{k}\otimes\vec{k})\nabla\Phi_i^{-T},
\end{align*}
where we used the definition of $H_{\vec{k}}(\xi)$ to deduce 
\begin{equation*}
	H_{\vec{k}}\nabla\varphi_{\vec{k}}=-(\vec{k}\times\nabla\varphi_{\vec{k}})\times\nabla\varphi_{\vec{k}}=|\nabla\varphi_{\vec{k}}|^2 \vec{k}.
\end{equation*}
Using the normalization in the definition of Mikado flows in Section \ref{s:Mikado} as well as Lemma \ref{l:Mikado}, we conclude
\begin{align*}
	\dashint H^{(i)}(\nabla_\xi\chi_i)^T\,d\xi=\frac{\sigma_q}{\lambda_{q+1}\kappa_{q+1}}\tilde R_{q,i}\nabla\Phi_i^{-T}\,.
\end{align*}
Hence
\begin{equation}\label{e:defbarA}
\begin{split}
\bar{A}_q(x,t) &=\kappa_{q+1}\sum_i\tilde\eta_i\nabla\Phi^{-1}_i\nabla\Phi_i^{-T}+\frac{\sigma_q}{\lambda_{q+1}^2\kappa_{q+1}}\sum_i \eta_i^2\nabla\Phi_i^{-1} \tilde R_{q,i}\nabla\Phi_i^{-T} \\
&=\kappa_{q+1}\sum_i\tilde\eta_i\nabla\Phi^{-1}_i\nabla\Phi_i^{-T}+\sigma_q \delta^{-1}_{q+1}\kappa_{q} \bar{\eta}^2(x,\tau_q^{-1}t) (\Id-\sigma_q^{-1}\mathring{\bar R}_q)\\
&=\kappa_{q+1}\sum_i\tilde\eta_i\nabla\Phi^{-1}_i\nabla\Phi_i^{-T}+\kappa_q \delta^{-1}_{q+1}\sigma_{q} \bar{\eta}^2(x,\tau_q^{-1}t)\Id-\kappa_{q} \delta^{-1}_{q+1} \bar{\eta}^2(x,\tau_q^{-1}t)\mathring{\bar R}_q,
\end{split}
\end{equation}
where for the second line we have used \eqref{e:propertyoftildeRqi} and \eqref{e:kappaqkappaq+1}, \eqref{e:propertyofetai}.
Let
\begin{equation}\label{e:defkappabarq}
\tilde{\kappa}_q(x,t):= \kappa_{q+1}+\kappa_{q} c_1^{-1}\bar{\eta}^2 (x,\tau_q^{-1}t),
\end{equation}
Therefore we have, using that $(\tilde\eta_i)_i$ is a partition of unity
\begin{equation}\label{e:bAparts}
\bar{A}_q	-\tilde\kappa_q\Id =\kappa_{q+1}\sum_i\tilde\eta_i(\nabla\Phi_i^{-1}\nabla\Phi_i^{-T}-\Id)+\kappa_q (\delta_{q+1}^{-1}\sigma_q-c_1^{-1}) \bar{\eta}^2 \Id-\kappa_q \delta_{q+1}^{-1} \bar{\eta}^2 \mathring{\bar{R}}_q,
\end{equation}
so, using \eqref{e:propertyofPhii}, \eqref{e:propertyofsigmaq}, \eqref{e:gluingR}
\begin{equation}\label{e:bAparts2}
|\bar{A}_q	-\tilde\kappa_q\Id| \le C \kappa_{q+1}\sum_i\tilde\eta_i \lambda_q^{-\gamma_T} + C\kappa_q \bar{\eta}^2 (\lambda_q^{-\gamma_E}+\lambda_q^{-\gamma_R} +\lambda_q^{-(b-1)\beta})+ C\kappa_q  \bar{\eta}^2 \delta_{q+1}^{-1} \mathring{\delta}_{q+1}\ell_q^{-2\alpha}. 
\end{equation}
Using now that by definition $\mathring{\delta}_{q+1}=\delta_{q+1}\lambda_q^{-\gamma_R}$, the fact that we can choose $\alpha$ arbitrarily small, and the definition \eqref{e:defkappabarq} of $\tilde{\kappa}_q$, we see that \eqref{e:bAparts2} yields for all $(x,t)$
\begin{equation}\label{e:barAbarkappa}
|\bar{A}_{q}(x,t)-\tilde{\kappa}_q(x,t)\Id|\leq C \tilde{\kappa}_{q}(x,t)\lambda_{q}^{-\min\{\gamma_T,\gamma_R,\gamma_E,(b-1)\beta\}}  \le C\tilde{\kappa}_{q}(x,t) \lambda_q^{-2\gamma}
\end{equation}
for some fixed constant $C$, where the latter inequality of \eqref{e:barAbarkappa} holds via the choice \eqref{e:global_gamma} of $\gamma$.
Then, as in Step 1 we choose $a_0$ sufficiently large to ensure that $C\lambda_q^{-2\gamma} \le \tfrac{1}{10}\lambda_q^{-\frac32\gamma}$. In particular, we may bound pointwise
\begin{equation}\label{e:comparebarAtildekappa}
\frac{1}{2}\tilde\kappa_q\Id\leq \bar{A}_q\leq 2\tilde\kappa_q\Id.
\end{equation}

We intend now to apply Corollary \ref{r:hom_dissipation} of Section \ref{s:homogenization}, with $\gamma$ given in \eqref{e:global_gamma} and the following choices
\begin{equation}\label{e:homo_connect}
\begin{split}
    \rho =& \tilde\rho_{q+1}, \\ 
    \bar \rho =& \bar{\rho}_q^{(1)}, \\
    u =& \bar u_q, \\
    \tilde R_{i} =& \tilde R_{q,i},
\end{split}
\quad \quad \quad 
\begin{split}
    \kappa =& \kappa_{q+1}, \\ 
    \lambda =& \lambda_{q+1}, \\ 
    \ell =& \ell_q, \\ 
    \tau =& \tau_q,
\end{split}
\quad \quad \quad 
\begin{split}
    \delta =& \delta_{q+1}, \\
    \sigma =& \sigma_q, \\
    \bar \kappa =& \tilde \kappa_q.
\end{split}
\end{equation}
In order to proceed, we need to verify the assumptions gathered in Section \ref{s:ass_homo} (for Proposition \ref{p:hom_trho}) and the assumption \eqref{r:hom_dissipation_eq0} (needed additionally for Corollary \ref{r:hom_dissipation}) for these choices of parameters. First of all, direct computations yield \eqref{e:hom_para_assump:deltaell}-\eqref{e:hom_para_assump:Nh} from \eqref{e:Dissipation}, \eqref{e:Dissipation2}, \eqref{e:global_gamma}, and \eqref{e:global_N}. The remaining assumptions of Section \ref{s:ass_homo} follow from the construction of the vector field in Section \ref{s:vectorfield}. More precisely, assumptions involving $\sigma$, i.e. \eqref{e:hom_assump:sigmadelta} and \eqref{e:hom_assump:sigma}, follow from \eqref{e:propertyofsigmaq}. Assumptions on $\Phi_i$ follow from \eqref{e:propertyofPhii} and \eqref{e:elltau}. The assumption \eqref{e:hom_assump:barkappa} is given by \eqref{e:comparebarAtildekappa} and the assumption \eqref{e:hom_assump:u} on $u$ is given by \eqref{e:gluingu}. The assumption \eqref{e:hom_assump:slow} involving $R$ follows from \eqref{e:propertyoftildeRqi2} and the fact that the functions $a_{\vec{k}}$ and their derivatives are bounded by absolute constants. The required properties of functions $\eta_i$ and  $\tilde \eta_i$ follow from \eqref{e:propertyofetai}, \eqref{e:propertyoftildeetai}. 
In turn, the explicit formula \eqref{e:hom_ed:chi} can then be used to verify assumption \eqref{e:hom_assump:chi}. Finally, \eqref{r:hom_dissipation_eq0} follows from \eqref{e:boundoninitialdatum}. 

Summarizing, Proposition \ref{p:hom_trho}  and Corollaries \ref{c:hom_trho}, \ref{c:hom_dissipation} and \ref{r:hom_dissipation} are applicable. We obtain:
\begin{align*}
   \sup_{t \le T}  \|\tilde\rho_{q+1}(t) - \bar{\rho}_{q}^{(1)}(t)\|_{L^2}^2 \lesssim& \lambda_q^{-4\gamma}\left|\|\rho_{in}\|^2_{L^2}-\|\bar{\rho}^{(1)}_{q}(T)\|_{L^2}^2\right|, \\ 
    \left|\|\tilde\rho_{q+1}(T)\|^2_{L^2}-\|\bar{\rho}_{q}^{(1)}(T)\|_{L^2}^2\right| \lesssim& \lambda_q^{-2\gamma}\left|\|\rho_{in}\|^2_{L^2}-\|\bar{\rho}^{(1)}_{q}(T)\|_{L^2}^2\right|\,.
\end{align*}
Thus, using $D_{q}^{(1)}=\tfrac12\left|\|\rho_{in}\|^2_{L^2}-\|\bar\rho^{(1)}_{q}(T)\|_{L^2}^2\right|$,
\begin{equation}\label{e:step2}
\begin{split}
   \sup_{t \le T}  \|\tilde\rho_{q+1}(t) - \bar{\rho}_{q}^{(1)}(t)\|_{L^2}^2 \le&   \lambda_q^{-3 \gamma}D_{q}^{(1)},\\
\left|\tilde D_{q+1}-D_{q}^{(1)}\right|\le C \lambda_q^{-2\gamma}D_{q}^{(1)} \le&   \lambda_q^{-\frac32 \gamma}D_{q}^{(1)}. 
\end{split}
\end{equation}

\smallskip

\noindent\emph{Step 3: Diagonal reduction ($ \bar\rho_{q}^{(1)} \leadsto \bar\rho_{q}^{(2)}$)}

Using once more \eqref{e:barAbarkappa}, we now apply Proposition \ref{p:stability_in_ellipticity} with 
$\varepsilon=\tfrac{1}{10}\lambda_q^{-\frac32 \gamma}$ to conclude
\begin{equation}
\label{e:step3}  
\begin{split}
\sup_{t \le T}  \|\bar{\rho}_{q}^{(2)}(t) - \bar{\rho}_{q}^{(1)}(t)\|_{L^2}^2 \le \lambda_q^{-3 \gamma}D_{q}^{(2)}, \\
\left|D_q^{(2)}-D_q^{(1)}\right|\le \lambda_q^{-\frac32\gamma}D_{q}^{(2)},
\end{split}
\end{equation}
where $D_{q}^{(i)}=\tfrac12\left|\|\rho_{in}\|^2_{L^2}-\|\bar\rho^{(i)}_{q}(T)\|_{L^2}^2\right|$ and $\bar{\rho}^{(2)}_{q}$ is the solution of 
\begin{equation}\label{e:equationkappaqbar}
\begin{split}
\partial_t\bar{\rho}^{(2)}_{q}+\bar{u}_{q}\cdot\nabla\bar{\rho}^{(2)}_{q}&=\div \tilde{\kappa}_q \nabla\bar{\rho}^{(2)}_{q}\,,\\	
\bar{\rho}_{q}^{(2)}|_{t=0}&=\rho_{in}\,.
\end{split}	
\end{equation}

\smallskip

\noindent\emph{Step 4: Time averaging ($ \bar\rho_{q}^{(2)} \leadsto \bar\rho_{q}^{(3)}$)}

The advection-diffusion equation \eqref{e:equationkappaqbar} has two different characteristic time-scales: the advective time-scale $\|\nabla \bar{u}_q\|_{L^\infty}^{-1}$ and the time-scale $\tau_q$ given by the time-oscillatory behaviour of the ellipticity coefficient $\tilde{\kappa}_q(x,t)$ given in \eqref{e:defkappabarq}, with the relationship given by \eqref{e:elltau}: $\|\nabla \bar{u}_q\|_{L^\infty}\tau_q\leq M \lambda_q^{-\gamma_T}$. Let us now invoke Proposition \ref{p:t_avg} with the following choices
(on the right-hand side there are ingredients of the current construction, which define the objects appearing in Proposition \ref{p:t_avg})
\[
    \eta := c_1^{-1}\bar{\eta}^2, \quad \kappa_0 := \kappa_{q+1}, \quad \kappa_1 := \kappa_{q},
    \quad \mu := \delta_q^{\sfrac{1}{2}} \lambda_q, \quad \tau := \tau_q, \quad u:=\bar u_q.
\]
Proposition \ref{p:t_avg} requires assumptions (A1)-(A4). The assumption (A1) follows from $\tau \mu = \lambda_q^{-\gamma_T}$ (valid in view of \eqref{e:elltau}), and the second lower bound for $\tilde N$ in \eqref{e:global_N}; observe that the resulting $N$ of (A1) does not depend on $q$. The assumptions (A2)-(A3) follow from assumptions \eqref{e:Dissipation}, \eqref{e:Dissipation2}, \eqref{e:boundoninitialdatum}, \eqref{e:gluingu}, $b>1$, as well as \eqref{e:elltau}. Control of the cutoff-functions in \eqref{e:propertyofetai} implies (A4).
Proposition \ref{p:t_avg}  then yields
\begin{align*}
	\sup_{t \le T} \|\bar{\rho}^{(3)}_{q}(t)-\bar{\rho}^{(2)}_{q}(t)\|_{L^2}^2 \leq C\lambda_q^{-2\gamma_T}(D^{(2)}_q+D^{(3)}_q),	
\end{align*}
\begin{equation}\label{e:step4-1}
	\left|\|\bar{\rho}^{(3)}_{q}(T)\|^2_{L^2}-\|\bar{\rho}^{(2)}_{q}(T)\|_{L^2}^2\right|\leq C\lambda_q^{-\gamma_T}(D^{(2)}_q+D^{(3)}_q),	
\end{equation}
where 
\begin{equation*}
D^{(i)}_q=\tfrac12 \left|\|\rho_{in}\|^2_{L^2}-\|\bar{\rho}^{(i)}_{q}(T)\|_{L^2}^2\right|
\end{equation*}
and $\bar{\rho}^{(3)}_q$ is the solution of  
\begin{equation}\label{e:equationkappaq}
\begin{split}
\partial_t\bar{\rho}^{(3)}_{q}+\bar{u}_{q}\cdot\nabla\bar{\rho}^{(3)}_{q}&=(\kappa_{q+1}+\kappa_{q})\Delta\bar{\rho}^{(3)}_{q}\,,\\	
\bar{\rho}_{q}^{(3)}|_{t=0}&=\rho_{in}\,.
\end{split}	
\end{equation}
Choosing $a_0$ sufficiently large, we may ensure that $C\lambda_q^{-\gamma_T}<\tfrac{1}{4}\lambda_q^{-\frac32 \gamma}$ in \eqref{e:step4-1}, from which we can then conclude  
\begin{equation*}
\left|\frac{D_q^{(2)}}{D_q^{(3)}}-1\right|\leq \frac{1}{4}\lambda_q^{-\frac32 \gamma}\left(1+\frac{D_q^{(2)}}{D_q^{(3)}}\right).
\end{equation*}
We deduce
\begin{equation}\label{e:step4}
\left|\frac{D_q^{(2)}}{D_q^{(3)}}-1\right|\leq \frac{1}{2}\lambda_q^{-\frac32 \gamma}.
\end{equation}

\smallskip

\noindent\emph{Step 5: Gluing estimate ($ \bar\rho_{q}^{(3)} \leadsto \rho_{q}$)}

Finally, we compare \eqref{e:equationkappaq} to \eqref{e:equationq} by using the estimate \eqref{e:gluingz}. Indeed, using potentials for velocities (cf. \eqref{e:Biot_Savart}) we can write \eqref{e:equationkappaq} as
\begin{equation*}
\begin{split}
\partial_t\bar{\rho}^{(3)}_{q}+u_{q}\cdot\nabla\bar{\rho}^{(3)}_{q}&=\div \left(\kappa_q \nabla\bar{\rho}^{(3)}_{q}+\kappa_{q+1}\nabla\bar{\rho}^{(3)}_{q}+(z_q-\bar{z}_q)\times\nabla\bar{\rho}^{(3)}_q\right)\,,\\	
\bar{\rho}_{q}^{(3)}|_{t=0}&=\rho_{in}\,.
\end{split}	
\end{equation*}
On the other hand \eqref{e:gluingz}, \eqref{e:elltau} and our choice of parameters \eqref{e:conditiononalpha1},  \eqref{e:global_gamma} imply
\begin{equation}\label{e:gluingestimate}
\|z_q-\bar{z}_q\|_{L^\infty}\le C \tau_q\delta_{q+1}\lambda_q^{-\gamma_R+\alpha(1+\gamma_L)}\leq  C \kappa_q\lambda_q^{-2\gamma}\,,\quad \kappa_{q+1}=\lambda_q^{-(b-1)\theta}\kappa_q\leq \kappa_q\lambda_q^{-4\gamma}.	
\end{equation}
A final application of Proposition \ref{p:stability_in_ellipticity}, assuming $a_0$ is sufficiently large to absorb the constant, leads to 
\begin{align*}  \sup_{t \le T} \|\bar{\rho}^{(3)}_{q}(t) - \rho_{q}(t)\|_{L^2}^2 &\le \lambda_q^{-3 \gamma}\left|\|\rho_{in}\|^2_{L^2}-\|\rho_{q}(T)\|_{L^2}^2\right|, \\
    \left|\|\bar{\rho}^{(3)}_{q}(T)\|^2_{L^2}-\|\rho_{q}(T)\|_{L^2}^2\right| &\le \lambda_q^{-\frac32 \gamma}\left|\|\rho_{in}\|^2_{L^2}-\|\rho_{q}(T)\|_{L^2}^2\right|,
\end{align*}
where $\rho_q$ is the solution of \eqref{e:equationq}, equivalently
\begin{equation}\label{e:step5}
\begin{split}
 \sup_{t \le T} \|\bar{\rho}^{(3)}_{q}(t) - \rho_{q}(t)\|_{L^2}^2 \le \lambda_q^{-3\gamma}D_{q}\,,\\
	\left|D_q^{(3)}-D_{q}\right|\le \lambda_q^{-\frac32\gamma}D_{q}\,.
     \end{split}
\end{equation}

\smallskip
\noindent\emph{Final estimate}

Overall, we see that after the five stability steps above, we achieve the estimate
\begin{equation}\label{e:stab_af}
\frac{D_{q+1}}{D_q}=\frac{D_{q+1}}{\tilde D_{q+1}}\frac{\tilde D_{q+1}}{D_q^{(1)}}\frac{D_{q}^{(1)}}{D_q^{(2)}}\frac{D_{q}^{(2)}}{D_q^{(3)}}\frac{D_{q}^{(3)}}{D_q}\geq \left(1-\lambda_q^{-\frac32\gamma} \right)^5D_q\geq \left(1-\frac12 \lambda_q^{-\gamma}\right)D_q,
\end{equation}
where we have again assumed $a_0$ is sufficiently large to absorb the exponent $5$ and gain the factor $\frac12$ at the expense of decreasing  $\frac32\gamma$ to $\gamma$.
The estimate \eqref{e:stab_af} together with the analogous estimate from above yields \eqref{e:mainstability}. Similarly we obtain \eqref{e:mainstability2}.
The proof of Proposition \ref{p:main} is concluded.
\end{proof}

\subsection{Construction of the vector field - h-principle}
\label{s:hprinciple}

In Section \ref{s:vectorfield} we detailed the main iteration scheme for producing H\"older-continuous weak solutions of the Euler equations. What remains is to produce an initial vector field and associated Reynolds tensor, which satisfies the inductive assumptions \eqref{e:R_q_inductive_est}-\eqref{e:energy_inductive_assumption} for \emph{some} $q\in\N$. 
To this end we recall that a smooth strict subsolution to the Euler equations is a smooth triple $(\bar{u},\bar{p},\bar{R})$ on $\T^3\times[0,T]$ solving the Euler-Reynolds system \eqref{e:EulerReynolds} with the normalizations $\dashint_{\T^3}\bar{u}\,dx=0$, $\dashint_{\T^3}\bar{p}\,dx=0$, such that $\bar{R}(x,t)$ is uniformly positive definite on $\T^3\times[0,T]$. The energy associated with a subsolution is (cf.~\cite{DaSz2017,DaRuSz2023,BDSV})
\begin{equation}\label{e:hprinciple_energy}
    e(t)=\dashint_{\T^3}(|\bar{u}|^2+\tr\bar{R})\,dx.
\end{equation}
We have the following statement, which is a variant of \cite[Proposition 3.1]{DaSz2017}, see also \cite[Theorem 7.1]{BDSV} and \cite[Proposition 4.1]{DaRuSz2023}. 

\begin{proposition}\label{p:initialu}
    Let $(\bar{u},\bar{p},\bar{R})$ be a smooth strict subsolution with energy $e$. Let the parameters $M,\bar{e}$,  $\beta,b,\gamma_T,\gamma_R,\gamma_E,\alpha>0$ and $a_0\gg 1$ be as in Proposition \ref{p:Onsager}, and $(\delta_q,\lambda_q)_{q\in\N}$ be as in \eqref{e:lambdadelta}. 
    
    Then, there exists $q_0\in\N$ such that, for any $q\geq q_0$ there exists $(u_q,\mathring{R}_q)$ solving \eqref{e:EulerReynolds} with \eqref{e:ERconstraints} such that \eqref{e:R_q_inductive_est}-\eqref{e:energy_inductive_assumption} hold. Moreover, for any $\alpha\in(0,1)$
    \begin{align}
        \|\bar{u}-u_q\|_{C^0}&\leq \bar{C}\|\bar{R}\|_{C^0}^{\sfrac12}\,,\label{e:hprinciple-C0}\\
        \|\bar{z}-z_q\|_{C^\alpha}&\leq \bar{C}\|\bar{R}\|_{C^\alpha}^{\sfrac12}\lambda_q^{-1+\alpha}\label{e:hpriciple-C-1}\,,
    \end{align}
    where the constant $\bar{C}$ only depends on $\alpha$, $\bar{u}$ and on $\delta:=\inf_{(x,t)}\frac{\min\{\bar{R}\zeta\cdot\zeta:|\zeta|=1\}}{\tr\bar{R}}$.
\end{proposition}

\begin{proof}
    We construct, following \cite{DaSz2017,DaRuSz2023}, $u_q:=\bar{u}+w$, where $w$ is from the perturbation step of Section \ref{s:Onsager}, but without the need for time-cutoffs. More precisely, we define
    \begin{equation*}
        \tilde R=\bar{R}-\tfrac13\delta_{q+1}\Id.
    \end{equation*}
    For $q$ sufficiently large, we ensure that
    \begin{equation*}
        \frac{1}{\tr \tilde R}\tilde R\geq \tfrac{1}{2}\delta\Id.
    \end{equation*}
    Further, define $\Phi=\Phi(x,t)$ to be the inverse flow associated to the vector field $\bar{u}$, i.e.~such that $(\partial_t+\bar{u}\cdot\nabla)\Phi=0$ and $\Phi(x,0)=x$. Since $\bar{u}\in C^\infty(\T^3\times[0,T])$ with $\div\bar{u}=0$, for every $t$ the map $\Phi(\cdot,t)$ is a volume-preserving diffeomorphism, and there exists a constant $C$ such that 
    \begin{equation*}
        C^{-1}\leq |\nabla\Phi|,|\nabla\Phi^{-1}|\leq C.
    \end{equation*}
    Now we apply Lemma \ref{l:Mikado} with 
    \begin{equation*}
        \mathcal{N}:=\left\{R\in \mathcal{S}^{3\times 3}_+:\,R=\tfrac{1}{\tr S}ASA^T\textrm{ where $C^{-1}\leq |A|,|A^{-1}|\leq C$ and $S\geq \tfrac12\delta\tr S \, \Id$}\right\}\,,
    \end{equation*}
    to obtain $\Lambda\subset \mathbb{Q}^3$ and $a_{\vec{k}}$ satisfying \eqref{e:MikadoProperty}.
    We set (cf.~\eqref{e:defRqi})
    \begin{align*}
        \tilde S=\frac{1}{\tr \tilde R}\nabla\Phi \tilde R\nabla\Phi^T,
    \end{align*}
    and observe that $\tilde S(x,t)\in \mathcal{N}$. The new perturbation $w$ is then defined, in analogy with \eqref{e:defwq+1}, as
\begin{equation}\label{e:hprinciple_defw}
	w=\frac{1}{\lambda}\curl\left[\sum_{\vec{k}\in\Lambda}(\tr\tilde R)^{\sfrac12}a_{\vec{k}}(\tilde S)\nabla\Phi^TU_{\vec{k}}(\lambda\Phi)\right]\,.
\end{equation}
The associated Reynolds stress is (cf.~\eqref{e:decompR})
\begin{equation*}
\mathring{R}_{q}=\mathcal{R}\left[\partial_tw+\bar{u}\cdot\nabla w+w\cdot\nabla\bar{u}+\div(w\otimes w-\tilde R)\right],
\end{equation*}
where $\mathcal R$ is the inverse divergence operator on symmetric $2$-tensors, introduced in \cite{DSz13} (cf.~\eqref{e:R:def} in Section \ref{s:Onsager}).
Since $\bar{R}-\tilde R$ is a constant multiple of the identity, it easily follows, as in \cite{DaSz2017,BDSV}, that $u_q,\mathring{R}_q$ is a solution of \eqref{e:EulerReynolds}. Moreover, the estimates in \cite{DaSz2017,BDSV} (see also Section \ref{s:perturbation}) imply
\begin{align*}
    \|\mathring{R}_q\|_{C^0}&\leq C\lambda^{-1+\alpha}\,,\\
    \|w\|_{C^n}&\leq C\|\bar{R}\|_{C^0}^{\sfrac12}\lambda^n\quad\textrm{ for all }n\leq \bar{N}\,,\\
    \left|\dashint_{\T^3}|u_q|^2-|\bar u|^2-\tr\tilde R\,dx\right|&\leq C\lambda^{-1+\alpha}\,. 
\end{align*}
Here the constant $C$ depends on $(\bar{u},\bar{R})$ and on $\mathcal{N}$.
It remains to choose $\lambda$ and $q$. We fix an exponent $0<\gamma_o$ such that
\begin{equation*}
\beta<\gamma_o\textrm{ and }\alpha+\gamma_o<1-2b\beta-\max\{\gamma_E,\gamma_R\}
\end{equation*}
and define $\lambda=\lambda_q^{1-\gamma_o}$ ($q$ still to be fixed).
Observe that \eqref{e:b_beta_rel}-\eqref{e:Onsager_Conditions} guarantee the existence of such $\gamma_o$. Now, validity of \eqref{e:R_q_inductive_est}-\eqref{e:energy_inductive_assumption} follows from
\begin{equation*}
C\lambda^{-1+\alpha}\leq \delta_{q+1}\lambda_q^{-\max\{\gamma_E,\gamma_R\}},\textrm{ as well as }C\lambda\leq M\delta_q^{\sfrac12}\lambda_q,\,\lambda\leq\lambda_q.
\end{equation*}
But by our choice of the exponent $\gamma_o$ these inequalities are satisfied for $q$ sufficiently large. 
\end{proof}

\section{Proofs of main results}\label{s:main_proof}

Now we are ready to state our main result in full generality: 
\begin{theorem}\label{t:main_general}
Let $T_0<\infty$. Let $(\bar{u},\bar{p},\bar{R})$ be a smooth strict Euler subsolution on $\T^3 \times [0,T_0]$ with smooth, strictly positive energy $e: [0,T_0] \to \R$. Let $0<\beta<1/3$ and $1<b<\min\{\sqrt{\tfrac32},\tfrac{1-\beta}{2\beta}\}$. There exists $\bar{C}$, depending only on $\bar{u}$ and on $\inf_{(x,t)}\frac{\min\{\bar{R}\zeta\cdot\zeta:|\zeta|=1\}}{\tr\bar{R}}$ such that, for any $\epsilon>0$, we have:

There exists a weak solution $u\in C^{\beta}(\T^3\times [0,T_0])$ to the Euler equations \eqref{e:Euler} with kinetic energy $\int_{\T^3}|u(x,t)|^2\,dx=e(t)$ for all $t$, such that 
    \begin{align}
        \|\bar{u}-u\|_{C^0}&\leq \bar{C}\|\bar{R}\|_{C^0}^{\sfrac12}\,,\label{e:hprincipleT1}\\
        \|\bar{z}-z\|_{C^0}&\leq \epsilon\label{e:hprincipleT2},
    \end{align}
where $\bar{z}, z$ are potentials for, respectively, $\bar{u}$, $u$ (cf.~\eqref{e:Biot_Savart}).

Moreover, there exists $\delta_0>0$ (depending on the previous choices) such that for any $T\le T_0$ and for every non-zero initial datum $\rho_{in}\in H^1(\T^3)$ with zero mean, the family of unique solutions $\{\rho_{\kappa}\}_{\kappa>0}$ to advection-diffusion equation \eqref{e:advectiondiffusion} with velocity field $u$ satisfies 
\begin{equation}\label{e:anomalousdissipation_gen}
\limsup_{\kappa\to 0}\quad \kappa\int_0^T\|\nabla\rho_\kappa\|_{L^2}^2\,dt\geq c_0 \|\rho_{in}\|_{L^2}^2,
\end{equation}
where $c_0>0$ satisfies
\begin{equation} \label{e:main_const}
c_0\geq  C(\beta,b)\ell_{in}^{-2}\min\left\{ \left(\frac{\ell_{in}}{\ell_0}\right)^{\frac{4b}{2-b\theta}}T^{\frac{2b}{2-b\theta}},\delta_0,\ell_{in}^{2b}\right\}.
    \end{equation}
where $\ell_0 := \sup_{f \in \dot H^1 (\T^3)} \frac{\|f\|_{L^2}} {\|\nabla f\|_{L^2}}$,  
$ \ell_{in}:=\frac{\|\rho_{in}\|_{L^2}} {\|\nabla \rho_{in}\|_{L^2}}$, and $\theta=\frac{2b}{b+1}(1+\beta)$ (c.f.~\eqref{e:defkappaq})
\end{theorem}

Before proving Theorem \ref{t:main_general}, we show how to deduce Theorem \ref{t:main} as well as the density statements in Section \ref{ss:dissipation}. 

\begin{proof}[Proof of Theorem \ref{t:main}]
Take $\bar{u}=0$, $\bar{p}=0$ and $\bar{R}=\tfrac{1}{3}e(t)\Id$ for a strictly positive, smooth function $e$ (say, constant) as the strict subsolution in Theorem \ref{t:main_general}. This generates an Euler solution $u\in C^{\beta}(\T^3\times [0,T_0])$ such that the advection-diffusion equation with velocity $u$ exhibits anomalous dissipation \eqref{e:anomalousdissipation_gen}. For fixed $u$ and $b$ the constant \eqref{e:main_const} depends only on $\ell_{in}, \beta, T$ and $\ell_{0}$.
\end{proof}

Now we come to the density statements in Section \ref{ss:hprinciple}. 
\begin{itemize}
\item[{\bf Strong density}] Let $(\bar{u}, \bar{p})$ be a smooth solution of the Euler equations and consider the strict subsolution $(\bar{u}, \bar{p}, \bar{R}_n)$, where $\bar{R}_n = \frac{1}{n} \Id$. Then $\inf_{(x,t)}\frac{\min\{\bar{R}\zeta\cdot\zeta:|\zeta|=1\}}{\tr\bar{R}}=3$, so that the constant $\bar C$ of \eqref{e:hprincipleT1} is uniform in $n$. Therefore, in light of estimate \eqref{e:hprincipleT1}, Theorem \ref{t:main_general} yields a sequence $u^n$ of weak solutions to the Euler equations inducing anomalous dissipation, with $u^n\to \bar{u}$ uniformly. 
\item[{\bf Weak density}] Let $\bar{u}$ be an arbitrary smooth divergence-free vector field. We may write, for an arbitrary smooth positive function $\tilde e=\tilde e(t)$, 
\[
\partial_t \bar{u}+\bar{u}\cdot \nabla \bar{u} = \div(\tilde e(t) \Id +\RR(\partial_t \bar{u})+\bar{u}\otimes\bar{u}),
\]
where $\RR$ is the symmetric antidivergence operator as in the proof of Proposition \ref{p:initialu} (cf.~\eqref{e:R:def}). In particular, choosing $\tilde e(t)$ sufficiently large we may ensure that the symmetric tensor $\bar{R}:=\tilde e(t) \Id +\RR(\partial_t \bar{u})+\bar{u}\otimes\bar{u}$ is positive definite, so that $(\bar{u},0,\bar{R})$ is a strict subsolution. In light of the estimate \eqref{e:hprincipleT2}, Theorem \ref{t:main_general} then yields a sequence $u^n$ of weak solutions to the Euler equations inducing anomalous dissipation, with the corresponding potentials $z^n\to \bar{z}$ uniformly. 
\end{itemize}

\subsection{Proof of Theorem \ref{t:main_general}}

We start with a smooth strict Euler subsolution $(\bar{u},\bar{p},\bar{R})$ with strictly positive energy profile $e: [0,T_0] \to \R$ and $0<\beta<1/3$. 

\smallskip
\noindent\emph{Step 1. Choice of parameters.} 
We fix $\gamma_T,\gamma_R,\gamma_E>0$ as follows
\begin{equation}\label{e:choiceofgammaTgammaR}
\gamma_E=\gamma_T=\gamma_R=\frac{b-1}{b(b+1)}(1-(2b+1)\beta),
\end{equation}
which implies
\begin{align}
    \gamma_T+b\gamma_R&<(b-1)(1-(2b+1)\beta)\,,\label{e:conditiontransport}\\
    \gamma_T&<\frac{b-1}{b+1}(1-(2b+1)\beta)\,,\label{e:conditionhomogenization}\\
    \gamma_R+\gamma_T&>\frac{b-1}{b+1}(1-(2b+1)\beta)\,.\label{e:conditiongluing}
\end{align}

Note that \eqref{e:conditiontransport} gives the requirement \eqref{e:Onsager_Conditions} in Proposition \ref{p:Onsager}, whereas \eqref{e:conditionhomogenization} and \eqref{e:conditiongluing} are the left and right inequalities in \eqref{e:Dissipation} in Proposition \ref{p:main}. 

Having fixed the parameters $e, \beta, b,  \gamma_T, \gamma_R, \gamma_E$, we choose $\gamma_L>0$ sufficiently small in accordance with the requirements of Propositions \ref{p:Onsager}, \ref{p:main}. This in turn fixes $ \bar N$ via Proposition  \ref{p:Onsager} and $\gamma, \tilde N$ via Propositions \ref{p:main}. Next, let small $\alpha$ and then  $a_0$ be chosen according to requirements of Propositions \ref{p:Onsager}, \ref{p:main}. Then, for any $a\geq a_0$ the two propositions are applicable, with the sequence $\{(\lambda_q,\delta_q)\}_q$ determined according to \eqref{e:lambdadelta}, i.e. $\lambda_q=2\pi\lceil a^{b^q}\rceil$ and $\delta_q=\lambda_q^{-2\beta}$. In particular, choosing $a\gg a_0$ sufficiently large we may ensure
\begin{equation}\label{e:pwrtail}
\prod_{q'\geq {1}}(1-\tfrac12\lambda_{q'}^{-\gamma}) > \tfrac12.
\end{equation}

\smallskip

\noindent\emph{Step 2. Construction of the vector field $u$.} We apply Proposition \ref{p:initialu} to obtain $(u_{q_0},\mathring{R}_{q_0})$ satisfying the estimates \eqref{e:R_q_inductive_est}-\eqref{e:energy_inductive_assumption} and \eqref{e:hprinciple-C0}- \eqref{e:hpriciple-C-1}. Then Proposition \ref{p:Onsager} applied inductively for any $q\geq q_0$ yields the sequence $\{u_q\}_{q\geq q_0}$. Arguing as in \cite{BDSV}, we deduce that $u_q\to u$ in $C([0,T];C^{\beta'}(\T^3))$ for any $\beta'<\beta$. Moreover, using the Euler equations, that $u\in C^{\beta''}([0,T]\times\T^3)$ for any $\beta''<\beta'$. Since $\beta''<\beta'<\beta<1/3$ is arbitrary in this argument, by renaming $\beta''$ to $\beta$ we deduce the existence of $u\in C^\beta([0,T]\times\T^3)$ as in the statement of Theorem \ref{t:main_general}. In particular, using \eqref{e:energy_inductive_assumption} and taking the limit implies that $e(t) =\int_{\T^3}\abs{u}^2\,dx$. The inequality \eqref{e:hprincipleT1} follows from \eqref{e:v_diff_prop_est} and \eqref{e:hprinciple-C0}, whereas the inequality \eqref{e:hprincipleT2} follows from \eqref{e:hpriciple-C-1}, \eqref{e:gluingz} and \eqref{e:defwq+1}.

\smallskip

\noindent\emph{Step 3. Choice of inertial scale.} Having constructed the vector field $u$ with active length-scales $\lambda^{-1}_q$, $q \ge q_0$, our next goal is to choose the length scale (more precisely, the index $q_I$), below which small-scale advection becomes dominant. To this end let $\rho_{in}\in \dot H^1(\T^3)\setminus\{0\}$ and define its characteristic length-scale as
\begin{equation*}
    \ell_{in}:=\frac{\|\rho_{in}\|_{L^2}} {\|\nabla \rho_{in}\|_{L^2}}.
\end{equation*}
We fix $q_I\geq q_0$ so that
\begin{equation}\label{e:conditionq1theta}
10 C_{\tilde N} \ell_{in}^{-1} \leq   \lambda_{q}\left(1-e^{-2\kappa_{q+1}\ell_0^{-2}T}\right)^{\sfrac12}\quad\textrm{ for all }q\geq q_I,
\end{equation}
where $C_{\tilde N}\geq 1$ is a constant (depending only on $\tilde N$ from Step 1) related to mollification as follows: 
We fix a standard symmetric mollifier $\psi$ (one can also use e.g.~$\psi$ from Section \ref{s:mollify}) and set $\psi_r(x)=r^{-3}\psi(\frac{x}{r})$. Then we have (cf.~Lemma \ref{l:mollify}) for all $g\in \dot H^1 (\T^3)$
\begin{equation}\label{e:moll_gen}
\begin{aligned}
    \|g*\psi_{r}-g\|_{L^2}&\leq C_{\tilde N} r\|\nabla g\|_{L^2}, \\
    \|g*\psi_{r}\|_{H^{n+1}}&\leq C_{\tilde N} r^{-n}\|\nabla g\|_{L^2}\textrm{ for all }n\leq \tilde N.
\end{aligned}
\end{equation}
To see that the choice of $q_I$ is possible, we verify that the expression on the right hand side of \eqref{e:conditionq1theta} is monotone increasing in $q$. To see this, in view of \eqref{e:defkappaq} it suffices to show that the function 
\begin{equation}\label{e:functionf}
f(\lambda):=\lambda^2(1-e^{-2\lambda^{-b\theta}})	
\end{equation}
is strictly monotone increasing for $\lambda\in(0,\infty)$ with $f(\lambda)\to\infty$ as $\lambda\to\infty$. First of all, the asymptotic behaviour as $\lambda\to\infty$ is easily deduced from the Taylor expansion of the exponential, which gives $f(\lambda)=O(\lambda^{2-b\theta})$. Recalling the choice of $\theta$ from \eqref{e:defkappaq} as well as $b<\sqrt{\tfrac32}$ and $\beta < \frac{1}{3}$, we obtain $b\theta<2$, showing that $\lim_{\lambda\to\infty}f(\lambda)=\infty$.  Then, computing the derivative we obtain
\begin{equation*}
f'(\lambda)=2\lambda\left(1-(1+\tfrac12b\theta \lambda^{-b\theta})e^{-\lambda^{-b\theta}}\right)=2\lambda\left(1-g(\lambda^{-b\theta})\right),	
\end{equation*}
where $g(s)=(1+\tfrac12b\theta s)e^{-s}$. Observe that $g(0)=1$ and $g(s)\to 0$ as $s\to\infty$. On the one hand direct computation shows that $g'(s)=0$ only if $s=1-\frac{2}{b\theta}$. Using again that $b\theta<2$ we deduce that $g$ has no local maxima/minima in $(0,\infty)$. Consequently $g(s)<1$ for all $s>0$, implying that $f$ is strictly monotone increasing on $(0,\infty)$.

Our goal in the remaining steps is to show that the solution $\rho_\kappa$ of \eqref{e:advectiondiffusion} can be approximated by the solution $\tilde\rho_{q_I}$ of the equation $\partial_t\tilde\rho_{q_I}+u_{q_I}\cdot\nabla \tilde\rho_{q_I}=\kappa_{q_I}\Delta\tilde\rho_{q_I}$.

\smallskip

\noindent\emph{Step 4. Mollification of initial data.} Define
\begin{equation}\label{e:choice_moll}
    \tilde\rho_{in}:=\rho_{in}*\psi_{r} \quad \text{ with } \quad r=\lambda^{-1}_{q_I}
\end{equation}
where $\psi$ is the mollifier of Step 3. Inequalities \eqref{e:moll_gen} and the definition of $\ell_{in}$ imply
\begin{align}
    \|\tilde\rho_{in}-\rho_{in}\|_{L^2}&\leq  C_{\tilde N} \lambda_{q_I}^{-1}  \ell_{in}^{-1}\|\rho_{in}\|_{L^2} \label{e:comparerhotilde}\\
    \|\tilde\rho_{in}\|_{H^{n+1}}&\leq  C_{\tilde N} \lambda_{q_I}^{n}\ell_{in}^{-1}\|\rho_{in}\|_{L^2}\quad\textrm{ for all }0\leq n\leq \tilde N.\label{e:Hnrhotilde}
\end{align}
From \eqref{e:comparerhotilde} together \eqref{e:conditionq1theta} we deduce
\begin{equation}\label{e:comparerhotilderho}
     \|\tilde\rho_{in}-\rho_{in}\|_{L^2}\leq \tfrac{1}{10}\|\rho_{in}\|_{L^2}\,.
\end{equation}

\noindent\emph{Step 5. Enhanced dissipation in the inertial regime.} For $q \ge q_I$, consider the solution $\tilde\rho_q$ of the equation
\begin{equation}  \label{e:main_proof:e5.5}
\begin{split}
	\partial_t \tilde\rho_{q} + u_{q} \cdot \nabla \tilde\rho_{q} =& \kappa_{q} \Delta \tilde\rho_{q}\,, \\ 
	\tilde\rho_{q}|_{t=0} =& \tilde\rho_{in}\,.
\end{split}
\end{equation}
Using the Poincar\'e inequality in conjunction with the classical energy estimate on \eqref{e:main_proof:e5.5} and Gronwall's inequality, we obtain the following lower bound on the dissipation rate:
\begin{equation} \label{e:mainproof_diffusionbound1a}
\begin{split}
	\tilde D_q&:=\kappa_{q} \int_0^T \| \nabla \tilde\rho_{q} \|^2_{L^2} dt \\
		&=  \frac{1}{2} \left( \|\tilde\rho_{in}\|_{L^2}^2 - \|\tilde\rho_{q}(T) \|_{L^2}^2 \right)\\
		&\geq \frac12\left( 1-\exp(- 2\ell^{-2}_0\kappa_{q}T) \right)  \|\tilde\rho_{in}\|_{L^2}^2.
\end{split}
\end{equation}
Using \eqref{e:conditionq1theta} and the fact that $\kappa_{q+1}\leq\kappa_q$ we then estimate from below as
\begin{equation*}
\begin{split}
	\tilde D_q&\geq 50C_{\tilde N}^2\ell_{in}^{-2}\lambda_q^{-2}\|\tilde\rho_{in}\|_{L^2}^2\\
	&\overset{\textrm{\eqref{e:comparerhotilderho}}}{\geq} 40C_{\tilde N}^2\ell_{in}^{-2}\lambda_q^{-2}\|\rho_{in}\|_{L^2}^2,
\end{split}
\end{equation*}
and similarly we obtain the same lower bound for $\tilde D_{q+1}$. In turn, applying \eqref{e:Hnrhotilde}, we deduce
\begin{equation*}
\|\tilde\rho_{in}\|_{H^{n+1}}\leq \tfrac{1}{6}\lambda_{q_I}^{n+1}\min\{\tilde D_{q}^{\sfrac12},\tilde D_{q+1}^{\sfrac12}\}\quad \textrm{ for any }0\leq n\leq\tilde N.	
\end{equation*}
Therefore Proposition \ref{p:main} is applicable to the equation \eqref{e:main_proof:e5.5} for any $q\geq q_I$. We obtain
\begin{align}
	(1-\tfrac12\lambda_q^{-\gamma}) \tilde D_q
		\leq \tilde D_{q+1},
\end{align}
and in particular deduce, using \eqref{e:pwrtail}, \eqref{e:comparerhotilderho} and once more \eqref{e:mainproof_diffusionbound1a}, for any $q\geq q_I$
\begin{equation} \label{e:mainproof:e10}
\begin{split}
	\kappa_{q} \int_0^T \| \nabla \tilde\rho_{q} \|^2_{L^2} dt &=\tilde D_q
	\geq	\prod_{q'=q_I}^{q-1} (1-\tfrac12\lambda_{q'}^{-\gamma})\tilde D_{q_I} 
	\overset{\textrm{\eqref{e:pwrtail}}}{\geq}  \frac12\tilde D_{q_I}\\
	&\geq \frac{1}{5}\left( 1-\exp(- 2\ell^{-2}_0\kappa_{q_I}T) \right)  \|\rho_{in}\|_{L^2}^2=:2c_0\|\rho_{in}\|_{L^2}^2\,,
\end{split}
\end{equation}
where we have introduced the constant of the dissipation rate
\begin{equation}\label{e:mainproof-c0}
c_0:=	\frac{1}{10}\left( 1-\exp(- 2\ell^{-2}_0\kappa_{q_I}T) \right)\,.
\end{equation}

\smallskip

\noindent\emph{Step 6. Dissipation in the viscous regime.} Next we
compare the solution $\tilde\rho_{q}$ of \eqref{e:main_proof:e5.5} with the solution $\tilde\rho_{u,q}$ of 
\begin{equation}  \label{e:mainproof-uqu}
\begin{split}
    \partial_t \tilde \rho_{u,q} + u \cdot \nabla \tilde \rho_{u,q} =& \kappa_q\Delta \tilde \rho_{u,q}, \\ 
    {\tilde \rho_{u,q}}|_{t=0} =& \tilde\rho_{in},
\end{split}
\end{equation}
i.e. with the same diffusion coefficient $\kappa_{q}$ and initial datum $\tilde\rho_{in}$, but advected by the full Euler flow $u$. Let us rewrite the equation \eqref{e:mainproof-uqu} as
\begin{equation}\label{e:moleq}
    \partial_t \tilde\rho_{u,q} + u_{q} \cdot \nabla \tilde \rho_{u,q} = \div(\kappa_q\nabla\tilde\rho_{u,q}+(z-z_{q})\times\nabla\tilde\rho_{u,q}),
\end{equation} 
where  $z_{q}, z$ are the vector potentials of $u_{q}, u$. We want to bound $z-z_{q}$ from above. To this end write
\begin{equation*}
    \|z_{q+1}-z_q\|_{C^0}\leq \|z_{q+1}-\bar{z}_q\|_{C^0}+\|\bar{z}_{q}-z_q\|_{C^0},
\end{equation*}
where $\bar{z}_q$ is the vector potential of $\bar{u}_q$ obtained in Proposition \ref{p:gluing}. Using the inverse divergence operator $\mathcal{R}$ introduced in \cite{DSz13} ((cf.~\eqref{e:R:def} in Section \ref{s:Onsager}), the formula \eqref{e:defwq+1} and standard Schauder estimates, we obtain
\begin{equation*}
    \|z_{q+1}-\bar{z}_q\|_{C^0}\le \|\mathcal{R}w_{q+1}\|_{C^\alpha}\lesssim \delta_{q+1}^{\sfrac12}\lambda_{q+1}^{-1+\alpha}= \lambda_q^{-b(1+\beta)+ b\alpha}
\end{equation*}
for some implicit constant depending only on the choice of $\alpha>0$. Next, observe that 
\begin{equation*}
-b(1+\beta)+b\alpha=-\theta-b\frac{b-1}{b+1}(1+\beta)+b\alpha,	
\end{equation*}
so that, by choosing $\alpha>0$ sufficiently small we deduce
\begin{equation}\label{e:zreg}
    \|z_{q+1}-z_q\|_{C^0}\lesssim \kappa_q\lambda_q^{-2\gamma},
\end{equation}
where we also used \eqref{e:gluingestimate} for estimating $\bar{z}_{q}-z_q$. 
In particular for any $q\ge q_0$
\begin{equation*}
\|z-z_{q}\|_{C^0}\leq \sum_{q'\ge q}\|z_{q+1}-z_{q}\|_{C^0}\lesssim \kappa_{q}\lambda_{q}^{-2\gamma}.
\end{equation*}
In particular, by choosing $a\gg a_0$ (the base of the double-exponential defining $(\lambda_q,\delta_q)$ in \eqref{e:lambdadelta}) in step 1 sufficiently large, we may conclude the estimate 
\begin{equation}\label{e:mol_clos1}
\|z-z_{q}\|_{C^0}\leq  \kappa_{q}\lambda_{q}^{-\tfrac32 \gamma}\quad\textrm{ for any }q\ge q_0.
\end{equation} 
Consequently, for any $q\geq q_0$ we may apply Proposition \ref{p:stability_in_ellipticity} to \eqref{e:mol_clos1} to obtain
\begin{equation*}
	| D_{u,q} - D_q | \le \lambda_q^{-\gamma} D_q\,,
\end{equation*}
where $ D_{u,q} :=  \kappa_q \int_0^T \| \nabla \tilde \rho_{u,q} \|^2_{L^2} dt$ and $D_q  :=  \kappa_q \int_0^T \| \nabla \tilde \rho_q \|^2_{L^2} dt$. In combination with the bound \eqref{e:mainproof:e10} we then have, for any $q\geq q_I$:
 \begin{equation}\label{e:mainproof-step6}
 \kappa_{q} \int_0^T \| \nabla \tilde\rho_{u,q} \|^2_{L^2} dt \geq  \frac{3}{2}c_0  \|\rho_{in}\|_{L^2}^2\,.
 \end{equation}

 \smallskip

\noindent\emph{Step 7. Anomalous dissipation for original initial datum.} Finally, we compare $\tilde\rho_{u,q}$ to the solution of \eqref{e:advectiondiffusion} with the original initial datum $\rho_{in}$, along the sequence $\kappa_q$. In other words, we consider
\begin{equation}  \label{e:mainproof:finaleq}
\begin{split}
    \partial_t \rho_{u,q} + u \cdot \nabla \rho_{u,q} =& \kappa_q \Delta \rho_{u,q}, \\ 
    {\rho_{u,q}}_{|{t=0}} =& \rho_{in}.
\end{split}
\end{equation}
The basic energy estimate together with \eqref{e:comparerhotilde} and \eqref{e:conditionq1theta} gives
\begin{equation*}
\begin{split}
\kappa_q \int_0^T\|\nabla\rho_{u,q}-\nabla\tilde\rho_{u,q}\|_{L^2}^2\,dt&\leq \frac{1}{2}\|\rho_{in}-\tilde\rho_{in}\|^2_{L^2}\\
&\leq  \frac{1}{2}C^2_{\tilde N}  \ell_{in}^{-2}  \lambda^{-2}_{q_I} \|\rho_{in}\|_{L^2}^2\\
&\leq  \frac{1}{200}(1-\exp(-2\ell_0^{-2}\kappa_{q_I}T))\|\rho_{in}\|_{L^2}^2=\frac{1}{20}c_0\|\rho_{in}\|_{L^2}^2\,.
\end{split}
\end{equation*}
Using the Minkowski inequality we deduce, using \eqref{e:mainproof-step6}, for any $q\geq q_I$,
\begin{equation}\label{e:mppen}
\begin{split}
\left(\kappa_q \int_0^T\|\nabla\rho_{u,q}\|_{L^2}^2\,dt\right)^{\sfrac12}&\geq\left(  \kappa_q \int_0^T\|\nabla\tilde\rho_{u,q}\|_{L^2}^2\,dt\right)^{\sfrac12}-  \left(\kappa_q \int_0^T\|\nabla\rho_{u,q}-\nabla\tilde\rho_{u,q}\|_{L^2}^2\,dt \right)^{\sfrac{1}{2}}\\
    &\geq \left(\left(\tfrac{3}{2}\right)^{\sfrac12}-  \left(\tfrac{1}{20} \right)^{\sfrac{1}{2}} \right) c_0^{\sfrac12}\|\rho_{in}\|_{L^2}\geq c_0^{\sfrac12}\|\rho_{in}\|_{L^2}.
\end{split}
\end{equation}

We conclude
\begin{equation*}
\lim_{q\to\infty} \kappa_q \int_0^T\|\nabla\rho_{u,q}\|_{L^2}^2\,dt\geq c_0\|\rho_{in}\|_{L^2}^2	
\end{equation*}
and in particular \eqref{e:anomalousdissipation_gen}, with constant $c_0$ given in \eqref{e:mainproof-c0}.

 \smallskip
 
\noindent\emph{Step 8. The constant $c_0$ of the dissipation rate.} It remains to verify the expressions for the lower bound of $c_0$ in terms of $\ell_0$, $\ell_{in}$ and $T$ given in \eqref{e:main_const}. First of all recall that $q_I\geq q_0$ is chosen according to \eqref{e:conditionq1theta}. Then, observe that, as a function of $T>0$, the left-hand side of \eqref{e:anomalousdissipation_gen} is monotone increasing, so we may assume the same for the lower bound of $c_0$. 

Next, let us fix $c_1=c_1(\tilde N)\geq 1$ so that $f(c_1)\geq (10C_{\tilde N})^2$, where $C_{\tilde N}$ is the constant from \eqref{e:moll_gen} and $f$ defined in \eqref{e:functionf}. Set 
\begin{equation*}
T_*=\ell_0^2\ell_{in}^{-b\theta}\min\left\{1,(c_1\ell_{in}^{-1}\lambda_{q_0}^{-1})^{2-b\theta}\right\}\,.    
\end{equation*}
\noindent\emph{Small time $T$.} If $T<T_*$, we choose $q_I\geq q_0$ to be the smallest integer such that
\begin{equation}\label{e:choiceofqI}
\lambda_{q_I}^{2-b\theta}\geq c_2 \ell_0^2\ell_{in}^{-2}T^{-1},
\end{equation}
where $c_2:=c_1^{2-b\theta}$. Then 
\begin{equation*}
    \kappa_{q_I+1}\ell_0^{-2}T\leq \kappa_{q_I+1}\ell_0^{-2}T_*=\lambda_{q_I}^{-b\theta}\ell_0^{-2}T_*\leq (\lambda_{q_I}\ell_{in})^{-b\theta}\leq c_1^{-b\theta}\,.
\end{equation*}
Since the function $s\mapsto 1-e^{-2s}$ is concave, we obtain
\begin{equation*}
\begin{split}
1-\exp(-2\kappa_{q_I+1}\ell_0^{-2}T)&\geq  \ell_0^{-2}\kappa_{q_I+1}Tc_1^{b\theta}(1-\exp(-2c_1^{-b\theta}))\\
&=\ell_0^{-2}\lambda_{q_I}^{-b\theta}c_2^{-1}f(c_1)T\\
&\geq (10C_{\tilde N})^2\ell_0^{-2}\lambda_{q_I}^{-b\theta}c_2^{-1}T
\end{split}
\end{equation*}
Therefore, using our choice of $q_I$ we deduce
\begin{equation}\label{e:smalldatasmalltime}
(10C_{\tilde N})	^2\ell_{in}^{-2}\leq  (10C_{\tilde N})	^2 c_2^{-1}\ell_0^{-2}T\lambda_{q_I}^{2-b\theta}\leq \lambda_{q_I}^2(1-\exp(-2\kappa_{q_I+1}\ell_0^{-2}T)),
\end{equation}
so that \eqref{e:conditionq1theta} is satisfied. In turn, we obtain the lower bound
\begin{equation}\label{e:lowerboundc0-1}
    c_0\geq 10C_{\tilde N}^2\ell_{in}^{-2}\lambda_{q_I}^{-2}.
\end{equation}
It remains to estimate $\lambda_{q_I}$. To this end observe that the choice of $T_*$ implies for $T< T_*$
\begin{equation*}
\lambda_{q_0}^{2-b\theta}<c_2\ell_0^2\ell_{in}^{-2}T^{-1},
\end{equation*}
so that, on the one hand $q_I>q_0$, and on the other hand, since $\lambda_{q+1}=\lambda_q^b$ and $q_I$ is chosen to be the smallest with property \eqref{e:choiceofqI},
\begin{equation*}
\lambda_{q_I}^{2-b\theta}<c_2^b\ell_0^{2b}\ell_{in}^{-2b}T^{-b}.
\end{equation*}
Applying this to \eqref{e:lowerboundc0-1} we obtain
\begin{equation}\label{e:lowerboundc0-2}
c_0\geq C(\tilde N)\ell_{in}^{-2}\left(\frac{\ell_{in}}{\ell_0}\right)^{\frac{4b}{2-b\theta}}T^{\frac{2b}{2-b\theta}}\,.
\end{equation}
\noindent\emph{Large time $T$.} If $T\geq T_*$, we may simply take as lower bound the value of the right hand side in \eqref{e:lowerboundc0-2} at $T=T_*$. This yields
\begin{equation}\label{e:lowerboundc0-3}
c_0\geq C(\tilde N)\ell_{in}^{-2}\min\{ \lambda_{q_0}^{-2b},(c_1\ell_{in}^{-1})^{-2b}\}.
\end{equation}
Combining \eqref{e:lowerboundc0-2} and \eqref{e:lowerboundc0-3} then yields \eqref{e:main_const} with $\delta_0:=\lambda_{q_0}^{-2b}$. This completes the proof of Theorem \ref{t:main_general}.

\section{Energy estimates}\label{s:energy} 
This section presents energy estimates, used recurrently in the paper. For short notation, in this section we denote sometimes the $L^2$ (space-time norm) on $ \T^3 \times [0,T]$ by $L^2_{xt}$. We use $\boldsymbol{\alpha}$ to denote the multiindex in a partial derivative and we use $|\boldsymbol{\alpha}|$ to denote its total order of differentiation. Notation $a \lesssim_n b$ means that there exists constant $C_n$ (depending only on $n$) such that $a \le C_n b$. By $D_t$ we denote the advective derivative, i.e. $D_t = \partial_t+u\cdot\nabla$.

\subsection{Estimates for advection-diffusion with Laplacian}
We consider the advection-diffusion equation
\begin{equation}\label{e:eqE1}
\begin{aligned}
    \partial_t\rho+u\cdot\nabla \rho&=\kappa\Delta\rho \qquad \textrm{ on }\T^3 \times [0, T],\\
    \rho|_{t=0}&=\rho_{in},
\end{aligned}
\end{equation}
with smooth initial datum. 

We assume the scales relationship
\begin{equation}\label{e:scales_u1}
  \left(\frac{\|\nabla u\|_{L^\infty}}{\kappa}\right)^n \ge \left( \frac{\|\nabla^{n} u\|_{L^\infty}}{ \kappa}  \right)^\frac{2n}{n+1}. 
\end{equation}

\begin{lemma}\label{l:energy_lapl}
Let $\kappa>0$ and $u\in C^{\infty}(\T^d;\R^d)$ be divergence free. Assume \eqref{e:scales_u1}. Then the advection-diffusion equation \eqref{e:eqE1} satisfies for any ${t \le T}$ and for any $n\geq 1$
\begin{equation}\label{e:ener_un}
\|\nabla^n \rho (t)\|^2_{L^2}+ \kappa\int_0^t \|\nabla^{n+1} \rho (s)\|_{L^2}^2 ds \le  \|\nabla^n \rho_{in}\|^2_{L^2} + C_n    \left(\frac{\|\nabla u\|_{L^\infty}}{\kappa}\right)^n  \kappa \int_0^T \|\nabla \rho (s)\|_{L^2}^2 ds.
\end{equation}
\end{lemma}
\begin{proof} 
Apply $\pal$, where $|\boldsymbol{\alpha}|=n$, to the equation and and test the result with $\partial^{\boldsymbol{\alpha}}$ to get after integration in time that for any ${t \le T}$
\[
\frac{1}{2}  \|\pal \rho (t) \|^2_{L^2}+ \kappa \int_0^t \|\nabla \pal \rho  (s)\|_{L^2}^2 ds  \le  \frac{1}{2} \|\nabla^n \rho_{in}\|^2_{L^2} + \| [u \cdot \nabla, \pal]  \rho \|_{L_{xt}^2} \|\pal \rho\|_{L_{xt}^2}.
  \]
For the last term use the commutator estimate
$$
\| [u \cdot \nabla, \pal] f  \|_{L^2} \lesssim_n \|  \nabla^n u  \|_{L^\infty}   \|  \nabla f \|_{L^2}  +  \|  \nabla u  \|_{L^\infty}   \|  \nabla^n f \|_{L^2}
$$
followed by H\"older's inequality in time to write
\[
\begin{split}
 \| [u \cdot \nabla, \pal]  \rho \|_{L_{xt}^2} \|\pal \rho\|_{L_{xt}^2}
 &\lesssim_n \|  \nabla^n u  \|_{L^\infty}   \|  \nabla \rho \|_{L_{xt}^2}  \|  \nabla^n \rho \|_{L_{xt}^2}  +  \|  \nabla u  \|_{L^\infty}   \|  \nabla^n \rho \|^2_{L_{xt}^2}\\ &\lesssim_n \|  \nabla^n u  \|_{L^\infty}   \| \nabla \rho  \|^\frac{n+1}{n}_{L_{xt}^2} \| \nabla^{n+1} \rho  \|^\frac{n-1}{n}_{L_{xt}^2}  +  \|  \nabla u  \|_{L^\infty} 
 \| \nabla \rho  \|^\frac{2}{n}_{L_{xt}^2} \| \nabla^{n+1} \rho  \|^\frac{2n-2}{n}_{L_{xt}^2},
\end{split}
\]
where for the latter $\lesssim$ we use interpolation. Thus, summing over all partial derivatives of order $n$ we have
\[
\begin{split}
 \frac{1}{2}  \|\nabla^n \rho(t) \|^2_{L^2}+ \kappa  \int_0^t &\|\nabla^{n+1} \rho  (s)\|_{L^2}^2 ds  \le C_n \kappa^{-1} \|  \nabla^n u  \|_{L^\infty}   (\kappa^\frac{1}{2}  \| \nabla \rho  \|_{L_{xt}^2})^\frac{n+1}{n} (\kappa^\frac{1}{2} \| \nabla^{n+1} \rho  \|_{L_{xt}^2})^\frac{n-1}{n}  \\
 + C_n  & \kappa^{-1}  \|  \nabla u  \|_{L^\infty}
 (\kappa^\frac{1}{2}  \| \nabla \rho  \|_{L_{xt}^2})^\frac{2}{n} (\kappa^\frac{1}{2}  \| \nabla^{n+1} \rho  \|_{L_{xt}^2})^\frac{2n-2}{n} + \frac{1}{2} \|\nabla^n \rho_{in}\|^2_{L^2}. 
\end{split}
\]
Hence, using Young's inequality we have
\[
 \|\nabla^n \rho(t) \|^2_{L^2}+ \kappa  \int_0^t \|\nabla^{n+1} \rho  (s)\|_{L^2}^2 ds  \le  \|\nabla^n \rho_{in}\|^2_{L^2} +
C_n    \left(  \left(\frac{\|\nabla u\|_{L^\infty}}{\kappa}\right)^n + \left( \frac{\|\nabla^{n} u\|_{L^\infty}}{ \kappa}  \right)^\frac{2n}{n+1}   \right)  \kappa \|\nabla \rho\|_{L_{xt}^2}^2,
 \]
which with the assumption \eqref{e:scales_u1} gives \eqref{e:ener_un}.
\end{proof}

We will need also the following immediate corollary for the forced advection-diffusion equation
\begin{equation}\label{e:eqE1f}
\begin{aligned}
    \partial_t\rho+u\cdot\nabla \rho&=\kappa\Delta\rho  +\div f \qquad \textrm{ on }\T^3 \times [0, T],\\
    \rho|_{t=0}&=\rho_{in},
\end{aligned}
\end{equation}
with smooth initial datum and forcing.

\begin{corollary}\label{c:energy_lapl}
Let $\kappa>0$ and $u\in C^{\infty}(\T^d;\R^d)$ be divergence free. Assume \eqref{e:scales_u1}. Then the forced advection-diffusion equation \eqref{e:eqE1f} satisfies for any ${t \le T}$ and for any $n\geq 1$
\begin{equation}\label{e:ener_un_f}
\begin{split}
 &\|\nabla^n \rho(t)\|^2_{L^2}+ \kappa\int_0^t \|\nabla^{n+1} \rho (s)\|_{L^2}^2 ds \le \\ &\|\nabla^n \rho_{in}\|^2_{L^2} + C_n    \left(\frac{\|\nabla u\|_{L^\infty}}{\kappa}\right)^n  \kappa \int_0^T \|\nabla \rho (s)\|_{L^2}^2 ds + \sum_{|\boldsymbol{\alpha}|=n}\left|\int_0^T \int \pal f (s) \cdot \nabla \pal \rho  (s) ds  \right|.
\end{split}
 \end{equation}
\end{corollary}

\subsection{Estimates for advection-diffusion with an elliptic matrix} Now we consider the following, more general advection-diffusion equation on $\T^3 \times [0, T]$, with smooth initial datum 
\[
 \partial_t\rho+u\cdot\nabla\rho=\div A\nabla\rho, 
\]
where $A$ is a smooth elliptic matrix.

\subsubsection{Comparison}
First, we prove an estimate that allows to compare two solutions whose ellipticity matrices differ slightly from each other.
\begin{proposition}[Stability estimates]    \label{p:stability_in_ellipticity}
Let $\varrho_1$ and $\varrho_2$ solve the following equations on $\T^3 \times [0, T]$
\begin{align*}
    \partial_t \varrho_1 + u \cdot \nabla \varrho_1 &= \div ( A_1 \nabla \varrho_1 ), \\ 
    \partial_t \varrho_2 + u \cdot \nabla \varrho_2 &= \div ( A_2 \nabla \varrho_2 ),
\end{align*}
with initial data $\varrho_1(0) = \varrho_2(0) = \rho_{\text{in}} \in L^2(\T^3)$ and uniformly elliptic symmetric matrices $A_1, A_2: \T^3 \times [0,T] \rightarrow \R^{3 \times 3}$ satisfying for $\varepsilon \le \frac{1}{10}$
\begin{align} \label{p:stability_in_ellipticity_c0}
    \big| (A_1 - A_2) \xi \cdot \zeta \big|
    \leq& \varepsilon (A_1 \xi \cdot \xi)^{\frac{1}{2}} (A_1 \zeta \cdot \zeta)^{\frac{1}{2}},
    \quad \text{for any } (x,t) \in \T^3 \times [0,T] \text{ and } \xi,\zeta \in \R^3.
\end{align}
Let $\tilde \varrho := \varrho_1 - \varrho_2$, then we have
\begin{align}
    \|\varrho_1(t) - \varrho_2(t)\|_{L^2}^2 &\le 2 \varepsilon^2      \left| \|\varrho_{in}\|_{L^2}^2 - \|\varrho_2(t)\|_{L^2}^2 \right|,
    \label{p:stability_in_ellipticity_e2d} \\
    \left| \|\varrho_1(t)\|_{L^2}^2 - \|\varrho_2(t)\|_{L^2}^2 \right|
    &\le 10 \varepsilon      \left| \|\varrho_{in}\|_{L^2}^2 - \|\varrho_2(t)\|_{L^2}^2 \right|\label{p:stability_in_ellipticity_e2} 
\end{align}
for any $t \le T$.
\end{proposition}
\begin{proof} Observe that from \eqref{p:stability_in_ellipticity_c0}, ellipticity, and $\varepsilon \le \frac{1}{10}$ one has the pointwise inequality 
\begin{equation}
 | A_1 \nabla \varrho_2 \cdot \nabla \varrho_2 |  \leq  \tfrac{10}{9} A_2 \nabla \varrho_2 \cdot \nabla \varrho_2.
\label{p:stability_in_ellipticity_e0}
\end{equation}
Take the difference of the two equations
\begin{align}   \label{p:stability_in_ellipticity_e5}
    \partial_t \tilde \varrho + u \cdot \nabla \tilde \varrho &= \div \big( A_1 \nabla \tilde \varrho \big) + \div \big( (A_1-A_2) \nabla \varrho_2 \big)
\end{align}
test it with $\tilde \varrho$ and integrate by parts. For the resulting last term $\int (A_1-A_2) \nabla \varrho_2  \cdot \nabla \tilde \varrho$ use the assumption \eqref{p:stability_in_ellipticity_c0} and Young's inequality to absorb the resulting term containing $A_1 \nabla \tilde \varrho$. This gives for any $t \le T$
\begin{equation}
 \| \tilde \varrho(t) \|^2_{L^2}  + \int_0^t \int A_1 \nabla \tilde \varrho \cdot \nabla \tilde \varrho
   \leq \varepsilon^2  \int_0^t \int A_1 \nabla \varrho_2 \cdot \nabla \varrho_2.  \label{p:stability_in_ellipticity_e1}
\end{equation}
The inequality \eqref{p:stability_in_ellipticity_e2d} follows directly from \eqref{p:stability_in_ellipticity_e1} and the energy identity for $\varrho_2$.

From the equations for $\varrho_1$ and $\varrho_2$, we can also derive
\begin{equation}    \label{p:stability_in_ellipticity_e6}
    \partial_t \big( \tilde \varrho \varrho_2 \big)
    = -\tilde \varrho u \cdot \nabla \varrho_2 + \tilde \varrho \div ( A_2 \nabla \varrho_2 ) \\ 
     -\varrho_2 u \cdot \nabla \tilde \varrho + \varrho_2  \div \big( A_1 \nabla \tilde \varrho \big) + \varrho_2 \div \big( (A_1-A_2) \nabla \varrho_2 \big).
\end{equation}
Integrate \eqref{p:stability_in_ellipticity_e6} over $\T^3 \times [0,t]$ and integrate by parts in space. The first and the third terms on the right hand side of \eqref{p:stability_in_ellipticity_e6} cancel and we obtain
\begin{equation}\label{p:stability_in_ellipticity_e7}
\begin{split}
 -\int& \tilde \varrho(t) \varrho_2(t) dx
    = \int_0^t \int (A_1 + A_2) \nabla\tilde \varrho \cdot \nabla \varrho_2+  \int_0^t \int  (A_1-A_2) \nabla \varrho_2   \cdot \nabla\varrho_2 \\ 
 &=  2\int_0^t \int A_1 \nabla\tilde \varrho  \cdot \nabla \varrho_2 + \int_0^t \int  (A_2-A_1) \nabla \tilde\varrho  \cdot \nabla \varrho_2 + \int_0^t \int  (A_1-A_2) \nabla \varrho_2  \cdot \nabla \varrho_2.
\end{split}
\end{equation}
For the first right-hand side term of the second line of \eqref{p:stability_in_ellipticity_e7} we use $|A_1  \xi \cdot \zeta | \le (A_1 \xi \cdot \xi)^{\frac{1}{2}} (A_1 \zeta \cdot \zeta)^{\frac{1}{2}}$ (Cauchy-Schwarz) and Young's inequality that utilizes $\varepsilon^2$ of \eqref{p:stability_in_ellipticity_e1}, to get 
\[
 2 \left| \int_0^t \int A_1 \nabla\tilde \varrho  \cdot \nabla \varrho_2 \right| \le 2 \varepsilon \int_0^t \int A_1 \nabla \varrho_2  \cdot \nabla \varrho_2. 
\]
For the second right-hand side term of \eqref{p:stability_in_ellipticity_e7} we use the assumption \eqref{p:stability_in_ellipticity_c0}
to write
\[
 \left| \int_0^t \int  (A_2-A_1) \nabla \tilde\varrho  \cdot \nabla \varrho_2  \right| \le \frac12 \int_0^t \int A_1 \nabla \tilde \varrho  \cdot \nabla \tilde \varrho  + \frac12 \varepsilon^2 \int_0^t \int A_1 \nabla \varrho_2  \cdot \nabla \varrho_2 \le \varepsilon^2 \int_0^t \int A_1 \nabla \varrho_2  \cdot \nabla \varrho_2 ,
\]
where we have used \eqref{p:stability_in_ellipticity_e1}.
Together with an analogous estimate for the last term of \eqref{p:stability_in_ellipticity_e7} we obtain
\[
   \left| \int \tilde \varrho(t) \varrho_2(t) dx \right| \le 4\varepsilon \int_0^t \int  A_1 \nabla \varrho_2 \cdot \nabla \varrho_2.
\]
Finally we compute, using the above and \eqref{p:stability_in_ellipticity_e1}
\begin{align*}
    \left| \|\varrho_1(t)\|_{L^2}^2 - \|\varrho_2(t)\|_{L^2}^2 \right| = \left| \int \tilde \varrho(t) \big( \varrho_1(t) + \varrho_2(t) \big) dx \right| \leq \| \tilde \varrho(t) \|^2_{L^2} + 2\left|\int \tilde \varrho(t) \varrho_2(t) dx\right| \\
    \leq \varepsilon^2  \int_0^t \int A_1 \nabla \varrho_2 \cdot \nabla \varrho_2+ 8 \varepsilon \int_0^t \int  A_1 \nabla \varrho_2 \cdot \nabla \varrho_2.
\end{align*}
Now the proof of 
\eqref{p:stability_in_ellipticity_e2} follows from \eqref{p:stability_in_ellipticity_e0} and the energy identity for equation for $\varrho_2$.
\end{proof}

\subsubsection{General weighted estimate} Estimates in this section are needed mainly in the space homogenization proposition, therefore we use the notation $\bar\rho, \bar A$.
We consider
\begin{equation}\label{e:e_homogenized}
\begin{aligned}
    \partial_t\bar\rho+u\cdot\nabla\bar\rho&=\div \bar A\nabla\bar\rho  \qquad \textrm{ on }\T^3 \times [0, T],\\
   \bar\rho|_{t=0}&=\rho_{in}.
\end{aligned}
\end{equation}
Let us define
\begin{equation}\label{e:e_D}\bar D = \int_0^T\|\bar \kappa^\frac{1}{2} \nabla \bar\rho (s)\|_{L^2}^2\,ds.
\end{equation}
Assume the following inequalities
\begin{equation}\label{e:e_ell}
2 \bar\kappa(x,t) \Id \geq \bar{A} (x,t) \geq \frac{\bar\kappa (x,t)}{2}  \Id
\end{equation}
and
\begin{equation}\label{e:e_kappatransported}
\left\|\frac{D_t \bar\kappa }{\bar\kappa} \right\|_{L^\infty} \leq C \tau^{-1}.
\end{equation}
Assume further that for $m \ge 1$
\begin{equation}\label{e:e_modelAestimates}
    \|\bar \kappa^\frac{m-2}{2} \nabla^m\bar A\|_{L^\infty} \leq C  (\tau^{-1})^\frac{m}{2} \qquad  \|\bar \kappa^\frac{m-2}{2} \nabla^m\bar \kappa\|_{L^\infty} \le C (\tau^{-1})^\frac{m}{2}
\end{equation}
and 
\begin{equation}\label{e:e_model-estimates2}
    \| \bar\kappa^\frac{m-1}{2}  \nabla^m u\|_{L^\infty} \leq C \tau^{-1} (\tau^{-1})^\frac{m-1}{2}.
\end{equation}

\begin{lemma} \label{l:energy_weighted}
Let $u$ be divergence free. Assume the inequalities \eqref{e:e_ell}, \eqref{e:e_kappatransported}, \eqref{e:e_model-estimates2}, \eqref{e:e_modelAestimates}
 hold. Then the general advection-diffusion equation \eqref{e:e_homogenized} satisfies for any ${t \le T}$ and for any $n \ge 1$ 
\begin{equation}\label{e:rhobarenerg}
\sup_{t\le T}  \|(\bar\kappa^\frac{n}{2} \nabla^n\bar\rho)(t)\|^2_{L^2}+\int_0^T \|\bar\kappa^\frac{n+1}{2}\nabla^{n+1}\bar\rho\|_{L^2}^2 \lesssim (\tau^{-1})^n \bar D + \sum_{i=1}^n (\tau^{-1})^{n-i} \|(\bar\kappa^\frac{i}{2} \nabla^i\bar\rho)_{in}\|^2_{L^2}.
\end{equation}
Furthermore, for any $|\boldsymbol{\alpha}|=n$
\begin{equation}\label{e:rhobarflowB}
\int_0^T \|\bar\kappa^\frac{n}{2} D_t \pal \bar\rho\|^2_{L^2} \lesssim  (\tau^{-1})^{n+1} \bar D  + \sum_{i=1}^{n+1}  \left(\tau^{-1} \right)^{{n+1-i}} \|(\bar\kappa^\frac{i}{2} \nabla^i\bar\rho)_{in}\|^2_{L^2}.
\end{equation}
The constants in $\lesssim$ depend on the constants in assumptions and $n$.
\end{lemma}

\begin{proof}
\noindent\emph{Step 1: Preliminary $n$-th order estimate}. Apply $\pal$, where $|\boldsymbol{\alpha}|=n$, to the equation \eqref{e:e_homogenized}  and test the result with $\bar\kappa^n  \partial^{\boldsymbol{\alpha}} \bar\rho$ to get
\begin{equation}\label{e:rhobarenergym2_m1}
\begin{split}
\frac{1}{2} \frac{d}{dt} &\int |\bar\kappa^\frac{n}{2} \pal\bar\rho|^2+ \int \bar\kappa^n \bar A \nabla \pal \bar\rho \cdot \nabla \pal \bar\rho  =\frac{1}{2} \int \frac{D_t (\bar\kappa^n) }{\bar\kappa^n}  \bar\kappa^n |\pal \bar\rho|^2  \\ 
&  -\sum_{{\boldsymbol{\beta}} + {\boldsymbol{\gamma}} ={\boldsymbol{\alpha}} , {\boldsymbol{\beta}} >0} c_{\boldsymbol{\beta, \gamma}} \int \partial^{\boldsymbol{\beta}} u \cdot \partial^{\boldsymbol{\gamma}}  \nabla \bar\rho (\bar\kappa^n \pal \bar\rho)  + \partial^{\boldsymbol{\beta}}  \bar A   \nabla \partial^{\boldsymbol{\gamma}}  \bar\rho \cdot \nabla (\bar\kappa^n \pal \bar\rho ) + \int \bar A   \nabla \pal \bar\rho  \cdot \nabla (\bar\kappa^n )  \pal \bar\rho,
\end{split}
\end{equation}
where $c_{\boldsymbol{\beta, \gamma}}$ are binomial coefficients.
Estimate four right-hand side terms of \eqref{e:rhobarenergym2_m1} in order of their appearance. 
For the first one, use $\frac{D_t (\bar\kappa^n) }{\bar\kappa^n} = n \frac{D_t \bar\kappa}{\bar\kappa}$ to have
\[
 \int \frac{D_t (\bar\kappa^n) }{\bar\kappa^n}  \bar\kappa^n |\pal \bar\rho|^2  
 \le n \|\bar\kappa^\frac{n}{2}\nabla^{n}\bar\rho\|_{L^2}^2 \left\|\frac{D^u_t \bar\kappa }{\bar\kappa} \right\|_{L^\infty}.
\]
For a single summand of the second one we have
\begin{equation}
\begin{split}
\int \partial^{\boldsymbol{\beta}} u \cdot \partial^{\boldsymbol{\gamma}}  \nabla \bar\rho (\bar\kappa^n \pal \bar\rho)  &= \int  \bar\kappa^{\frac{n-|\boldsymbol{\gamma}|-1}{2} } \partial^{\boldsymbol{\beta}} u \cdot (\bar\kappa^\frac{|\boldsymbol{\gamma}|+1}{2} \partial^{\boldsymbol{\gamma}} \nabla \bar\rho) ( \bar\kappa^\frac{n}{2} \pal \bar\rho)\\ & \le   \|\bar\kappa^\frac{n}{2}\nabla^{n}\bar\rho\|_{L^2} 
\|\bar\kappa^\frac{|\boldsymbol{\gamma}|+1}{2}\nabla^{|\boldsymbol{\gamma}|+1}\bar\rho\|_{L^2}   \left\| \bar\kappa^{\frac{n-|\boldsymbol{\gamma}|-1}{2} } \nabla^{n-|\boldsymbol{\gamma}|} u \right\|_{L^\infty}.    \end{split}
\end{equation}
For a single summand of the third one we write
\begin{equation}\label{e:rhobarenergym2_2short3}
\begin{split}
\int\partial^{\boldsymbol{\beta}} \bar A   \nabla \partial^{\boldsymbol{\gamma}} \bar\rho \cdot & \nabla (\bar\kappa^n \pal \bar\rho )
 =  n \int\partial^{\boldsymbol{\beta}} \bar A   (\nabla \partial^{\boldsymbol{\gamma}} \bar\rho)  \bar\kappa^{n-1} \cdot \nabla \bar\kappa \pal \bar\rho  + \int\partial^{\boldsymbol{\beta}} \bar A   (\nabla \partial^{\boldsymbol{\gamma}} \bar\rho) \bar\kappa^n \cdot \nabla  \pal \bar\rho \le 
 \\
    & n  \|\bar\kappa^{\frac{n-|\boldsymbol{\gamma}|-2}{2} } \nabla^{n-|\boldsymbol{\gamma}|} \bar A \|_{L^\infty}  \|\bar\kappa^{-\frac{1}{2}} \nabla \bar\kappa\|_{L^\infty}  \| \bar\kappa^\frac{|\boldsymbol{\gamma}|+1}{2}\nabla^{|\boldsymbol{\gamma}|+1} \bar\rho\|_{L^2}  \|\bar\kappa^{\frac{n}{2}} \nabla^n \bar\rho \|_{L^2}  \\
 & +  \|\bar\kappa^{\frac{n-|\boldsymbol{\gamma}|-2}{2} } \nabla^{n-|\boldsymbol{\gamma}|} \bar A \|_{L^\infty}  \|\bar\kappa^\frac{|\boldsymbol{\gamma}|+1}{2}\nabla^{|\boldsymbol{\gamma}|+1} \bar\rho\|_{L^2}  \|\bar\kappa^\frac{n+1}{2} \nabla^{n+1} \bar\rho\|_{L^2}.
 \end{split}
\end{equation}
For the fourth, last term of  \eqref{e:rhobarenergym2_m1} we use the upper bound \eqref{e:e_ell} and the estimate
\begin{equation*}
\int   \bar\kappa^{n-1} \bar A   \nabla \bar\kappa   \pal \bar\rho \cdot \nabla \pal  \bar\rho \le 2\int  |\bar\kappa|^{n} |\nabla \bar\kappa | | \pal \bar\rho |\nabla \pal  \bar\rho| \le 2 \|\bar\kappa^{-\frac{1}{2}} \nabla \bar \kappa \|_{L^\infty}  \|\bar\kappa^\frac{n}{2}\nabla^{n} \bar\rho\|_{L^2}  \|\bar\kappa^\frac{n+1}{2} \nabla^{n+1} \bar\rho\|_{L^2}. 
\end{equation*}
Together, these estimates for four right-hand side terms of  \eqref{e:rhobarenergym2_m1}, absorbing  their terms $ \|\bar\kappa^\frac{n+1}{2} \nabla^{n+1} \bar\rho\|_{L^2} $  by the dissipative part using Young's inequality, after summing over all multiindices of order $n$ yield
\begin{equation}\label{e:rhobarenergym2_m2}
\begin{split}
\frac{1}{2} \frac{d}{dt} \|(\bar\kappa^\frac{n}{2} \nabla^n\bar\rho)(t)\|^2_{L^2}+ &\frac{1}{2} \|\bar\kappa^\frac{n+1}{2}\nabla^{n+1}\bar\rho\|_{L^2}^2 \lesssim_n   \|\bar\kappa^\frac{n}{2}\nabla^{n}\bar\rho\|_{L^2}^2 \left\|\frac{D_t \bar\kappa }{\bar\kappa} \right\|_{L^\infty} \\ 
& + \|\bar\kappa^\frac{n}{2}\nabla^{n}\bar\rho\|_{L^2} \sum_{j=0}^{n-1}  \|\bar\kappa^\frac{j+1}{2}\nabla^{j+1}\bar\rho\|_{L^2}   \left\| \bar\kappa^{\frac{n-j-1}{2} } \nabla^{n-j} u \right\|_{L^\infty}  \\
 & + \|\bar\kappa^{-\frac{1}{2}} \nabla \bar\kappa\|_{L^\infty}  \|\bar\kappa^\frac{n}{2} \nabla^n \bar\rho \|_{L^2}  \sum_{j=0}^{n-1}  \|\bar\kappa^{\frac{n-j-2}{2} } \nabla^{n-j} \bar A\|_{L^\infty}  \| \bar\kappa^\frac{j+1}{2}\nabla^{j+1} \bar\rho\|_{L^2}   \\
 &+   \sum_{j=0}^{n-1} \|\bar\kappa^{\frac{n-j-2}{2} } \nabla^{n-j} \bar A\|^2_{L^\infty}   \|\bar\kappa^\frac{j+1}{2}\nabla^{j+1} \bar\rho \|^2_{L^2} + \|\bar\kappa^{-\frac{1}{2}} \nabla \bar \kappa \|^2_{L^\infty}  \|\bar\kappa^\frac{n}{2}\nabla^{n} \bar\rho\|^2_{L^2}. 
\end{split}
\end{equation}

\noindent\emph{Step 2: Plugging in scales assumptions.} Use the assumptions \eqref{e:e_kappatransported}, \eqref{e:e_modelAestimates}, \eqref{e:e_model-estimates2} for right-hand side of  \eqref{e:rhobarenergym2_m2} to obtain
\begin{equation}\label{e:rhobarenergym2_m3}
\frac{d}{dt} \|(\bar\kappa^\frac{n}{2} \nabla^n\bar\rho)(t)\|^2_{L^2}+ \|\bar\kappa^\frac{n+1}{2}\nabla^{n+1}\bar\rho\|_{L^2}^2 \lesssim_n   \sum_{j=0}^{n-1}  \|\bar\kappa^\frac{j+1}{2}\nabla^{j+1}\bar\rho\|^2_{L^2}   (\tau^{-1})^{n-j},
\end{equation}
which after integrating in time and writing $P(i) =\int_0^T \|\bar\kappa^\frac{i}{2}\nabla^{i}\bar\rho\|_{L^2}^2$ yields for any $n \ge 1$
\begin{equation}\label{e:rhobarenergym2_m4}
\sup_{t\le T}  \|(\bar\kappa^\frac{n}{2} \nabla^n\bar\rho)(t)\|^2_{L^2} + P(n+1) \lesssim_n  \|(\bar\kappa^\frac{n}{2} \nabla^n\bar\rho)_{in}\|^2_{L^2} +  
 \sum_{j=0}^{n-1} P(j+1) (\tau^{-1})^{n-j}.
\end{equation}

\noindent\emph{Step 3: Inductive proof of \eqref{e:rhobarenerg}.} Take \eqref{e:rhobarenergym2_m4} with $n=1$. Observing that $P(1) = \bar D$ by the definition \eqref{e:e_D}, we have
\begin{equation}\label{e:rhobarenergym2_1}
\sup_{t\le T} \|(\bar\kappa^\frac{1}{2} \nabla\bar\rho)(t)\|^2_{L^2}+\int_0^T\|\bar\kappa\nabla^2\bar\rho\|_{L^2}^2 \lesssim \tau^{-1}\bar D + \|(\bar\kappa^\frac{1}{2} \nabla\bar\rho)_{in}\|^2_{L^2}
\end{equation}
i.e. \eqref{e:rhobarenerg} with $n=1$, which allows to start induction. For the inductive step, assume that \eqref{e:rhobarenerg} holds for any $j \le n$, in particular 
\[
P(j+1) \lesssim_n (\tau^{-1})^j \bar D + \sum_{i=1}^j (\tau^{-1})^{j-i} \|(\bar\kappa^\frac{i}{2} \nabla^i\bar\rho)_{in}\|^2_{L^2}.
\]
For $n+1$ we have via \eqref{e:rhobarenergym2_m4}
\[
\begin{split}
P(n+2) & \lesssim_n  \|(\bar\kappa^\frac{n+1}{2} \nabla^{n+1}\bar\rho)_{in}\|^2_{L^2} + \sum_{j=0}^{n} P(j+1) (\tau^{-1})^{n+1-j} \\
&\lesssim_n  \|(\bar\kappa^\frac{n+1}{2} \nabla^{n+1}\bar\rho)_{in}\|^2_{L^2} + \sum_{j=0}^{n} \left((\tau^{-1})^j \bar D + \sum_{i=1}^j (\tau^{-1})^{j-i} \|(\bar\kappa^\frac{i}{2} \nabla^i\bar\rho)_{in}\|^2_{L^2} \right)(\tau^{-1})^{n+1-j},
\end{split}
\]
which gives \eqref{e:rhobarenerg} for $n+1$.

\noindent\emph{Step 4: Transport estimate.} Apply $\pal$, where $|\boldsymbol{\alpha}|=n$, to the equation \eqref{e:e_homogenized}
\[
\begin{split}
D_t \pal \bar\rho =  -\sum_{{\boldsymbol{\beta}} + {\boldsymbol{\gamma}} ={\boldsymbol{\alpha}} , {\boldsymbol{\beta}} >0} c_{\boldsymbol{\beta}}  \partial^{\boldsymbol{\beta}}  u \cdot \partial^{\boldsymbol{\gamma}}  \nabla \bar\rho + \div  \sum_{{\boldsymbol{\beta}} + {\boldsymbol{\gamma}} ={\boldsymbol{\alpha}}} c_{\boldsymbol{\beta}}   \partial^{\boldsymbol{\beta}}  \bar A \partial^{\boldsymbol{\gamma}}  \nabla \bar\rho. 
\end{split}
\]
Multiply both sides with $\bar\kappa^\frac{n}{2}$, distribute the powers of $\bar\kappa$ appropriately, and take space-time $L^2$ norms (denoted by $L^2_{xt}$) to obtain the following estimate
\[
\begin{split}
\|\bar\kappa^\frac{n}{2} D_t \pal \bar\rho\|_{L^2_{xt}} &\lesssim_n \sum_{i+j=n, i>0} \|\bar\kappa^\frac{i-1}{2}  \nabla^i u\|_{L^\infty}  \| \bar\kappa^\frac{j+1}{2}  \nabla^{j+1} \bar\rho\|_{L^2_{xt}} \\
&+  \sum_{i+j=n}  \|\bar\kappa^\frac{i-1}{2} \nabla^{i+1} \bar A \|_{L^\infty}  \| \bar\kappa^\frac{j+1}{2}  \nabla^{j+1} \bar\rho\|_{L^2_{xt}} +   \sum_{i+j=n} \| \bar\kappa^\frac{i-2}{2} \nabla^i \bar A \|_{L^\infty}  \|  \bar\kappa^\frac{j+2}{2} \nabla^{j+2} \bar\rho\|_{L^2_{xt}}. 
\end{split}
\]
Use the assumptions \eqref{e:e_modelAestimates}, \eqref{e:e_model-estimates2}
\[
\|\bar\kappa^\frac{n}{2} D_t \pal \bar\rho\|_{L^2_{xt}}  \lesssim_n 
\tau^{-1}  \sum_{j=1}^{n+2} (\tau^{-1})^\frac{n-j}{2}   \| \bar\kappa^\frac{j}{2}  \nabla^{j} \bar\rho\|_{L^2_{xt}}.
\]
Squaring both sides above yields
\[
\int_0^T \|\bar\kappa^\frac{n}{2} D_t \pal \bar\rho\|^2_{L^2} \lesssim_n (\tau^{-1})^2  \sum_{j=1}^{n+2} (\tau^{-1})^{n-j} P(j).
\]
The estimate \eqref{e:rhobarflowB} now follows from using \eqref{e:rhobarenerg} to control $P(j)$.\end{proof}

\section{Spatial homogenization}
\label{s:homogenization}

The main result of this section are Proposition \ref{p:hom_trho} and its corollaries, which allow to prove the space homogenization step of the key Proposition \ref{p:main}.

In this section, we use $(e_1,e_2,e_3)$ to denote the unit Euclidean coordinate  vectors. We use $\xi \in \T^3$ to denote the variable in the cell for homogenization. For a function $f: \T^3 \times [0,T] \times \T^3 \rightarrow \R$, i.e. taking the variable $(x,t,\xi) \in \T^3 \times [0,T] \times \T^3$ as its argument, we use $\langle f \rangle$ to denote its average in the cell variable $\xi \in \T^3$.
We use $\|f(x,t,\cdot)\|_{L^\infty_\xi}$ to denote the supremum norm in $\xi$ variable. Note that $\|f\|_{L^\infty_\xi}$ is still a function in $(x,t)$. With $u$ defined below in \eqref{e:hom_base:eq}, we define the transport derivative $D_t$ on $f$ to be $$D_t f := \partial_t f + (u \cdot \nabla_x) f.$$
Without further specification, the $L^2$ and $H^{-1}$ norms are in $x$ variable, and the supremum norm $L^{\infty}$ is the supremum norm in $(x,t)$.
The constant in $\lesssim$ in this section depends on $N_h$ defined in \eqref{e:hom_para_assump:Nh}.

\subsection{Setup}
Consider the following advection-diffusion equation for $\rho: \T^3 \times [0,T] \rightarrow \R$
\begin{equation} \label{e:hom_base:eq}
\begin{split}
    \partial_t\rho+u\cdot\nabla \rho =& \div \tilde A\nabla\rho, \\ 
    \rho|_{t=0} =& \rho_{\text{in}}, 
\end{split}
\end{equation}
with divergence-free $u$ and the elliptic tensor $\tilde A: \T^3 \times [0,T] \rightarrow \R^{3 \times 3}$ given by
\begin{equation}  
    \tilde A(x,t) := \sum_i \tilde\eta_i (x,t) \nabla \Phi_i^{-1} (x,t) \Bigg( \kappa \Id + \frac{\eta_i (x,t) \sigma^{1/2} (t)}{\lambda} \sum_{\vec{k} \in \Lambda} a_{\vec{k}} \left( \tilde R_{i} (x,t) \right) H_{\vec{k}}(\lambda \Phi_i(x,t) ) \Bigg) \nabla \Phi_i^{-T} (x,t)\label{e:hom_base:tA}.
\end{equation}
The tensor defined in \eqref{e:hom_base:tA} is the counterpart of $\tilde A_{q+1}(x,t)$ of \eqref{e:defAq+1tilde} with the $q+1$ index dropped. The precise connection between the parameters is given in \eqref{e:homo_connect}.

We also define the elliptic tensor in cell variable $\xi$ by 
\begin{equation}
    A(x,t,\xi) := \sum_i \tilde\eta_i(x,t) \nabla \Phi_i^{-1}(x,t) \Bigg( \kappa \Id + \frac{\eta_i(x,t) \sigma^{1/2}(t)}{\lambda} \sum_{\vec{k} \in \Lambda} a_{\vec{k}} \left( \tilde R_{i}(x,t) \right) H_{\vec{k}}(\xi) \Bigg) \nabla \Phi_i^{-T}(x,t). \label{e:hom_base:A}
\end{equation}

Our goal in this section is to show the solution $\rho$ to \eqref{e:hom_base:eq} homogenizes to the solution $\bar \rho: \T^3 \times [0,T] \rightarrow \R$ of the following equation
\begin{equation} \label{e:hom_ed:eq}
\begin{split}
    \partial_t \bar \rho + u \cdot \nabla \bar \rho =& \div \big( \bar A \nabla \bar \rho \big), \\ 
    \bar \rho|_{t=0} =& \rho_{\text{in}}, 
\end{split}
\end{equation}
with the elliptic tensor $\bar A: \T^3 \times [0,T] \rightarrow \R^{3 \times 3}$ given by
\begin{align} \label{e:hom_ed:bA}
    \bar A(x,t) = \dashint A(x,t,\xi) \Big( \Id + \sum_i \eta_i(x,t) \nabla\Phi_i^T(x,t) \nabla_\xi\chi_i^T(x,t,\xi) \Big) \,d\xi, 
\end{align}
where $\chi_i, \chi: \T^3 \times [0,T] \times \T^3 \rightarrow \R^3$ are given by
\begin{align} \label{e:hom_ed:chi}
    \chi_i(x,t,\xi) =&\frac{\sigma^{1/2}(t)}{\kappa \lambda} \nabla \Phi_i^{-1}(x,t) \sum_{\vec{k} \in \Lambda} a_{\vec{k}} \left( \tilde R_i(x,t) \right) \varphi_{\vec{k}}(\xi) \vec{k}, \\ 
    \chi(x,t,\xi) =&  \sum_i \chi_i(x,t,\xi) \eta_i(x,t).
\end{align}

\subsubsection{Assumptions for homogenization}\label{s:ass_homo}
In order to make the homogenization section possibly independent from other sections, in what follows we extract from Section \ref{s:blocks} assumptions needed for our quantitative homogenization results.

We introduce the parameters $\kappa, \lambda, \delta, \tau, \ell \in \R^+$ and $\bar \kappa: \T^3 \times [0,T] \rightarrow \R^+$ for which we require
\begin{align}
    \delta^{1/2} \ell \kappa^{-1} \lambda^{-1} \leq & 1, \label{e:hom_para_assump:deltaell} \\
    \tau \|\bar \kappa\|_{L^\infty} \ell^{-2} \leq & 1, \label{e:hom_para_assump:taukappaell} \\ 
    \tau \|\bar \kappa\|_{L^\infty} \lambda^{\frac{2}{b}} \leq & \lambda^{-2\gamma}, \label{e:hom_para_assump:for_rem} \\ 
    \|\bar \kappa\|_{L^\infty}^{-1} \tau^{-1} \lambda^{-2}
    \leq \kappa^{-1} \tau^{-1} \lambda^{-2} \leq & \lambda^{-4\gamma}.  \label{e:hom_para_assump:kappataulambda}
\end{align}

We choose $N_h \geq 3$ sufficiently large such that
\begin{align}
    N_h \left( \frac{b-1}{b+1} \big( 1 - (2b+1)\beta \big) - \gamma_T \right) \geq 2.  \label{e:hom_para_assump:Nh}
\end{align}
We require for the functions $\eta_i, \tilde\eta_i : \T^3 \times [0,T] \rightarrow \R,\, \sigma : [0,T] \rightarrow \R,\, \Phi_i:\T^3 \times [0,T] \rightarrow \T^3,\, a_{\vec{k}} \circ \tilde R_{i}: \T^3 \times [0,T] \rightarrow  \R$ the following conditions
\begin{align}
    \sigma &\sim \delta,   \label{e:hom_assump:sigmadelta}  \\
    \| \nabla \Phi_i - \Id \|_{L^\infty} &\leq \lambda^{-2\gamma},
\label{e:hom_assump:Phi} \\ 
    \frac12\bar \kappa &\leq \bar A \leq 2 \bar \kappa,
    \quad\textrm{ where }\bar \kappa= \kappa \left( 1 + \sum_i \frac{\sigma \eta_i^2}{\kappa^2\lambda^2} \right), \label{e:hom_assump:barkappa}
\end{align}
and for $n=0,1$ and $m \leq N_h$
\begin{align}
        \|\nabla u\|_{C^m}&\lesssim \delta^{\sfrac{1}{2b}} \lambda^{\sfrac{1}{b}} \ell^{-m}  \label{e:hom_assump:u} \\
    \| \nabla_x^m D_t^n \nabla \Phi_i \|_{L^\infty} + \| \nabla_x^m D_t^n(a_{\vec{k}} \circ \tilde R_{i}) \|_{L^\infty} &\lesssim \tau^{-n}\ell^{-m}, \label{e:hom_assump:slow} \\
        \| \nabla_x^m D_t^n  \tilde \eta_i \|_{L^\infty} +
    \|\nabla_x^m  D_t^n \eta_i \|_{L^\infty}
        &\lesssim \tau^{-n}, \label{e:hom_assump:slow:eta} \\
    \| \nabla_x^m D_t^n\chi_i(x,t,\cdot) \|_{C^1_\xi}
        &\lesssim \frac{\sigma^{1/2}}{\kappa \lambda} \tau^{-n}\ell^{-m}, \label{e:hom_assump:chi}\\ 
        \| \pat \sigma \|_{L^\infty}
        &\lesssim \sigma \tau^{-1}. \label{e:hom_assump:sigma} 
\end{align}
Furthermore, we require that $\det \nabla\Phi_i^T = 1$ and that $\eta_i=\eta_i(x,t)$, $i\in\N$ are smooth nonnegative cutoff functions with pairwise disjoint supports, whereas $(\tilde\eta_i)_{i\in\N}$ is a partition of unity such that $\tilde\eta_i\eta_i=\eta_i$,  $\tilde\eta_i\eta_j=0$ for $j\neq i$. Finally, $H_{\vec{k}}=H_{\vec{k}}(\xi)$ is the antisymmetric zero-mean matrix with the property that for any $v \in \R^3$ we have $H_{\vec{k}} v = - U_{\vec{k}} \times v$, where $U_{\vec{k}}= U_{\vec{k}} (\xi)$ is given in \eqref{e:defUW} (and the functions defining $U_{\vec{k}}$ are provided at the beginning of Section \ref{s:Mikado}).

Observe that \eqref{e:hom_para_assump:Nh} gives the following relation
\begin{align}
    \lambda^2 & \lesssim \big( \tau \kappa \lambda^2 \big)^{N_h}.  \label{e:hom_para_assump:Nh_result}
\end{align}

\subsection{Quantitative homogenization estimates}
\begin{proposition} \label{p:hom_trho}
Given the assumptions gathered in Section \ref{s:ass_homo}, let $\rho$ be the solution to \eqref{e:hom_base:eq}-\eqref{e:hom_base:tA} and $\bar \rho$ be the solution to \eqref{e:hom_ed:eq}-\eqref{e:hom_ed:chi}. Define $\tilde \rho: \T^3 \times [0,T] \rightarrow \R$ such that
\begin{equation} \label{e:hom_trho_ansatz}
    \rho(x,t) = \bar{\rho}(x,t) + \frac{1}{\lambda} \sum_i \eta_i(x,t) \chi_i \big( x,t,\lambda \Phi_i(x,t) \big)\cdot\nabla\bar\rho(x,t)+\tilde\rho(x,t).
\end{equation}
Then
\begin{equation}  
\begin{split}
\label{e:hom_trho_est}
    \sup_{t \le T}\| \tilde \rho(t) \|^2_{L^2} + \kappa \int_0^T \| \nabla \tilde \rho \|^2_{L^2} dt
        \lesssim \frac{1}{\lambda^2 \kappa \tau} \mathcal{D}_{N_h},\\
        \text{ where } \quad
    \mathcal{D}_l := \int_0^T \| \bar \kappa^{\frac{1}{2}} \nabla \bar \rho \|^2_{L^2} dt 
        + \sum_{j=1}^{l} \tau^{j} \big\| \big(\bar \kappa^{\frac{j}{2}} \nabla^j \rho\big)_{\text{in}} \big\|_{L^2}^2 \quad \text{ for } l \in \N.
    \end{split}
\end{equation}
\end{proposition}

We also have the following corollaries.
\begin{corollary} \label{c:hom_trho}
Let $\rho, \bar \rho$ and $\mathcal{D}_l$ be as in Proposition \ref{p:hom_trho}, then
\begin{equation}  \label{c:hom_trho_est}
      \sup_{t \le T}  \| \rho(t) - \bar\rho(t) \|^2_{L^2}
        \lesssim \frac{1}{\lambda^2 \kappa \tau} \mathcal{D}_{N_h}.
\end{equation}
\end{corollary}

\begin{corollary}   \label{c:hom_dissipation}
Let $\rho, \bar \rho$ and $\mathcal{D}_l$ be as in Proposition \ref{p:hom_trho}, then
\begin{equation}  \label{c:hom_dissipation_est}
\begin{split}
       \sup_{t \le T} \Big| \| \rho(t) \|^2_{L^2} - \| \bar\rho(t) \|^2_{L^2} \Big|
        \lesssim \Big( \frac{1}{\lambda \kappa^{\sfrac{1}{2}} \tau^{\sfrac{1}{2}} } + \lambda^{-2\gamma} \Big) \mathcal{D}_{N_h}. \\
\end{split}
\end{equation}
\end{corollary}

\begin{corollary}   \label{r:hom_dissipation}
Let $\rho, \bar \rho$ be as in Proposition \ref{p:hom_trho}. Assume for any $j \in [0,N_h]$
\begin{align}   \label{r:hom_dissipation_eq0}
    \|\nabla^j \rho_{in}\|_{L^2}^2 \leq \lambda^{\frac{2j}{b}} \kappa \int_0^T \| \nabla \rho \|_{L^2}^2 dt,
\end{align}

then
\begin{align}
     \sup_{t \le T}   \Big| \| \rho(t) \|^2_{L^2} - \| \bar\rho(t) \|^2_{L^2} \Big|
    \lesssim \lambda^{ -2\gamma }
    \min \left\{ \int_0^T \| \bar \kappa^{\frac{1}{2}} \nabla \bar \rho \|^2_{L^2} dt, \kappa \int_0^T \| \nabla \rho \|_{L^2}^2 dt \right\}.
\end{align}
\end{corollary}

\begin{proof}[Proof of Corollary \ref{r:hom_dissipation}]
With \eqref{r:hom_dissipation_eq0}, we have
\begin{equation} \label{r:hom_dissipation_e1}
\begin{split}
    \mathcal{D}_{N_h} \leq& \int_0^T \| \bar \kappa^{\frac{1}{2}} \nabla \bar \rho \|^2_{L^2} dt
        + \sum_{j=1}^{N_h} \tau^j \|\bar \kappa\|_{L^\infty}^j \lambda^{\frac{2j}{b}} \kappa 
        \int_0^T \| \nabla \rho \|_{L^2}^2 dt, \\
    \overset{\textrm{\eqref{e:hom_para_assump:for_rem}}}{\leq} &
     \int_0^T \| \bar \kappa^{\frac{1}{2}} \nabla \bar \rho \|^2_{L^2} dt + \kappa \int_0^T \| \nabla \rho \|_{L^2}^2 dt.
\end{split}
\end{equation}
Therefore, Corollary \ref{r:hom_dissipation} follows from \eqref{c:hom_dissipation_est}, \eqref{r:hom_dissipation_e1} and \eqref{e:hom_para_assump:kappataulambda}.
\end{proof}

The proofs of Proposition \ref{p:hom_trho} and Corollaries \ref{c:hom_trho}, \ref{c:hom_dissipation} are given in the end of this section.

The term $\tilde \rho$ in Proposition \ref{p:hom_trho} is the error term in homogenization. In the following  two lemmas, in order to estimate  $\tilde \rho$, we derive and analyse the equation it satisfies.

\begin{lemma} \label{l:hom_trho_comp} Let the assumptions of Proposition \ref{p:hom_trho} hold and let $\tilde \rho$ be as in \eqref{e:hom_trho_ansatz}. Then $\tilde \rho$ satisfies
\begin{equation}\label{e:hom_trho_eq}
\begin{split}
    \partial_t\tilde\rho &+ u\cdot\nabla \tilde\rho - \div \big( \tilde A\nabla\tilde\rho \big)
    = \div \big( \tilde B \nabla\bar{\rho} \big) \\ 
    &+ \frac{1}{\lambda} \sum_i \left( \div \left( \tilde A \nabla^2\bar\rho \eta_i \chi_i(x,t,\lambda \Phi_i) \right) + \div \left( \tilde A \nabla_x ( \eta_i\chi_i )^T (x,t,\lambda \Phi_i) \nabla\bar{\rho} \right) \right) \\
    &- \frac{1}{\lambda} \sum_i \Big( \eta_i \chi_i(x,t,\lambda \Phi_i) \cdot D_t\nabla\bar{\rho} + \nabla\bar\rho \cdot D_t (\eta_i\chi_i) (x,t,\lambda \Phi_i) \Big),
\end{split}
\end{equation}
with the matrices $B: \T^3 \times [0,T] \times \T^3 \rightarrow \R^{3 \times 3}$ and $\tilde B: \T^3 \times [0,T] \rightarrow \R^{3 \times 3}$ given by
\begin{align}
    B(x,t,\xi) =& A(x,t,\xi) \Big( \Id+ \sum_i \eta_i \nabla\Phi_i^T \nabla_\xi\chi_i^T(x,t,\xi) \Big) - \Big\langle A \Big( \Id + \sum_i \eta_i \nabla\Phi_i^T \nabla_\xi\chi_i^T \Big) \Big\rangle,     \label{e:hom_trho:B} \\ 
    \tilde B(x,t) =& \tilde A \Big( \Id+ \sum_i \eta_i \nabla\Phi_i^T \nabla_\xi\chi_i^T(x,t,\lambda \Phi_i) \Big) - \Big\langle A \Big( \Id + \sum_i \eta_i \nabla\Phi_i^T \nabla_\xi\chi_i^T \Big) \Big\rangle.   \label{e:hom_trho:tB}
\end{align}
\end{lemma}

\begin{proof}[Proof of Lemma \ref{l:hom_trho_comp}]
From the ansatz \eqref{e:hom_trho_ansatz}, omitting the argument $\big( x, t, \lambda \Phi_i(x,t) \big)$, direct computations give
\begin{align}
    \nabla\tilde\rho &= \nabla\rho - \Big( \Id + \sum_i \eta_i \nabla\Phi_i^T \nabla_{\xi} \chi_i^T \Big) \nabla\bar{\rho} - \frac{1}{\lambda} \sum_i \Big( \nabla^2\bar\rho\chi_i\eta_i + \nabla_x (\chi_i\eta_i)^T\nabla\bar{\rho} \Big), \label{l:hom_trho_comp:e3} \\
    \partial_t\tilde\rho &= \partial_t\rho - \partial_t\bar\rho - \sum_i \eta_i \partial_t\Phi_i^T \nabla_\xi\chi_i^T \nabla\bar\rho - \frac{1}{\lambda} \sum_i \Big( \partial_t\nabla\bar{\rho}\cdot\chi_i\eta_i + \nabla\bar\rho \cdot \partial_t(\chi_i\eta_i) \Big). \label{l:hom_trho_comp:e4}
\end{align}
Using $\partial_t\Phi_i + (u\cdot\nabla)\Phi_i=0$ and omitting the argument $\big( x, t, \lambda \Phi_i(x,t) \big)$, we obtain
\begin{equation*}
\begin{split}
    \partial_t\tilde\rho + u\cdot\nabla \tilde\rho - \div \tilde A\nabla\tilde\rho 
    =& \div\left(\left[ \tilde A \Big( \Id+ \sum_i \eta_i \nabla\Phi_i^T \nabla_\xi\chi_i^T \Big) - \Big\langle A \Big( \Id + \sum_i \eta_i \nabla\Phi_i^T \nabla_\xi\chi_i^T \Big) \Big\rangle \right]\nabla\bar{\rho} \right)\\
    &+ \frac{1}{\lambda} \sum_i \Big( \div \left( \tilde A \nabla^2\bar\rho \chi_i \eta_i \right) + \div \left( \tilde A \nabla_x (\chi_i\eta_i) ^T\nabla\bar{\rho} \right) \Big) \\ 
    &- \frac{1}{\lambda} \sum_i \Big( \partial_t\nabla\bar{\rho} \cdot \chi_i \eta_i + \partial_t(\chi_i\eta_i) \cdot \nabla \bar\rho \Big) - \frac{1}{\lambda} \sum_i \Big( u \cdot \big( \nabla^2 \bar\rho\chi_i\eta_i \big) + u \cdot \big( \nabla_x \chi_i^T \eta_i \nabla\bar{\rho} \big) \Big) \\ 
    =& \div \big( \tilde B \nabla\bar{\rho} \big) + \frac{1}{\lambda} \sum_i \Big( \div \left( \tilde A \nabla^2\bar\rho \chi_i \eta_i \right) + \div \left( \tilde A \nabla_x (\chi_i\eta_i)^T \nabla\bar{\rho} \right) \Big) \\
    &- \frac{1}{\lambda} \sum_i \Big( \eta_i \chi_i \cdot D_t\nabla\bar{\rho} + \nabla\bar\rho \cdot D_t(\chi_i\eta_i) \Big).
\end{split}
\end{equation*}
\end{proof}
Next, we reformulate the right-hand side of \eqref{e:hom_trho_eq} and provide appropriate estimates. In particular, we need to use potentials of high order (cf. $l_0$ in what follows), which allows to close estimates in Theorem \ref{t:main}.
\begin{lemma} \label{l:hom_b_comp}
Let the assumptions of Proposition \ref{p:hom_trho} hold and let $\tilde \rho$ be as in \eqref{e:hom_trho_ansatz}. Then for any $\,l_0 \leq N_h$,  $\tilde \rho$ satisfies
\begin{equation} \label{e:hom_trho_eq_detail}
\begin{split}
    \partial_t\tilde\rho &+ u\cdot\nabla \tilde\rho - \div \big( \tilde A\nabla\tilde\rho \big)
    = \frac{1}{\lambda} \Big( E_1 + E_2 + E_3 + E_4 + \sum_{l=1}^{l_0} F_l + G_{l_0} \Big),
\end{split}
\end{equation}
with
\begin{align*}
    E_1 =& - \sum_{i,j} \div \left( \nabla\partial_j\bar\rho \times \left( \nabla\Phi_i^T c^{(i)}_j(x,t,\lambda \Phi_i) \right) \right),  \\ 
    E_2 =& - \sum_{i,j,m} \div \left( \left( \nabla_x c^{(i)}_{jm} (x,t,\lambda \Phi_i) \times \nabla\Phi_{i,m} \right) \partial_j\bar\rho \right),  \\
    E_3 =&  \sum_i \div \left( \tilde A \nabla^2\bar\rho \eta_i\chi_i(x,t,\lambda \Phi_i) \right),  \\ 
    E_4 =&  \sum_i \div \left( \tilde A \nabla_x \big( \eta_i\chi_i(x,t,\lambda \Phi_i) \big)^T \nabla\bar{\rho} \right)
\end{align*}
and
\begin{align}
    F_1 =& 0, \\ 
    F_l =& \frac{1}{\lambda^{l-1}} \sum_{i,1 \leq |\boldsymbol{\alpha}| \leq l-1} \div \Big( f^{(l-1)}_{0,\boldsymbol{\alpha},i}( x, t, \lambda \Phi_i) D_t \partial^{\boldsymbol{\alpha}} \bar \rho + f^{(l-1)}_{1,\boldsymbol{\alpha},i}( x, t, \lambda \Phi_i) \partial^{\boldsymbol{\alpha}} \bar \rho \Big), \text{ for } l \geq 2,     \label{e:hom_trho_FN} \\ 
    G_l =& \frac{1}{\lambda^{l-1}} \sum_{i,1 \leq |\boldsymbol{\alpha}| \leq l} \Big( g^{(l)}_{0,\boldsymbol{\alpha},i} ( x, t, \lambda \Phi_i) D_t \partial^{\boldsymbol{\alpha}} \bar \rho + g^{(l)}_{1,\boldsymbol{\alpha},i} ( x, t, \lambda \Phi_i) \partial^{\boldsymbol{\alpha}} \bar \rho \Big),     \label{e:hom_trho_GN}
\end{align}
for certain functions $c_j^{(i)} : \T^3 \times [0,T] \times \T^3 \rightarrow \R^3$, $f^{(l)}_{0,\boldsymbol{\alpha},i}$, $f^{(l)}_{1,\boldsymbol{\alpha},i}$, $g^{(l)}_{0,\boldsymbol{\alpha},i}$, $g^{(l)}_{1,\boldsymbol{\alpha},i}: \T^3 \times [0,T] \times \T^3 \rightarrow \R$, satisfy the following estimates
\begin{align}
    \langle f^{(l)}_{0,\boldsymbol{\alpha},i} \rangle = \langle f^{(l)}_{1,\boldsymbol{\alpha},i} \rangle &=
        \langle g^{(l)}_{0,\boldsymbol{\alpha},i} \rangle = \langle g^{(l)}_{1,\boldsymbol{\alpha},i} \rangle = 0,  \label{e:hom_trho_div_free} \\ 
    \| c^{(i)}_j \|_{L^\infty_\xi}
        &\lesssim \kappa \Big( 1 + \frac{\sigma^{1/2}\eta_i}{\kappa \lambda} \Big) \frac{\sigma^{1/2}\eta_i}{\kappa \lambda},    \label{e:hom_trho_c_est} \\ 
    \| \nabla_x c^{(i)}_j \|_{L^\infty_\xi}
        &\lesssim \kappa \Big( 1 + \frac{\sigma^{1/2}\eta_i}{\kappa \lambda} \Big)
        \Big( \frac{\sigma^{1/2}\eta_i}{\kappa \lambda} \ell^{-1} + \frac{\sigma^{1/2}}{\kappa \lambda} \Big),    \label{e:hom_trho_c_est1} \\ 
    \big\| \nabla_x^m f^{(l)}_{0,\boldsymbol{\alpha},i} \big\|_{L^\infty_\xi}
    + \big\| \nabla_x^m g^{(l)}_{0,\boldsymbol{\alpha},i} \big\|_{L^\infty_\xi}
        &\lesssim \frac{\sigma^{1/2}}{\kappa \lambda} \ell^{-(l-|\boldsymbol{\alpha}|+m)} ( \eta_i + l ), 
        \quad \text{for } l-|\boldsymbol{\alpha}|+m \leq N_h,  \label{e:hom_trho_f0g0_est} \\ 
    \big\| \nabla_x^m f^{(l)}_{1,\boldsymbol{\alpha},i} \big\|_{L^\infty_\xi}
    + \big\| \nabla_x^m g^{(l)}_{1,\boldsymbol{\alpha},i} \big\|_{L^\infty_\xi}
        &\lesssim \frac{\sigma^{1/2}}{\kappa \lambda} \tau^{-1} \ell^{-(l-|\boldsymbol{\alpha}|+m)} ( \eta_i + l ),
        \quad \text{for } l-|\boldsymbol{\alpha}|+m \leq N_h \label{e:hom_trho_f1g1_est},
\end{align}
and
\begin{align}   \label{e:hom_trho_fg_supp}
    \supp f^{(l)}_{0,\boldsymbol{\alpha},i}, \supp f^{(l)}_{1,\boldsymbol{\alpha},i}
    \supp g^{(l)}_{0,\boldsymbol{\alpha},i}, \supp g^{(l)}_{1,\boldsymbol{\alpha},i}
        \subset \supp \eta_i,
\end{align}
where $c_{jm}^{(i)} := c_j^{(i)} \cdot e_m$ and $\nabla\Phi_{i,m} := \nabla\Phi_{i} \cdot e_m$.
\end{lemma}

\begin{proof}[Proof of Lemma \ref{l:hom_b_comp}]
Define $\chi_{ij} = \chi_i \cdot e_j$. Let $b_j = B e_j$ and $\tilde b_j = \tilde B e_j$, with $B$ and $\tilde B$ as in Lemma \ref{l:hom_trho_comp}. Direct computations show that one can define $b^{(i)}_j$ such that
\begin{equation}    \label{e:hom_b_comp:e2}
\begin{split}
    \sum_i& \eta_i b^{(i)}_j := b_j(x,t,\xi) \\
     =& \sum_i \eta_i \nabla\Phi_i^{-1} \Bigg[ \kappa \nabla_\xi\chi_{ij}
    + \frac{\sigma^{1/2}}{\lambda} \sum_{\vec{k} \in \Lambda}  a_{\vec{k}}(\tilde R_i) \Big( H_{\vec{k}} \nabla\Phi_i^{-T}e_j + \eta_i H_{\vec{k}} \nabla_\xi \chi_{ij}
    - \eta_i \langle H_{\vec{k}} \nabla_\xi \chi_{ij} \rangle \Big) \Bigg].
\end{split}
\end{equation}
We claim $\div_\xi (\nabla \Phi_i b^{(i)}_j)=0$. Indeed, notice that
\begin{equation}    \label{e:hom_b_comp:e3}
\begin{split}
    \div_\xi \big( H_{\vec{k}} \nabla_\xi \chi_{ij} \big) 
        =& -\big( W_{\vec{k}} \cdot \nabla_\xi \big) \chi_{ij}
        = -\frac{\sigma^{1/2}}{\kappa \lambda} \nabla \Phi_i^{-1} \sum_{\vec{k}} a_{\vec{k}} \big( \tilde R_i \big) \big( \psi_{\vec{k}} \vec{k} \cdot \nabla_\xi \big) \varphi_{\vec{k}} \vec{k} \cdot e_j = 0, \\ 
    \div_\xi \big( H_{\vec{k}} \nabla \Phi_i^{-T} e_j \big) 
        =& -W_{\vec{k}} \cdot \big( \nabla \Phi_i^{-T}e_j \big) = -\psi_{\vec{k}} \big( \nabla \Phi_i^{-1} \vec{k} \big) \cdot e_j.
\end{split}
\end{equation}
Then plugging \eqref{e:hom_b_comp:e3} and \eqref{e:hom_ed:chi} into \eqref{e:hom_b_comp:e2} gives $\div_\xi (\nabla \Phi_i b^{(i)}_j)=0$. Since also $\langle \nabla \Phi_i b^{(i)}_j \rangle = 0$, we can find a vector potential $c^{(i)}_j=c^{(i)}_j(x,t,\xi)$ so that $\eta_i \nabla\Phi_i \, b^{(i)}_j=\nabla_\xi\times c^{(i)}_j$. Using the fact $\det \nabla\Phi_i^T = 1$ and omitting the argument $\big( x, t, \lambda \Phi_i(x,t) \big)$, we have
\begin{equation*}
\begin{split}
    \frac{1}{\lambda}\nabla\times \big( \nabla \Phi_i^T c^{(i)}_j \big) 
    &= \frac{1}{\lambda} \nabla\times \big( c^{(i)}_{jm}\nabla\Phi_{i,m} \big) = \frac{1}{\lambda} \nabla c^{(i)}_{jm} \times\nabla\Phi_{i,m}\\
    &= \frac{1}{\lambda}\nabla_x c^{(i)}_{jm}\times\nabla\Phi_{i,m} + \big( \nabla\Phi_i^T \nabla_\xi c^{(i)}_{jm} \big) \times \big( \nabla\Phi_i^Te_m \big)\\
    &= \frac{1}{\lambda} \nabla_x c^{(i)}_{jm}\times\nabla\Phi_{i,m} + \nabla\Phi_i^{-1} \big( \nabla_\xi\times c^{(i)}_j \big),
\end{split}
\end{equation*}
where for the last term we have used the formula $(Av_1) \times (Av_2) = A^{-T} (v_1 \times v_2) \det A$.
Therefore, omitting the argument $\big( x, t, \lambda \Phi_i(x,t) \big)$, we have
\begin{align*}
    b_j
    =& \sum_i \nabla \Phi_i^{-1}\nabla_\xi\times c^{(i)}_j
    = \sum_i \Big( \frac{1}{\lambda}\nabla\times (\nabla\Phi_i^T c^{(i)}_j) - \frac{1}{\lambda}\nabla_x c^{(i)}_{jm}\times \nabla\Phi_{i,m} \Big), \\ 
    b_j\partial_j\bar\rho =& \frac{1}{\lambda}\nabla \times \big( \partial_j\bar\rho \nabla\Phi_i^T c^{(i)}_j \big)
    - \frac{1}{\lambda} \nabla\partial_j\bar\rho \times \big( \nabla\Phi_i^T c^{(i)}_j \big)
    - \frac{1}{\lambda} \big( \nabla_x c^{(i)}_{jm} \times\nabla\Phi_{i,m} \big) \partial_j\bar\rho.
\end{align*}
This gives $E_1$ and $E_2$ from the first line of \eqref{e:hom_trho_eq}. The quantities $E_3$ and $E_4$ are given directly by the second line of \eqref{e:hom_trho_eq}. The estimate \eqref{e:hom_trho_c_est} follows from \eqref{e:hom_assump:chi} by Schauder estimate, and from boundedness of respective terms given by \eqref{e:hom_assump:slow}. For the estimate \eqref{e:hom_trho_c_est1}, one gets a factor $\ell^{-1}$ when $\nabla_x$ hits $\chi_{ij}$. Together with the fact $|\nabla \eta_i| \lesssim 1$ from \eqref{e:hom_assump:slow:eta}, this gives \eqref{e:hom_trho_c_est1}. 

\par

Next, by induction, we show the last line of \eqref{e:hom_trho_eq} equals $\sum_{l=1}^{l_0} F_l+G_{l_0}$ for any $1 \leq l_0 \leq N_h$, which gives \eqref{e:hom_trho_eq_detail}. We also show the estimates \eqref{e:hom_trho_f0g0_est} and \eqref{e:hom_trho_f1g1_est} in this step. When $l_0=1$, $|\boldsymbol{\alpha}| = 1$, let $g^{(1)}_{0,\boldsymbol{\alpha},i}$ be given by $-\eta_i\chi_i$, and $g^{(1)}_{1,\boldsymbol{\alpha},i}$ be given by $-D_t(\eta_i\chi_i)$. Let $F_1=0$, then $F_1 + G_1$ equals the the last line of \eqref{e:hom_trho_eq}. Using \eqref{e:hom_assump:chi} and Schauder estimate, we have \eqref{e:hom_trho_div_free} and the estimates \eqref{e:hom_trho_f0g0_est}, \eqref{e:hom_trho_f1g1_est} for $l=1$. 
\par 

Assume \eqref{e:hom_trho_eq_detail} with \eqref{e:hom_trho_FN}, \eqref{e:hom_trho_GN} is true for $l_0$. To close the induction argument, it suffices to show
\begin{align} \label{e:hom_b_comp:e5}
    G_{l_0} = F_{l_0+1} + G_{l_0+1}
\end{align}
for suitable $F_{l_0+1}$ and $G_{l_0+1}$ satisfying the estimate \eqref{e:hom_trho_f0g0_est} and \eqref{e:hom_trho_f1g1_est}.
Using chain rule, for any smooth function $h(x,t,\xi): \T^3 \times [0,1] \times \T^3 \rightarrow \R^3$,
we have
\begin{equation*}
\begin{split}
    \frac{1}{\lambda} \div \bigl( \nabla\Phi_i^{-1} h (x,t,\lambda\Phi_i) \bigr)
    &= (\div_\xi h)(x,t,\lambda\Phi_i)
    + \frac{1}{\lambda} \div_x( \nabla\Phi_i^{-1} h) (x,t,\lambda\Phi_i).
\end{split}
\end{equation*}

Let\footnote{Here it is enough to use the simple inverse divergence $\div_\xi^{-1} = \nabla_\xi \Delta_\xi^{-1}$.} $h_0 := \div^{-1}_\xi g^{(l_0)}_{0,\boldsymbol{\alpha},i}$ and $h_1 := \div^{-1}_\xi g^{(l_0)}_{1,\boldsymbol{\alpha},i}$, then $\langle h_0 \rangle = \langle h_1 \rangle = 0$. We can deduce (dropping the arguments $(x,t,\lambda\Phi_i)$)
\begin{equation}\label{e:hom_b_comp:e7}
\begin{split}
    \frac{1}{\lambda} \div \big( \nabla\Phi_i^{-1} h_0 D_t \partial^{\boldsymbol{\alpha}} \bar\rho 
     +& \nabla\Phi_i^{-1} h_1 \partial^{\boldsymbol{\alpha}} \bar\rho \big) 
     = g^{(l_0)}_{0,\boldsymbol{\alpha},i} D_t \partial^{\boldsymbol{\alpha}} \bar\rho + g^{(l_0)}_{1,\boldsymbol{\alpha},i} \partial^{\boldsymbol{\alpha}} \bar\rho\\
     &+ \frac{1}{\lambda} \div_x \big( \nabla\Phi_i^{-1} h_0 \big) D_t \partial^{\boldsymbol{\alpha}} \bar\rho
     + \frac{1}{\lambda} \div_x \big( \nabla\Phi_i^{-1} h_1 \big) \partial^{\boldsymbol{\alpha}} \bar\rho\\
     &+ \frac{1}{\lambda}\nabla\Phi_i^{-1} h_0 \cdot D_t \partial^{\boldsymbol{\alpha}} \nabla\bar\rho
     + \frac{1}{\lambda}\nabla\Phi_i^{-1} h_1 \cdot \partial^{\boldsymbol{\alpha}} \nabla \bar\rho.
\end{split}
\end{equation}
Now let $f^{(l_0)}_{0,\boldsymbol{\alpha},i} = \nabla\Phi_i^{-1} h_0$ and $f^{(l_0)}_{1,\boldsymbol{\alpha},i} = \nabla\Phi_i^{-1} h_1$, corresponding to the left hand side of the first line in \eqref{e:hom_b_comp:e7}. For \eqref{e:hom_b_comp:e5} to hold, the second and the third lines in \eqref{e:hom_b_comp:e7} give the formulas for $g^{(l_0+1)}_{0,\boldsymbol{\alpha},i}$ and $g^{(l_0+1)}_{1,\boldsymbol{\alpha},i}$. Because $h_0$ and $h_1$ have zero mean in $\xi$, $g^{(l_0+1)}_{0,\boldsymbol{\alpha},i}$ and $g^{(l_0+1)}_{1,\boldsymbol{\alpha},i}$ also have zero mean in $\xi$. 
\par 

The estimates \eqref{e:hom_trho_f0g0_est} and \eqref{e:hom_trho_f1g1_est} for $l_0+1$ on $g^{(l_0+1)}_{0,\boldsymbol{\alpha},i}$ and $g^{(l_0+1)}_{1,\boldsymbol{\alpha},i}$ are similar to \eqref{e:hom_trho_c_est} and \eqref{e:hom_trho_c_est1}. They follow from \eqref{e:hom_assump:chi}, \eqref{e:hom_assump:sigma} and \eqref{e:hom_assump:slow}. Indeed, if no derivative $\nabla_x$ hits $\eta_i$, we get the first term in the right hand side of \eqref{e:hom_trho_f0g0_est} and \eqref{e:hom_trho_f1g1_est}. If at least one derivative $\nabla_x$ hits $\eta_i$, we get the second term in \eqref{e:hom_trho_f0g0_est} and \eqref{e:hom_trho_f1g1_est} due to the fact $|\nabla^m \eta_i| \lesssim 1$ for any $m \leq N_h$ with the constants in $\lesssim$ depending on $N_h$. Also, for these estimates, we get contributions from terms like $\nabla\Phi_i^{-1}$ and $\nabla\Phi_i$ as in \eqref{e:hom_b_comp:e7}. These contributions are also controlled by the constants in $\lesssim$ depending on $N_h$.
This concludes our proof.
\end{proof}
We are ready to prove  Proposition \ref{p:hom_trho} and  Corollaries \ref{c:hom_trho}, \ref{c:hom_dissipation}.

\begin{proof}[Proof of Proposition \ref{p:hom_trho} and Corollary \ref{c:hom_trho}
]
In this proof, for $c_j^{(i)}$, $\chi_i$, $f^{(l)}_{0,\boldsymbol{\alpha},i}$ and $g^{(l)}_{0,\boldsymbol{\alpha},i}$, we omit the argument $( x, t, \lambda \Phi_i )$ as usual.
Test \eqref{e:hom_trho_eq_detail} with $\tilde \rho$. By H\"older's inequality and Young's inequality, we have 
\begin{equation} \label{p:hom_trho:p:e1}
\begin{split}
   \sup_{t \le T} \| \tilde \rho(t) \|^2_{L^2} + \kappa \int_0^T \| \nabla \tilde \rho \|^2_{L^2} dt
        \lesssim& \frac{1}{\lambda^2 \kappa} \bigg( \sum_{l=1}^4 \int_0^T \|E_l\|^2_{H^{-1}} dt + \sum_{l=1}^N \int_0^T \|F_l\|^2_{H^{-1}} dt + \int_0^T \|G_N\|^2_{L^2} dt \bigg) \\ 
        &+ \frac{1}{\lambda^2} \| (\chi \cdot \nabla \bar \rho) (0)\|^2_{L^2}.
\end{split}
\end{equation}
In order to estimate the right-hand side terms involving $E_l$, we invoke estimates \eqref{e:rhobarenerg} and \eqref{e:rhobarflowB} for $\bar \rho$ from Lemma \ref{l:energy_weighted}. To verify the assumptions of Lemma \ref{l:energy_weighted}, note that \eqref{e:e_modelAestimates} and \eqref{e:e_model-estimates2} follow from  \eqref{e:hom_para_assump:taukappaell}, \eqref{e:hom_assump:barkappa}, \eqref{e:hom_assump:u} and direct computations, whereas \eqref{e:e_ell} and \eqref{e:e_kappatransported} follow from \eqref{e:hom_assump:barkappa}, \eqref{e:hom_assump:sigma} and \eqref{e:hom_assump:slow}. Therefore, with $\mathcal{D}_l$ defined in \eqref{e:hom_trho_est}, we have
\begin{equation} \label{p:hom_trho:p:e3}
\begin{split}
   \sup_{t\le T}  \|(\bar\kappa^\frac{m-1}{2} \nabla^{m-1}\bar\rho)(t)\|^2_{L^2} + \int_0^T \| \bar \kappa^{\frac{m}{2}} \nabla^{m} \bar \rho \|_{L^2}^2 dt
        \lesssim& \tau^{-(m-1)} \mathcal{D}_{m-1}, \\ 
    \int_0^T \| \bar \kappa^{\frac{m}{2}} D_t \nabla^{m} \bar \rho \|_{L^2}^2 dt
        \lesssim& \tau^{-(m+1)} \mathcal{D}_{m+1}.
\end{split}
\end{equation}
Observe that \eqref{e:hom_trho_c_est} and \eqref{e:hom_trho_c_est1}, together with \eqref{e:hom_para_assump:deltaell} and \eqref{e:hom_assump:sigmadelta}, give for $m \in \{0,1\}$
\begin{align} \label{p:hom_trho:p:e4}
    \| \bar \kappa^{-1} \nabla_x^m c^{(i)}_j \|_{L^\infty}
        &\lesssim \ell^{-m},
\end{align}
\par 
using \eqref{e:hom_assump:chi} and \eqref{e:hom_base:tA} we obtain for $n \in \{0,1\}$, $m \in \{0,1\}$,\begin{align} \label{p:hom_trho:p:e5}
    \| \bar \kappa^{-1} \tilde A D_t^n\nabla_x^m \chi \|_{L^\infty}
        &\lesssim \tau^{-n}\ell^{-m}.
\end{align}
Thanks to \eqref{p:hom_trho:p:e4} and \eqref{p:hom_trho:p:e5}, the estimates for $E_1$ and $E_3$ are the same. By the same token, the estimates for $E_2$ and $E_4$ are also the same. Here, we estimate $E_1$ and $E_2$ as examples. Omitting the argument $(x,t,\lambda \Phi_i)$, we have, using \eqref{p:hom_trho:p:e3} and $|\nabla\Phi_i|<2$
\begin{equation} \label{p:hom_trho:p:e6}
\begin{split}
    \int_0^T \|E_1\|^2_{H^{-1}} dt
    \lesssim&  \| \bar \kappa^{-1} c^{(i)}_j \|^2_{L^\infty} \int_0^T \| \bar \kappa \nabla^2 \bar \rho \|^2_{L^2} dt \lesssim \tau^{-1} \mathcal{D}_1, \\ 
    \int_0^T \|E_2\|^2_{H^{-1}} dt
    \lesssim&  \|\bar \kappa^{\frac{1}{2}}\|^2_{L^\infty} \| \bar \kappa^{-1} \nabla_x c^{(i)}_j \|^2_{L^\infty}  \int_0^T \| \bar \kappa^{\frac{1}{2}} \nabla \bar \rho \|^2_{L^2} dt \\ 
    \lesssim& \|\bar \kappa\|_{L^\infty} \ell^{-2} \mathcal{D}_0
    \lesssim_{\eqref{e:hom_para_assump:taukappaell}} \tau^{-1} \mathcal{D}_0.
\end{split}
\end{equation}
Next, we estimate $F_l$ and $G_{N_h}$. In the rest of this proof, notice the fact $|\nabla\Phi_i|<2$, contributing to a factor finally depending on $N_h$. Note that from \eqref{e:hom_trho_f0g0_est}, \eqref{e:hom_trho_f1g1_est}, \eqref{e:hom_assump:barkappa}, \eqref{e:hom_para_assump:deltaell} and \eqref{e:hom_assump:sigmadelta}, we have
\begin{equation} \label{p:hom_trho:p:e8}
\begin{split}
    \big\| \bar \kappa^{-\frac{1}{2}} f^{(l)}_{0,\boldsymbol{\alpha},i} \big\|_{L^\infty}
    + \big\| \bar \kappa^{-\frac{1}{2}} g^{(l)}_{0,\boldsymbol{\alpha},i} \big\|_{L^\infty}
        &\lesssim  \kappa^{-\frac{1}{2}} \ell^{-(l-|\boldsymbol{\alpha}|)}, \\ 
    \big\| \bar \kappa^{-\frac{1}{2}} f^{(l)}_{1,\boldsymbol{\alpha},i} \big\|_{L^\infty}
    + \big\| \bar \kappa^{-\frac{1}{2}} g^{(l)}_{1,\boldsymbol{\alpha},i} \big\|_{L^\infty}
        &\lesssim  \kappa^{-\frac{1}{2}} \tau^{-1} \ell^{-(l-|\boldsymbol{\alpha}|)}.
\end{split}
\end{equation}
\par 

Note that $F_1 = 0$. For $F_l$ with $l \geq 2$, using \eqref{p:hom_trho:p:e3}, \eqref{p:hom_trho:p:e8} and the support condition \eqref{e:hom_trho_fg_supp}, we have
\begin{equation} \label{p:hom_trho:p:e9}
\begin{split}
    \int_0^T \|F_l\|_{H^{-1}}^2 dt
    \lesssim& \frac{1}{\lambda^{2(l-1)}} \sum_{m=1,|\boldsymbol{\alpha}|=m}^{l-1} \sum_i \bigg( \| \bar \kappa^{-(m-1)} \|_{L^\infty} \| \bar \kappa^{-\frac{1}{2}} f^{(l-1)}_{0,\boldsymbol{\alpha},i} \|^2_{L^\infty} \int_0^T \| \bar \kappa^{\frac{m}{2}} D_t \nabla^m \bar \rho \|^2_{L^2} dt \\ 
        &+ \| \bar \kappa^{-(m-1)} \|_{L^\infty} \| \bar \kappa^{-\frac{1}{2}} f^{(l-1)}_{1,\boldsymbol{\alpha},i} \|^2_{L^\infty} \int_0^T \| \bar \kappa^{\frac{m}{2}} \nabla^m \bar \rho \|^2_{L^2} dt \bigg) \\ 
    &\lesssim \frac{1}{\lambda^{2(l-1)}} \sum_{m=1}^{l-1} \| \bar \kappa^{-1} \|_{L^\infty}^{m-1} \kappa^{-1} \ell^{-2(l-1-m)} \tau^{-(m+1)} \mathcal{D}_{m+1} \\ 
    &\lesssim (\lambda \ell)^{-2(l-1)} \kappa \kappa^{-1} \sum_{m=1}^{l-1} \big( \tau \kappa \ell^{-2} \big)^{-m} \tau^{-1} \mathcal{D}_{m+1} \\ 
    &\overset{\eqref{e:hom_para_assump:taukappaell}, m=l-1}{\lesssim} (\lambda \ell)^{-2(l-1)} \big( \tau \kappa \ell^{-2} \big)^{-(l-1)} \tau^{-1} \mathcal{D}_{l} \\ 
    &\lesssim \big( \tau \kappa \lambda^2 \big)^{-(l-1)} \tau^{-1} \mathcal{D}_{l} \\
    &\overset{\eqref{e:hom_para_assump:Nh_result}, l \geq 2}{\lesssim}
    \tau^{-1} \mathcal{D}_{l}.
\end{split}
\end{equation}

For $G_{N_h}$, using \eqref{p:hom_trho:p:e3} and \eqref{p:hom_trho:p:e8} and the support condition \eqref{e:hom_trho_fg_supp}, we have
\begin{equation} \label{p:hom_trho:p:e10}
\begin{split}
    \int_0^T \|G_{N_h}\|_{L^2}^2 dt
    &\lesssim \frac{1}{\lambda^{2(N_h-1)}} \sum_{m=1,|\boldsymbol{\alpha}|=m}^{N_h} \sum_i \bigg( \| \bar \kappa^{-(m-1)} \|_{L^\infty} \| \bar \kappa^{-\frac{1}{2}} g^{(N_h)}_{0,\boldsymbol{\alpha},i} \|^2_{L^\infty} \int_0^T \| \bar \kappa^{\frac{m}{2}} D_t \nabla^m \bar \rho \|^2_{L^2} dt \\ 
        &+ \| \bar \kappa^{-(m-1)} \|_{L^\infty} \| \bar \kappa^{-\frac{1}{2}} g^{(N_h)}_{1,\boldsymbol{\alpha},i} \|^2_{L^\infty} \int_0^T \| \bar \kappa^{\frac{m}{2}} \nabla^m \bar \rho \|^2_{L^2} dt \bigg) \\ 
    &\lesssim \lambda^{-2(N_h-1)} \sum_{m=1}^{N_h} \| \bar \kappa^{-1} \|_{L^\infty}^{m-1}
        \kappa^{-1} \ell^{-2(N_h-m)} \tau^{-(m+1)} \mathcal{D}_{m+1} \\ 
    &\lesssim (\lambda \ell)^{-2(N_h-1)} \kappa \kappa^{-1} \ell^{-2}
    \sum_{m=1}^{N_h} \big( \tau \kappa \ell^{-2} \big)^{-m} \tau^{-1} \mathcal{D}_{m+1} \\ 
    &\overset{\eqref{e:hom_para_assump:taukappaell}, m=N_h}{\lesssim} (\lambda \ell)^{-2(N_h-1)} \ell^{-2} \big( \tau \kappa \ell^{-2} \big)^{-N_h} \tau^{-1} \mathcal{D}_{N_h+1} \\
    &\lesssim \lambda^2 \big( \tau \kappa \lambda^2 \big)^{-N_h} \tau^{-1} \mathcal{D}_{N_h+1} \\
    &\overset{\eqref{e:hom_para_assump:Nh_result}}{\lesssim}
 \tau^{-1} \mathcal{D}_{N_h+1}.
\end{split}
\end{equation}
 For the last term in \eqref{p:hom_trho:p:e1} involving initial data, we have
\begin{align} \label{p:hom_trho:p:e7}
    \| (\chi \nabla \bar \rho) (0)\|^2_{L^2}
        \lesssim& \| \bar \kappa^{-\frac{1}{2}} \chi \|_{L^\infty}^2 \| \bar \kappa^{\frac{1}{2}} \nabla \rho_{\text{in}} \|_{L^2}^2
        \lesssim \kappa^{-1} \| \bar \kappa^{\frac{1}{2}} \nabla \rho_{\text{in}} \|_{L^2}^2
        \lesssim \kappa^{-1}\tau^{-1} \mathcal{D}_1.
\end{align}
This finishes the proof of Proposition \ref{p:hom_trho}.

For the proof of Corollary \ref{c:hom_trho}, observe that in addition to the already obtained estimates, we have to control only the corrector term, i.e. the term of \eqref{e:hom_trho_ansatz} involving $\chi$. Its estimate is analogous to \eqref{p:hom_trho:p:e7} via \eqref{p:hom_trho:p:e3}. This concludes our proofs.
\end{proof}

\begin{proof}[Proof of Corollary \ref{c:hom_dissipation}]
To simplify the formulas, we introduce the following shorthand notation
\begin{align}    \label{c:hom_dissipation_eq2}
    M(x,t,\xi) :=& \Id + \sum_{i}\eta_i(x,t)\nabla\Phi_i^T( x, t) \nabla_{\xi} \chi_i^T\big( x, t, \xi\big)
\end{align}
and denote by $\tilde M (x,t)$ the matrix $M (x,t,\lambda \Phi_i)$ (analogously to $A$ and $\tilde A$, cf. \eqref{e:hom_base:A} and \eqref{e:hom_base:tA}).
We first claim the following estimates
\begin{align} 
    \left| \kappa \int_0^T \int \big( \tilde M^T \tilde M - \langle \tilde M^T\tilde M \rangle \big) \nabla \bar \rho \cdot \nabla \bar \rho \right|
        \lesssim&  \Big( \frac{1}{\lambda \ell} 
        + \frac{1}{ \lambda \kappa^{\sfrac{1}{2}} \tau^{\sfrac{1}{2}} } \Big) \mathcal{D}_1 ,   \label{c:hom_dissipation:claim1} \\ 
    \frac{\kappa}{\lambda^2} \int_0^T \| \nabla^2\bar\rho \chi \|^2_{L^2} + \| \nabla_x \chi^T \nabla\bar{\rho} \|^2_{L^2}
        \lesssim& \frac{1}{\lambda^2 \kappa \tau} \mathcal{D}_1. \label{c:hom_dissipation:claim2}
\end{align}
Before proving these claims, let us conclude the proof of the Corollary \ref{c:hom_dissipation}.
First, direct computations give
\begin{align*} 
    \tilde M^T \tilde M &= \Id + \sum_i\eta_i(\nabla\Phi_i^T \nabla_{\xi} \chi_i^T + \nabla_{\xi} \chi_i \nabla\Phi_i) + \sum_i\eta_i^2\nabla_{\xi} \chi_i \nabla\Phi_i \nabla\Phi_i^T \nabla_{\xi} \chi_i^T,  \\
    \tilde M^T\tilde M - \langle \tilde M^T\tilde M \rangle
        &= \sum_i\eta_i(\nabla\Phi_i^T \nabla_{\xi} \chi_i^T + \nabla_{\xi} \chi_i \nabla\Phi_i) +\\
  &+\sum_i\eta_i^2(\nabla_{\xi} \chi_i \nabla\Phi_i \nabla\Phi_i^T \nabla_{\xi} \chi_i^T - \langle \nabla_{\xi} \chi_i \nabla\Phi_i \nabla\Phi_i^T \nabla_{\xi} \chi_i^T \rangle).
\end{align*}
Furthermore, by using \eqref{e:hom_ed:chi} and the properties of Mikado flows in Section \ref{s:Mikado}, we compute
\begin{align*}
    \kappa\nabla_\xi\chi_i\nabla\Phi_i\nabla\Phi_i^T\nabla_\xi\chi_i^T &= \frac{\sigma}{\kappa\lambda^2} \sum_{\vec{k} \in \Lambda} a_{\vec{k}}^2\nabla\Phi_i^{-1}(\vec{k}\otimes\nabla_\xi\varphi_{\vec{k}})\nabla\Phi_i\nabla\Phi_i^T(\nabla_\xi\varphi_{\vec{k}}\otimes\vec{k})\nabla\Phi_i^{-T}\\
    &=\frac{\sigma}{\kappa\lambda^2} \sum_{\vec{k} \in \Lambda} a_{\vec{k}}^2\nabla\Phi_i^{-1}|\nabla\Phi_i^T\nabla_\xi\varphi_{\vec{k}}|^2 (\vec{k} \otimes \vec{k}) \nabla\Phi_i^{-T}.
\end{align*}
In particular, comparing with the derivation of $\bar{A}$, we deduce that
\begin{align} \label{c:hom_dissipation_eq3}
    \big| \kappa \langle \tilde M^T\tilde M \rangle - \bar A \big| 
        & \lesssim \sum_i \tilde \eta_i \|\nabla\Phi_i-\Id\|_{L^\infty}|\bar{A}|\lesssim \lambda^{-2\gamma} |\bar A|.
\end{align}
Here we use \eqref{e:hom_ed:chi}, \eqref{e:hom_assump:Phi}, \eqref{e:hom_assump:barkappa} and the formula for the homogenized matrix $\bar A$ (compare the first line of the formula \eqref{e:defbarA}).

 Next, using \eqref{l:hom_trho_comp:e3} and omitting $(x,t,\lambda \Phi_i)$, we have for an absolute constant $C>1$
\begin{align*}
    \kappa &\int_0^T \| \nabla \rho\|_{L^2}^2 dt 
        = \kappa \int_0^T \Big\| \tilde M \nabla \bar \rho + \frac{1}{\lambda} (\nabla^2\bar\rho \chi + \nabla_x \chi^T \nabla\bar{\rho}) + \nabla \tilde\rho \Big\|_{L^2}^2 dt \\
    \leq& \Big( 1+\frac{1}{\lambda \kappa^{\sfrac{1}{2}} \tau^{\sfrac{1}{2}} } \Big) \kappa \int_0^T \| \tilde M \nabla \bar \rho \|_{L^2}^2 dt + \big( 1 + C\lambda \kappa^{\sfrac{1}{2}} \tau^{\sfrac{1}{2}} \big) \kappa \int_0^T \| \nabla \tilde\rho \|_{L^2}^2 dt \\
        & + \big( 1 + C\lambda \kappa^{\sfrac{1}{2}} \tau^{\sfrac{1}{2}} \big) \frac{\kappa}{\lambda^2} \int_0^T \| \nabla^2\bar\rho \chi \|^2_{L^2} + \| \nabla_x \chi^T \nabla\bar{\rho} \|^2_{L^2} dt \\ 
    = & \Big( 1+\frac{1}{\lambda \kappa^{\sfrac{1}{2}} \tau^{\sfrac{1}{2}} } \Big) \kappa \int_0^T \int \big( \tilde M^T\tilde M - \langle \tilde M^T\tilde M \rangle \big) \nabla \bar \rho \cdot \nabla \bar \rho dxdt 
        + \Big( 1+\frac{1}{\lambda \kappa^{\sfrac{1}{2}} \tau^{\sfrac{1}{2}} } \Big) \int_0^T \int \bar A \nabla \bar \rho \cdot \nabla \bar \rho dxdt \\ 
        &+ \Big( 1+\frac{1}{\lambda \kappa^{\sfrac{1}{2}} \tau^{\sfrac{1}{2}} } \Big) \int_0^T \int \big( \kappa \langle \tilde M^T\tilde M \rangle - \bar A \big) \nabla \bar \rho \cdot \nabla \bar \rho dxdt 
        + \big( 1 + C\lambda \kappa^{\sfrac{1}{2}} \tau^{\sfrac{1}{2}} \big) \kappa \int_0^T \int \| \nabla \tilde\rho \|_{L^2}^2 dt  \\
        &+ 
        \big( 1 + C\lambda \kappa^{\sfrac{1}{2}} \tau^{\sfrac{1}{2}} \big) \frac{\kappa}{\lambda^2} \int_0^T \| \nabla^2\bar\rho \chi \|^2_{L^2} + \| \nabla_x \chi^T \nabla\bar{\rho} \|^2_{L^2} dt \\ 
    \leq& \Big( \frac{C}{\lambda \kappa^{\sfrac{1}{2}} \tau^{\sfrac{1}{2}} } + \frac{1}{\lambda \ell} + \lambda^{-2\gamma} \Big) \mathcal{D}_{N_h}
        +\Big( 1+\frac{C}{\lambda \kappa^{\sfrac{1}{2}} \tau^{\sfrac{1}{2}} } \Big) \int_0^T \int \bar A \nabla \bar \rho \cdot \nabla \bar \rho dxdt,
\end{align*}
where for the last line we used, respectively, \eqref{c:hom_dissipation:claim1}, \eqref{c:hom_dissipation_eq3}, Proposition \ref{p:hom_trho} and \eqref{c:hom_dissipation:claim2}. Now we can conclude the proof of Corollary \ref{c:hom_dissipation} with
\begin{align*}
    \left| \kappa \int_0^T \| \nabla \rho\|_{L^2}^2 dt- \int_0^T \int \bar A \nabla \bar \rho \cdot \nabla \bar \rho dxdt \right|
    \lesssim \Big( \frac{1}{\lambda \kappa^{\sfrac{1}{2}} \tau^{\sfrac{1}{2}} } + \frac{1}{\lambda \ell} + \lambda^{-2\gamma} \Big)\mathcal{D}_{N_h}.
\end{align*}
\par

It remains to prove the claims \eqref{c:hom_dissipation:claim1}, \eqref{c:hom_dissipation:claim2}. The proof of \eqref{c:hom_dissipation:claim2} is the same as estimates for the terms $E_3$ and $E_4$ in the proof of Proposition \ref{p:hom_trho}, so we omit the details. For the claim \eqref{c:hom_dissipation:claim1}, we find a matrix potential $\omega := \sum_i \eta_i \omega^{(i)}$ taking arguments $(x,t,\xi)$ such that
\begin{align}
    \big( M^TM - \langle M^TM \rangle \big)_{jl} =&  \Delta_\xi \omega_{jl} 
        = \sum_i \eta_i \Delta_\xi \omega_{jl}^{(i)}, \label{c:hom_dissipation_eq6a}  \\ 
    \big\| \nabla_\xi \omega_j \big\|_{L^\infty_\xi} 
    \lesssim& \sum_i\Big( \frac{\sigma \eta_i^2}{\kappa^2 \lambda^2} + \frac{\sigma^{1/2} \eta_i}{\kappa \lambda} \Big)
    \lesssim \sum_i \Big( 1 + \frac{\sigma \eta_i^2}{\kappa^2 \lambda^2} \Big)
    \lesssim \frac{ \bar \kappa } {\kappa}, \label{c:hom_dissipation_eq7} \\ 
    \big\| \nabla_x \nabla_\xi \omega_j \big\|_{L^\infty_\xi} 
    \lesssim& \sum_i \Big( \frac{\sigma \eta_i}{\kappa^2 \lambda^2} + \frac{\sigma^{1/2} }{\kappa \lambda} \Big) ( \eta_i\ell^{-1} + 1 )
    \lesssim \sum_i \Big( 1 + \frac{\sigma \eta_i^2}{\kappa^2 \lambda^2} \Big) \ell^{-1}
    \lesssim \frac{ \bar \kappa } {\kappa \ell}. \label{c:hom_dissipation_eq8}
\end{align}
The estimates \eqref{c:hom_dissipation_eq7} and \eqref{c:hom_dissipation_eq8} can be proved in the same way as the estimates \eqref{e:hom_trho_c_est} and \eqref{e:hom_trho_c_est1} in Lemma \ref{l:hom_b_comp}. Now by the chain rule, for any $i,j,l$, we have
\begin{equation}
\label{c:hom_dissipation_eq6b}
    (\Delta_\xi \omega_{jl}^{(i)})(x,t,\lambda\Phi_i)
    = \frac{1}{\lambda} \div \bigl( \nabla\Phi_i^{-1} \nabla_\xi \omega_{jl}^{(i)} (x,t,\lambda\Phi_i) \bigr)
    - \frac{1}{\lambda} \div_x( \nabla\Phi_i^{-1} \nabla_\xi \omega_{jl}^{(i)}) (x,t,\lambda\Phi_i).
\end{equation}
Consider the left-hand side of the wanted \eqref{c:hom_dissipation:claim1} and use the formulas \eqref{c:hom_dissipation_eq6a}, \eqref{c:hom_dissipation_eq6b}, omitting the argument $(x,t,\lambda \Phi_i)$,  to write 
\begin{equation} \label{c:hom_dissipation_eq9}
\begin{split}
    \kappa \int_0^T & \int \big( (\tilde M^T\tilde M) - \langle \tilde M^T\tilde M \rangle \big) \nabla \bar \rho \cdot \nabla \bar \rho dxdt \\
    =& \frac{\kappa}{\lambda} \sum_{i,j,l} \Big( \int_0^T \int \eta_i \div \bigl( \nabla\Phi_i^{-1} \nabla_\xi \omega_{jl}^{(i)} \bigr) \partial_j \bar \rho \partial_l \bar \rho dxdt - \int_0^T \int \eta_i \div_x \bigl( \nabla\Phi_i^{-1} \nabla_\xi \omega_{jl}^{(i)} \bigr) \partial_j \bar \rho \partial_l \bar \rho dxdt \Big) \\ 
    =& \frac{\kappa}{\lambda} \sum_{i,j,l} \bigg( - \int_0^T \int \bigl( \nabla\Phi_i^{-1} \nabla_\xi \omega_{jl}^{(i)} \bigr) \cdot \nabla_{x} \big( \eta_i \partial_j \bar \rho \partial_l \bar \rho \big) dxdt - \int_0^T \int \eta_i \div_x \bigl( \nabla\Phi_i^{-1} \nabla_\xi \omega_{jl}^{(i)} \bigr) \partial_j \bar \rho \partial_l \bar \rho dxdt \bigg).
\end{split}
\end{equation}
Now, we estimate the terms on the right-hand side of \eqref{c:hom_dissipation_eq9} using \eqref{c:hom_dissipation_eq7}, \eqref{e:hom_assump:slow}, and Young's inequality with the splitting $\bar \kappa | \nabla \bar \rho|| \nabla^2 \bar \rho| = (\tau^{-\sfrac{1}{4}} \bar \kappa^{\sfrac{1}{4}} | \nabla \bar \rho|) ( \tau^{\sfrac{1}{4}} \bar \kappa^{\sfrac{3}{4}} | \nabla^2 \bar \rho|)$. Consequently we obtain
\begin{align*}
  &\left|  \kappa \int_0^T \int \big( (\tilde M^T\tilde M) - \langle \tilde M^T\tilde M \rangle \big) \nabla \bar \rho \cdot \nabla \bar \rho dxdt \right| \\
&\overset{{\eqref{c:hom_dissipation_eq7}, \eqref{e:hom_assump:slow}} }{\lesssim} \frac{1}{\lambda \ell} \int_0^T \int \bar \kappa | \nabla \bar \rho|^2 dxdt 
        + \frac{1}{\lambda} \int_0^T \int \bar \kappa | \nabla \bar \rho|| \nabla^2 \bar \rho| dxdt \\ 
    &\lesssim \frac{1}{\lambda \ell} \int_0^T \int \bar \kappa | \nabla \bar \rho|^2 dxdt + \frac{1}{\lambda \kappa^{\sfrac{1}{2}} \tau^{\sfrac{1}{2}} } \int_0^T \int \bar \kappa | \nabla \bar \rho|^2 dxdt + \frac{\tau^{\sfrac{1}{2}}}{\lambda\kappa^{\sfrac{1}{2}}} \int_0^T \int | \bar \kappa \nabla^2 \bar \rho|^2 dxdt \\ 
    &\overset{\eqref{p:hom_trho:p:e3}}{\lesssim} \Big( \frac{1}{\lambda \ell} 
        + \frac{1}{ \lambda \kappa^{\sfrac{1}{2}} \tau^{\sfrac{1}{2}} } \Big) \mathcal{D}_1, 
\end{align*}
which is the estimate \eqref{c:hom_dissipation:claim1}.
\end{proof}

\section{Time averaging}\label{s:time} 
The main result of this section is Proposition \ref{p:t_avg}, which allows to prove the time averaging step of the key Proposition \ref{p:main}.

\subsection{Setup}
Consider constants $\kappa_1>\kappa_0>0$, and the oscillatory $\tilde\kappa=\tilde\kappa(x,t,\tau^{-1}t)$ of the form
\begin{equation}\label{e:t-osc-tildekappa}
	\tilde\kappa(x,t,\tau^{-1}t)=\kappa_0+\kappa_1\,\eta(x,t,\tau^{-1}t),
\end{equation}
where $s\mapsto\eta(x,t,s)$ is a $1$-periodic, nonnegative and smooth function with 
\begin{equation*}
 \langle\eta\rangle:=\int_0^1\eta(x,t,s)\,ds=1\,.
\end{equation*}
We write
\begin{equation}\label{e:t_defg}
\kappa:=\kappa_0+\kappa_1,\quad g(x,t,s):=\frac{\kappa_1}{\kappa}(\eta(x,t,s)-1),\quad g_\tau(x,t):=g(x,t,\tau^{-1}t),	
\end{equation}
so that $\langle g\rangle=0$, $g$ is bounded, and we have the identity
\begin{equation}\label{e:t-osc-g}
\tilde\kappa(x,t,\tau^{-1}t)=\kappa_0+\kappa_1 \eta(x,t,\tau^{-1}t)=\kappa+\kappa g_\tau(x,t).
\end{equation}

In this section we consider the solutions $\rho$ and $\bar{\rho}$ of the following equations
\begin{equation}\label{e:equ-t-oscillating}
\begin{split}
	\partial_t\rho+u\cdot\nabla \rho&=\div\left(\tilde\kappa\nabla\rho\right),\\
	\rho|_{t=0}&=\rho_{in},
	\end{split}
\end{equation}
and 
\begin{equation}\label{e:equ-t-average}
\begin{split}
	\partial_t\bar{\rho}+u\cdot\nabla \bar{\rho}&=\div\left(\kappa\nabla\bar{\rho}\right),\\
	\bar\rho|_{t=0}&=\rho_{in}.
	\end{split}
\end{equation}

In view of definitions of $\kappa$, $\tilde \kappa$, the equation \eqref{e:equ-t-average} can be seen as the time-averaged version of \eqref{e:equ-t-oscillating}. 
Further, we define the respective total dissipations 
\begin{equation}\label{e:t-osc-D}
\tilde D:=\int_0^T\|\tilde\kappa^{\sfrac12}\nabla\rho\|_{L^2}^2\,dt\,,\quad D:=\kappa\int_0^T\|\nabla\bar\rho\|_{L^2}^2\,dt.	
\end{equation}
We make the following assumptions: 

There exists $\mu \ge \|\nabla u\|_{\infty}$ (we should think of $\mu^{-1} \sim \|\nabla u\|_{\infty}^{-1}$ being the advection time-scale) such that
\begin{enumerate}
\item[(A1)] The fast time-scale is shorter than the advection time-scale, i.e. 
\begin{equation}\label{Atime_timeLadv}
\tau \mu <1/2.
\end{equation}
In relation to \eqref{Atime_timeLadv}, we fix $N\in\N$ such that
\begin{equation}
(\tau\mu )^{N-1}\kappa<\kappa_0.	
\end{equation}
\item[(A2)] (Control of higher-order spatial derivatives) For any $n\leq N$ 
\begin{equation}\label{Atime_scales}
\|\nabla u\|_{C^n}^2\leq C_N\left(\frac{\mu}{\kappa}\right)^n\mu^2.	
\end{equation}
\item[(A3)] (Bound on the initial data) For any $1\leq n\leq N$
\begin{equation}\label{Atime_Ini}
\|\rho_{in}\|_{H^n}^2\leq C_N \left(\frac{\mu}{\kappa}\right)^nD. 	
\end{equation}
\item[(A4)] (Control of cutoff and its slow derivatives) For any $1\leq n\leq N$
\begin{equation}\label{e:Ggcon}
\| \nabla^n \eta \|_{L^\infty} \leq C_N (\tau\mu )^2  \left(\frac{\mu}{\kappa}\right)^\frac{n}{2}, \qquad \| D^{slow}_t \eta \|_{L^\infty} \leq C_N \mu,
\end{equation}
\end{enumerate}
where for a function $\eta(x,t,s)$ we define its slow advective derivative as follows
\begin{equation}
    D^{slow}_t \eta(x,t,s) := \partial_t \eta(x,t,s) + u \cdot \nabla \eta(x,t,s).
\end{equation}
\subsection{Time averaging proposition}
\begin{proposition}\label{p:t_avg}
	Under the assumptions (A1), (A2), (A3), (A4) we have
	\begin{equation}\label{e:osc-t-proximity}
		\sup_{t \le T}\|\rho(t)-\bar{\rho}(t)\|_{L^2}^2+\kappa\int_0^T\|\nabla\rho-\nabla\bar{\rho}\|^2_{L^2}\,dt\lesssim (\tau\mu )^2D.
	\end{equation}
Moreover, the total dissipation satisfies
\begin{equation}\label{e:osc-t-dissipation}
|D-\tilde D|\lesssim (\tau\mu )(D+\tilde D).	
\end{equation}
The implicit constants depend only on $N$, $C_N$ in assumptions, and on $\|\eta \|_{L^\infty}$.
\end{proposition}

\begin{proof}
Starting with $\rho^{(0)}:=\bar{\rho}$ being the solution of \eqref{e:equ-t-average}, we construct the solution $\rho$ of \eqref{e:equ-t-oscillating} by successive approximations $\rho^{(i)}$, $i=1,2,\dots,N$, defined to be solutions of 
\begin{equation}\label{e:time_setup}
\begin{split}
	\partial_t\rho^{(i)}+u\cdot\nabla \rho^{(i)}&=\kappa\Delta\rho^{(i)}+\kappa\div\left(g_\tau\nabla\rho^{(i-1)}\right), \\
	\rho|_{t=0}&=\rho_{in}.
	\end{split}
\end{equation}
We will show in Lemma \ref{l:improved-energy} the improved energy bound 
\begin{equation}\label{eq:tavg_ieb}
	\sup_{t \le T} \|\rho^{(i)}(t)-\rho^{(i-1)}(t)\|_{L^2}^2+\kappa\int_0^T\|\nabla\rho^{(i)}-\nabla\rho^{(i-1)}\|^2_{L^2}\,dt\leq 
C_N(\tau \mu)^{2i}D,	
\end{equation}
see \eqref{e:time_main_impr} for $n=0$.
Setting
\begin{equation*}
\rho^{(error)}:=\rho-\rho^{(N)}	
\end{equation*}	
and using \eqref{e:equ-t-oscillating}, \eqref{e:time_setup} and \eqref{e:t-osc-g}, we may write the equation for $\rho^{(error)}$ as
\begin{equation*}
\begin{split}
	\partial_t\rho^{(error)}+u\cdot\nabla\rho^{(error)}&=\kappa\Delta\rho^{(error)}+\kappa\div (g_\tau\nabla\rho^{(error)})+\kappa\div (g_\tau (\nabla\rho^{(N)}-\nabla\rho^{(N-1)}))\,\\
	&= \kappa_0\Delta\rho^{(error)}+{\kappa_1}\div(\eta\nabla\rho^{(error)})+\kappa \div (g_\tau (\nabla\rho^{(N)}-\nabla\rho^{(N-1)})),
\end{split}
\end{equation*}
where for the second line we used $\kappa+\kappa g_\tau = \kappa_0+\kappa_1 \eta$.
Since $\eta\geq 0$ and $|g|\leq 2$, the standard energy estimate and an application of Young's inequality yields
\begin{equation*}
		\sup_{t \le T} \|\rho^{(error)}(t)\|_{L^2}^2+\kappa_0\int_0^T\|\nabla\rho^{(error)}\|_{L^2}^2\,dt\leq C\frac{\kappa}{\kappa_0}\kappa\int_0^T\|\nabla\rho^{(N)}-\nabla\rho^{(N-1)}\|_{L^2}^2\,dt.
\end{equation*}
Combining this with the improved energy bound \eqref{eq:tavg_ieb} and (A1) we deduce
\begin{equation}
\begin{split}
		\sup_{t \le T} \|\rho^{(error)}(t)\|_{L^2}^2+\kappa\int_0^T\|\nabla\rho^{(error)}\|_{L^2}^2\,dt&\leq C_N\left(\frac{\kappa}{\kappa_0}\right)^2(\tau\mu)^{2N}D \leq C_N(\tau\mu)^2D\,.
	\end{split}
\end{equation}
Consequently, since  $\rho^{(0)}:=\bar{\rho}$,
\begin{align*}
	\sup_{t \le T} \|\rho(t)-\bar{\rho}(t)\|_{L^2}^2+\kappa\int_0^T\|\nabla\rho-\nabla\bar{\rho}\|^2_{L^2}&\,dt\leq 2\|\rho(T)-\rho^{(N)}(T)\|_{L^2}^2+2\kappa\int_0^T\|\nabla\rho-\nabla\rho^{(N)}\|^2_{L^2}\,dt \\
+2^N\sum_{i=0}^{N-1}&\|\rho^{(i+1)}(T)-\rho^{(i)}(T)\|_{L^2}^2+\kappa\int_0^T\|\nabla\rho^{(i+1)}-\nabla\rho^{(i)}\|^2_{L^2}\,dt\\
&\lesssim  (\tau\mu)^2D\,,
\end{align*}
verifying \eqref{e:osc-t-proximity}.

\smallskip
Next, we consider the difference in total dissipation
\begin{equation*}
	D-\tilde D= \int_0^T\int_{\T^3}\kappa|\nabla\bar\rho|^2-\tilde\kappa|\nabla\rho|^2\,dx\,dt=\underbrace{-\kappa\int_0^T\int_{\T^3}g_\tau|\nabla\bar{\rho}|^2\,dx\,dt}_{(I)}+\underbrace{\int_0^T\int_{\T^3}\tilde\kappa(|\nabla\bar\rho|^2-|\nabla\rho|^2)\,dx\,dt}_{(II)}.	
\end{equation*} 
We will show in Lemma \ref{l:improved-energy}  (see \eqref{l:time_indR}) the bound
\begin{equation*}
|(I)|\lesssim (\tau\mu)D.
\end{equation*}
Indeed, in the notation introduced for Lemma \ref{l:improved-energy} (cf. \eqref{n:bili}), it holds $(I)=B^{(0)}(\nabla\rho^{(0)},\nabla\rho^{(0)})$. 
Next, we use Young's inequality and $\tilde \kappa \lesssim \kappa$ to estimate
\begin{align*}
|(II)|&=\left|\int_0^T\int_{\T^3}\tilde\kappa(\nabla\rho+\nabla\bar\rho)\cdot (\nabla\rho-\nabla\bar\rho)\,dx\,dt\right|\,\\
&\lesssim \varepsilon^{-1}\kappa\int_0^T\|\nabla\rho-\nabla\bar\rho\|_{L^2}^2\,dt+\varepsilon\left(\kappa\int_0^T\|\nabla\bar\rho\|_{L^2}^2\,dt+\int_0^T\|\tilde\kappa^{\sfrac12}\nabla\rho\|_{L^2}^2\,dt\right).
\end{align*}
Choosing $\varepsilon=\tau\mu$ and using \eqref{e:osc-t-proximity} we deduce
\begin{equation*}
|(II)|	\lesssim (\tau\mu)(D+\tilde D), 
\end{equation*}
verifying \eqref{e:osc-t-dissipation}.
\end{proof}

\subsection{Estimates for successive approximations}
This section is concerned with estimates for the successive approximations \eqref{e:time_setup}, resulting in the improved energy bound  \eqref{eq:tavg_ieb} and in \eqref{l:time_indR}, used in the proof of Proposition \ref{p:t_avg}. Since the improved energy bound may be of separate interest, we present it independently from the specific choices made in the previous section. To this end let us observe that the definition \eqref{e:t_defg} of $g$ and of $\eta$
imply that there exists the fast time potential i.e. $G (x,t,s)= \int_0^s g(x,t,a) da$ such that
\[
\pas G (x,t,s) = g(x,t,s).
\]
\subsubsection{Notation}
Let us denote $g_\tau (x,t) := g(x,t, \tau^{-1}t)$. We introduce \[g^{0}:=g, \; G^{0}:=G, \; \text{ and inductively } \; g^{l+1}:=G^{l} g^{0}, \; \pas G^{l} = g^{l}.\] 
Hence $ G^{l} = \frac{1}{(l+1)!} (G^{0})^{(l+1)}$.
Furthermore, let us write
\begin{equation}\label{n:D}
D := \kappa\int_0^T\|\nabla\rho_{0} (s)\|_{L^2}^2,
\end{equation}
\begin{equation}\label{n:pnorm}
\VE f \VE^2_n :=  \sup_{t \le T} \|(\nabla^n f)(t)\|^2_{L^2}+ \kappa\int_0^T \|\nabla^{n+1} f\|_{L^2}^2 dt + \left(\frac{\|\nabla u\|_{L^\infty}}{\kappa}\right)^n   \left(\kappa\int_0^T \|\nabla f\|_{L^2}^2 dt\right),
\end{equation}
\begin{equation}\label{n:bili}
B^l (f,h) := \kappa \int_0^T \int g^{l}_\tau f \cdot h.
\end{equation}

\subsubsection{Commutators}
We will need several estimates for commutators. Let $|\boldsymbol{\alpha}|=n$, then it holds
\begin{equation}\label{Atime_comm1}
\| [g \cdot \nabla, \pal] f  \|_{L^2} \lesssim_n \|  \nabla^n g  \|_{L^\infty}   \|  \nabla f \|_{L^2}  +  \|  \nabla g  \|_{L^\infty}   \|  \nabla^n f \|_{L^2},
\end{equation}
\begin{equation}\label{Atime_comm2}
          \|  [\pal , g ] \nabla f\|_{L^2} \lesssim_n  \|\nabla^{n} g \|_{L^\infty} \| \nabla f \|_{L^2}   +   \|\nabla g \|_{L^\infty} \| \nabla^{n} f \|_{L^2}. 
\end{equation}

\subsubsection{Improved energy bound}
Observe that
\begin{equation}\label{e:fastslow}
g_\tau^{l} = \tau D_t (G^{l}_\tau) -  \tau (D^{slow}_t G^{l})_\tau.
\end{equation}

Since $\frac{\kappa_1}{\kappa} \le 1$ and $\eta$ is bounded, Assumption (A4) and definitions yield via a direct computation that for  $f= g^l$ or $f= G^l$ we have\footnote{In fact, the estimates \eqref{Atime_Gspace} can be made $l$-independent, which results in estimate \eqref{l:time_indR} uniform in $l$. Since we use in Proposition \ref{p:t_avg} merely the case $l=0$ of \eqref{l:time_indR}, its dependence on $l$ plays no role.}
\begin{equation}\label{Atime_Gspace}
\begin{split}
    &\|  f  \|_{L^\infty} \lesssim_l 1, \qquad \| (D^{slow}_t f) \|_{L^\infty} \lesssim_l \mu,\\  
    &\text{and  for } n>0   \qquad  \| \nabla^n f\|_{L^\infty} \lesssim_{n,l}  (\tau \mu)^2  \left(\frac{\mu}{\kappa}\right)^\frac{n}{2}. 
\end{split}
\end{equation}

Consider 
\begin{equation}\label{e:time_setup2}
  \begin{split}
    \partial_t \rho^{(i)}+u \cdot\nabla \rho^{(i)}&- \kappa \Delta \rho^{(i)} = \kappa \div \left(g_\tau \nabla \rho^{(i-1)} \right), \\
\rho^{(i)}|_{t=0}&=0 \; \text{ for } i>0, \qquad \rho^{(0)}|_{t=0}=\rho_{in} 
   \end{split}
\end{equation}
and $\rho^{(-1)} \equiv 0$.

\begin{lemma}\label{l:improved-energy}
Assume that there is $\pas G (x,t,s) = g(x,t,s)$ such that \eqref{Atime_Gspace} holds for both $f=g$, $f=G$.
Assume \eqref{Atime_scales}, \eqref{Atime_Ini}, \eqref{e:Ggcon}.
Then the energy solutions of \eqref{e:time_setup2} satisfy: \\
\textbf{(1) for $\rho^{(0)}$}
\begin{equation}\label{e:time_main_0}
\|\rho^{(0)} (t)\|_{L^2}^2+2\kappa\int_0^t\|\nabla\rho^{(0)} (s)\|_{L^2}^2\,ds = \|\rho_{in}\|_{L^2}^2
\end{equation}
for any $t \in [0,T]$. Moreover, for any $n\geq 1$
\begin{equation}\label{e:time_main_0n}
\VE  \rho^{(0)} \VE^2_n  \lesssim_{n} \left(\frac{ \mu}{\kappa}\right)^n D.
 \end{equation}
 \textbf{(2) for $\rho^{(i)}$, $i>0$, $n\geq 0$}
 \begin{equation}\label{e:time_main_impr}
\begin{split}
\VE  \rho^{(i)} \VE^2_n  \lesssim_{n,i} (\tau \mu)^{2i}   \left(\frac{\mu}{\kappa}\right)^n  D.
 \end{split}
 \end{equation}
 Furthermore
 \begin{equation}\label{l:time_indR}
\begin{split}
 &\sup_{|\boldsymbol{\alpha}|=n, |\boldsymbol{\beta}|=m}  \left| B^{l} (\nabla \pal \rho^{(i)}, \nabla \pab \rho^{(i)})    \right| \lesssim_{n,m,i,l}  (\tau \mu)^{2i+1}   \left(\frac{\mu}{\kappa}\right)^{\frac{n+m}{2}}  D.
       \end{split}
\end{equation}
\end{lemma}

Note that the equation \eqref{e:time_setup} with Assumptions (A2)-(A4) falls within the above lemma.

\begin{proof}
\noindent\emph{Step 1: Preliminary bound for the bilinear term.} Let $|\boldsymbol{\alpha}|=n$ and $|\boldsymbol{\beta}|=m$. It holds by definition and \eqref{e:fastslow} 
\begin{equation}\label{e:time_Bpre}
\begin{split}
B^{l} (\nabla \pal \rho^{(i)}, &\nabla \pab \rho^{(j)}) =  \kappa \tau \int_0^T \int  (D_t (G^{l}_\tau) -  (D^{slow}_t G^{l})_\tau) \nabla  \pal \rho^{(i)} \cdot \nabla  \pab \rho^{(j)} \\
=& \kappa \tau  \int  G^{l}_\tau  \nabla  \pal  \rho^{(i)} \cdot \nabla  \pab \rho^{(j)}  \bigg|_{t=0}^{t=T}  -\kappa \tau \int_0^T \int  G^{l}_\tau  D_t (\nabla  \pal  \rho^{(i)} \cdot \nabla  \pab \rho^{(j)}) \\
&-    \kappa \tau \int_0^T (D^{slow}_t G^{l})_\tau \nabla  \pal  \rho^{(i)} \cdot \nabla  \pab \rho^{(j)}. \\
\end{split}
\end{equation}
For the time boundary term of \eqref{e:time_Bpre}, integrating once by parts in space and using \eqref{Atime_Gspace} we write 
\[
\left|\int  G^{l}_\tau  \nabla  \pal  \rho^{(i)} \cdot \nabla  \pab \rho^{(j)}  \bigg|_{t=0}^{t=T} \right|\lesssim
\VE \rho^{(i)} \VE_n \VE \rho^{(j)} \VE_{m+2} + (\tau \mu)  \left(\frac{\mu}{\kappa}\right)^\frac{1}{2} \VE \rho^{(i)} \VE_n \VE \rho^{(j)} \VE_{m+1}.
\]
For the last term of \eqref{e:time_Bpre} we simply write the bound
 \[  \kappa \tau \left| \int_0^T (D^{slow}_t G^{l})_\tau \nabla  \pal  \rho^{(i)} \cdot \nabla  \pab \rho^{(j)} \right|\le
  \tau \| D^{slow}_t G^{l} \|_{L^\infty}  \VE \rho^{(i)} \VE_n \VE \rho^{(j)} \VE_{m} \lesssim   \tau \mu \VE \rho^{(i)} \VE_n \VE \rho^{(j)} \VE_{m}.
\]
Together, for any $|\boldsymbol{\alpha}|=n$ and $|\boldsymbol{\beta}|=m$
  \begin{equation}\label{e:time_Bpre2}
  \begin{split}
  \left| B^{l} (\nabla \pal \rho^{(i)}, \nabla \pab \rho^{(j)}) \right| \lesssim \kappa \tau   \VE \rho^{(i)} \VE_n \VE \rho^{(j)} \VE_{m+2} + \kappa \tau   (\tau \mu)  \left(\frac{\mu}{\kappa}\right)^\frac{1}{2} \VE \rho^{(i)} \VE_n \VE \rho^{(j)} \VE_{m+1} \\
   +  \tau \mu \VE \rho^{(i)} \VE_n \VE \rho^{(j)} \VE_{m} +
\kappa \tau   \left| \int_0^T \int  G^{l}_\tau  D_t (\nabla  \pal  \rho^{(i)} \cdot \nabla  \pab \rho^{(j)}) \right|.
\end{split}
\end{equation}

\noindent\emph{Step 2: Reformulating energy estimates.} \label{ss:ree} The case $i=0$ (nonzero initial datum, zero forcing since $\rho^{(-1)} \equiv 0$) is the energy estimate of Lemma \ref{l:energy_lapl}, together with the assumption \eqref{Atime_Ini}. 

For $i>0$,
the energy estimate with forcing (Corollary \ref{c:energy_lapl}) gives (zero initial datum for $i>0$)
\begin{equation}\label{e:ener_n1}
\begin{split}
 \sup_{t \le T}& \|\nabla^n \rho^{(i)}(t)\|^2_{L^2}+ \kappa \|\nabla^{n+1} \rho^{(i)}\|_{L_{xt}^2}^2 \lesssim_n  \\ &\kappa  \sum_{|\boldsymbol{\alpha}|=n} \left| \int_0^T \int \pal  (g_\tau\nabla \rho^{(i-1)})  \cdot\nabla \pal \rho^{(i)}    \right| +  \left(\frac{\mu}{\kappa}\right)^n  \kappa \|\nabla \rho^{(i)}\|_{L_{xt}^2}^2,
 \end{split}
\end{equation}
where in case $n=0$ the second term vanishes. We rewrite it as
\[
\begin{split}
\VE \rho^{(i)} \VE^2_n \lesssim_n  &  \sum_{|\boldsymbol{\alpha}|=n}  \left| B^{0} (\nabla \pal \rho^{(i-1)}, \nabla \pal \rho^{(i)})    \right| + \kappa \sum_{|\boldsymbol{\alpha}|=n} \left| \int_0^T \int([g_\tau  \cdot \nabla, \pal]  \rho^{(i-1)})\nabla \pal \rho^{(i)}     \right| \\
+   &\left(\frac{\mu}{\kappa}\right)^n  \kappa \|\nabla \rho^{(i)}\|_{L_{xt}^2}^2,
 \end{split}
\]
where in case $n=0$ only the first term is present. Using for the last right-hand side term of the preceding inequality its case with $n=0$ we obtain
\begin{equation}\label{e:ener_n2a}
\begin{split}
\VE \rho^{(i)} \VE^2_n \lesssim &    \sum_{|\boldsymbol{\alpha}|=n}   \left| B^{0} (\nabla \pal \rho^{(i-1)}, \nabla \pal \rho^{(i)})    \right| + \kappa    \sum_{|\boldsymbol{\alpha}|=n}  \left| \int_0^T \int ([g_\tau  \cdot \nabla, \pal]  \rho^{(i-1)})\nabla \pal \rho^{(i)}     \right| +\\ 
&\left(\frac{\mu}{\kappa}\right)^n \left| B^{0} (\nabla \rho^{(i-1)}, \nabla\rho^{(i)}) \right|. 
\end{split}
 \end{equation}
The commutator estimate \eqref{Atime_comm1} and Young's inequality allow to bound the middle right-hand side term of \eqref{e:ener_n2a} as follows 
\[
\begin{split}
 &\kappa    \sum_{|\boldsymbol{\alpha}|=n}  \left| \int_0^T \int ([g_\tau  \cdot \nabla, \pal]  \rho^{(i-1)})\nabla \pal \rho^{(i)}     \right| \\
 &\le   \frac{1}{C} \kappa \| \nabla^{n+1}\rho^{(i)} \|^2_{L_{xt}^2} + C  \kappa \|  \nabla^n g  \|^2_{L^\infty}   \|  \nabla \rho^{(i-1)} \|^2_{L_{xt}^2}  +  C  \kappa \|  \nabla g  \|^2_{L^\infty}   \|  \nabla^n \rho^{(i-1)} \|^2_{L_{xt}^2},
 \end{split}
\]
  so absorbing the $\nabla^{n+1}$ term into left-hand side of  \eqref{e:ener_n2a} and using \eqref{Atime_Gspace} we arrive at
\begin{equation}\label{e:ener_n2}
\begin{split}
\VE \rho^{(i)} \VE^2_n &\lesssim_n   \left(\frac{\mu}{\kappa}\right)^n \left| B^{0} (\nabla \rho^{(i-1)}, \nabla\rho^{(i)})    \right| +    \sum_{|\boldsymbol{\alpha}|=n}  \left| B^{0} (\nabla \pal \rho^{(i-1)}, \nabla \pal \rho^{(i)})    \right| \\
+& (\tau \mu)^2 \left(\frac{\mu}{\kappa}\right)^n   \kappa \|  \nabla \rho^{(i-1)} \|^2_{L_{xt}^2}  + (\tau \mu)^2 \left(\frac{\mu}{\kappa}\right)   \kappa    \|  \nabla^n \rho^{(i-1)} \|^2_{L_{xt}^2},
\end{split}
 \end{equation}
 where in case $n=0$ only the first term is present.

\noindent\emph{Step 3: Setting up induction.} Consider the second line of \eqref{e:ener_n2}: (i) it agrees with the desired estimate \eqref{e:time_main_impr} for $i=1$, if we use there definition \eqref{n:D} and \eqref{e:time_main_0n}; (ii) moreover, if we knew the estimate \eqref{e:time_main_impr} for $\rho^{(i-1)}$, it agrees with it for  $\rho^{(i)}$.
Therefore, if we knew that for $\boldsymbol{\alpha}=0$ and $|\boldsymbol{\alpha}|=n$ it holds
\begin{equation}\label{e:Bneeded}
\left| B^{0} (\nabla \pal \rho^{(i-1)}, \nabla \pal \rho^{(i)})    \right| \lesssim (\tau \mu)^{2i}   \left(\frac{\mu}{\kappa}\right)^k  D,
\end{equation}
then \eqref{e:time_main_impr} can be shown by induction.
It is clear then that we should prove our \eqref{e:time_main_impr} inductively, keeping track of $B$.

We induce over $i=0,1, \dots$. The inductive assumption is that for any $n,m,l$
\begin{equation}\label{l:time_ind}
\begin{split}
   (a)\qquad &    \VE \rho^{(j)} \VE^2_n \lesssim (\tau \mu)^{2j}   \left(\frac{\mu}{\kappa}\right)^n  D \qquad  \forall_{j \le i},\\
      (b)\qquad &\sup_{|\boldsymbol{\alpha}|=n, |\boldsymbol{\beta}|=m}  \left| B^{l} (\nabla \pal \rho^{(j)}, \nabla \pab \rho^{(k)})    \right| \lesssim (\tau \mu)^{j+k+1}   \left(\frac{\mu}{\kappa}\right)^{\frac{n+m}{2}}  D  \qquad  \forall_{j,k \le i}.
       \end{split}
\end{equation}
(where constants in $\lesssim$ may depend on $n,m,l,i$),
and we want to show it for $i+1$.

\noindent\emph{Step 4: Initializing the induction, $i=0$.} At the beginning of Section \ref{ss:ree} we have already shown \eqref{e:time_main_0}, \eqref{e:time_main_0n} for  $i=0$, so the part (a) of the inductive assuption for  $i=0$ holds. We need to show now   $i=0$ for the part (b). From \eqref{e:time_Bpre2} with  $i=0$ it holds, thanks to the already known part (a)
\begin{equation}\label{e:time_Bpre2ini}
\begin{split}
    \left| B^{l} (\nabla \pal \rho^{(0)}, \nabla \pab \rho^{(0)})   \right|
    \lesssim  \tau \mu  \left(\frac{ \mu}{\kappa}\right)^{\frac{n+m}{2}} D +
    \kappa \tau   \left| \int_0^T \int  G^{l}_\tau  D_t (\nabla \pal \rho^{(0)} \cdot \nabla \pab \rho^{(0)})\right|
\end{split}
\end{equation}  
for any $|\boldsymbol{\alpha}|=n$, $|\boldsymbol{\beta}|=m$. 
We focus on the last term of \eqref{e:time_Bpre2ini}. Since via \eqref{e:time_setup} for $\rho^{(0)}$ (no forcing) we have
 \[
  D_t (\pal \rho^{(0)}) = \kappa \Delta \pal \rho^{(0)} + [u \cdot \nabla, \pal] \rho^{(0)},
  \]
it holds
\begin{equation}\label{e:time_B0a}
    \kappa \tau  \int_0^T \int  G^{l}_\tau  D_t (\nabla \pal \rho^{(0)}) \cdot \nabla \pab \rho^{(0)} = \kappa \tau \int_0^T \int  G^{l}_\tau  ( \kappa \Delta \nabla \pal  \rho^{(0)} + [u \cdot \nabla, \nabla \pal] \rho^{(0)}) \cdot \nabla \pab\rho^{(0)}. 
\end{equation}
For the first right-hand side term of \eqref{e:time_B0a} we use \eqref{Atime_Gspace} and  \eqref{e:time_main_0n} to write
\[ 
\begin{split}
\kappa \tau  \left|  \int_0^T  \int G^{l}_\tau  ( \kappa \Delta \nabla \pal  \rho^{(0)} \cdot \nabla \pab \rho^{(0)} \right| &\lesssim
    \kappa \tau  \left(\kappa^\frac{1}{2} \|  \nabla^{n+3} \rho^{(0)} \|_{L_{xt}^2}\right)   \left(\kappa^\frac{1}{2} \|\nabla^{m+1}\rho^{(0)} \|_{L_{xt}^2} \right) \\&\lesssim \kappa \tau \VE \rho^{(0)} \VE_{n+2} \VE \rho^{(0)} \VE_{m}  
    \lesssim  (\tau \mu)  \left(\frac{ \mu}{\kappa}\right)^\frac{n+m}{2} D.
      \end{split}
\]
For the second right-hand side term of \eqref{e:time_B0a}, using the assumption \eqref{Atime_Gspace}, the commutator estimate \eqref{Atime_comm1} (for $n+1$), and then \eqref{e:time_main_0n} we write
   \[
   \begin{split}
& \kappa \tau  \left|  \int_0^T  \int  G^{l}_\tau [u \cdot \nabla, \nabla \pal] \rho^{(0)} \cdot \nabla \pab \rho^{(0)} \right| \\
 &\lesssim  \tau \left(\|  \nabla^{n+1} u  \|_{L^\infty}  \kappa^\frac{1}{2}  \|  \nabla \rho^{(0)} \|_{L_{xt}^2}  +  \|  \nabla u  \|_{L^\infty}  \kappa^\frac{1}{2} \|  \nabla^{n+1} \rho^{(0)} \|_{L_{xt}^2} \right) \kappa^\frac{1}{2} \|\nabla^{m+1}\rho^{(0)} \|_{L_{xt}^2} \\
  &\lesssim  \tau  \left( \|  \nabla^{n+1} u  \|_{L^\infty}  +  \|  \nabla u  \|_{L^\infty} \left(\frac{ \mu}{\kappa}\right)^\frac{n}{2} \right) \left(\frac{ \mu}{\kappa}\right)^\frac{m}{2} D \lesssim (\tau \mu)  \left(\frac{ \mu}{\kappa}\right)^\frac{n+m}{2} D,
 \end{split}
       \]
where the last inequality follows from the scales assumption \eqref{Atime_scales}. Hence the entire right-hand side of \eqref{e:time_B0a} is estimated by $(\tau \mu)  \left(\frac{ \mu}{\kappa}\right)^\frac{n+m}{2} D$. This and the other estimate with $D_t$ acting on $\nabla \pab \rho_{0}$ (performed analogously) gives via \eqref{e:time_Bpre2ini} the bound
          \begin{equation}\label{e:indB0}
   \sup_{|\boldsymbol{\alpha}|=n, |\boldsymbol{\beta}|=m}      \left| B^{l} (\nabla \pal \rho^{(0)}, \nabla \pab \rho^{(0)})   \right|  \lesssim (\tau \mu)  \left(\frac{ \mu}{\kappa}\right)^\frac{n+m}{2} D,
                \end{equation}
                which is the inductive hypothesis for $i=0$.

\noindent\emph{Step 5: The inductive step $i \to i+1$.} Now we assume \eqref{l:time_ind} and want to show (for any $n$, $m$, $l$)
\begin{equation}\label{l:time_indH}
\begin{split}
   (a')\qquad &    \VE \rho^{(i+1)} \VE^2_n \lesssim (\tau \mu)^{2i+2}   \left(\frac{\mu}{\kappa}\right)^n  D, \\
      (b')\qquad &  \sup_{|\boldsymbol{\alpha}|=n, |\boldsymbol{\beta}|=m}  \left| B^{l} (\nabla \pal \rho^{(i+1)}, \nabla \pab \rho^{(i)})    \right|  \lesssim (\tau \mu)^{2i+2}   \left(\frac{\mu}{\kappa}\right)^{\frac{n+m}{2}}  D.
       \end{split}
\end{equation}
\noindent\emph{Substep 5.1: Estimate of $B^l$ in terms of $\rho^{(i+1)}$}.
For any $|\boldsymbol{\alpha}|=n$, $|\boldsymbol{\beta}|=m$ the preliminary inequality \eqref{e:time_Bpre2} for $i+1$, $j\le i$ gives via part (a) of the inductive assumption \eqref{l:time_ind}
   \begin{equation}\label{e:time_Bpre2H}
  \begin{split}
  &\left| B^{l} (\nabla \pal \rho^{(i+1)}, \nabla \pab \rho^{(j)})   \right| \lesssim   \VE \rho^{(i+1)} \VE_n  (\tau \mu)^{j+1}  \left(\frac{\mu}{\kappa}\right)^\frac{m}{2} D^\frac{1}{2} +
\kappa \tau   \left| \int_0^T \int  G^{l}_\tau  D_t (\nabla \pal \rho^{(i+1)} \cdot \nabla \pab\rho^{(j)})\right|.
\end{split}
\end{equation}
We will focus on the last term of \eqref{e:time_Bpre2H}. By \eqref{e:time_setup} for $\rho^{(k)}$ (arbitrary natural $k$), we have
\begin{equation}\label{e:time_Dt}
  D_t (\partial^{\boldsymbol{\gamma}}  \rho^{(k)}) = \kappa \Delta \partial^{\boldsymbol{\gamma}}    \rho^{(k)} + [u \cdot \nabla, \partial^{\boldsymbol{\gamma}}  ] \rho^{(k)} +  \kappa \div \partial^{\boldsymbol{\gamma}}    \left(g \nabla \rho^{(k-1)} \right).
\end{equation}
Now we consider two cases, depending on what $D_t$ in the last term of
\eqref{e:time_Bpre2H} acts on.

\noindent\emph{The case when $D_t$ in the last term of \eqref{e:time_Bpre2H} acts on $\rho^{(j)}$}. Use \eqref{e:time_Dt} with $k=j$ to write
\begin{equation}\label{e:time_low1}
  \begin{split}
&  \kappa \tau   \left| \int_0^T \int  G^{l}_\tau  \nabla \pal \rho^{(i+1)} \cdot  D_t \nabla \pab\rho^{(j)}\right| = \\
&\kappa \tau   \left| \int_0^T \int  G^{l}_\tau \nabla \pal \rho^{(i+1)} \cdot \left(  \kappa \Delta \nabla \pab  \rho^{(j)} + [u \cdot \nabla, \nabla \pab] \rho^{(j)} +  \kappa \nabla \div  \pab  \left(g_\tau \nabla \rho^{(j-1)} \right) \right)  \right| =\\
&\kappa \tau   \left| \int_0^T \int  G^{l}_\tau \nabla \pal \rho^{(i+1)} \cdot (  \kappa \Delta \nabla \pab  \rho^{(j)} + [u \cdot \nabla, \nabla \pab] \rho^{(j)} +  \kappa  g_\tau \nabla \div  \pab  \nabla \rho^{(j-1)}+ \kappa  [\nabla \div  \pab , g_\tau ] \nabla \rho^{(j-1)}   \right| \\
 &\lesssim \kappa \tau \VE \rho^{(i+1)} \VE_n  \left( \VE \rho^{(j)} \VE_{m+2} + \frac{\|  \nabla^{m+1} u  \|_{L^\infty}}{\kappa}   \VE \rho^{(j)} \VE_0  +  \frac{\|  \nabla u  \|_{L^\infty}}{\kappa}    \VE   \rho^{(j)}  \VE_m  \right) \qquad =: I\\
& +\kappa \tau \VE \rho^{(i+1)} \VE_n \left( \|\nabla^{m+2} g \|_{L^\infty} \VE \rho^{(j-1)}\VE_0   +   \|\nabla g \|_{L^\infty} \VE \rho^{(j-1)}\VE_{m+1} \right) \qquad =: II\\
&+\kappa^2 \tau  \left| \int_0^T \int  G^{l}_\tau g_\tau \nabla \pal \rho^{(i+1)} \cdot     \nabla  \Delta \pab \rho^{(j-1)} \right| \qquad =: III.
 \end{split}
\end{equation}
 The inequality in \eqref{e:time_low1} holds via commutator estimates \eqref{Atime_comm1} and \eqref{Atime_comm2}. 
 
 For the three terms on the right-hand side of \eqref{e:time_low1} we use the inductive assumption (a) and assumption \eqref{Atime_scales} to bound them as follows
 \[
 I \lesssim \VE \rho^{(i+1)} \VE_n  (\tau \mu)^{j+1}  \left(\frac{\mu}{\kappa}\right)^\frac{m}{2} D^\frac{1}{2},  
 \]
 for the second one additionally using \eqref{Atime_Gspace}
\[ 
  II \lesssim \VE \rho^{(i+1)} \VE_n (\tau \mu)^{j+1}  \left(\frac{\mu}{\kappa}\right)^\frac{m}{2} D^\frac{1}{2},
\]
    and we keep the term $III$ of  \eqref{e:time_low1}, written as $\kappa \tau \left| B^{l+1} (\nabla \pal \rho^{(i+1)},    \nabla  \Delta \pab \rho^{(j-1)})   \right|$, in accordance with the definition of the bilinear form $B^l$.
 These two estimates used in \eqref{e:time_low1} show that
\begin{equation}\label{e:time_B1} 
 \begin{split}
&\kappa \tau   \left| \int_0^T \int  G^{l}_\tau  \nabla \pal \rho^{(i+1)} \cdot  D_t \nabla \pab\rho^{(j)}\right| \\
&\lesssim \VE \rho^{(i+1)} \VE_n  (\tau \mu)^{j+1}  \left(\frac{\mu}{\kappa}\right)^\frac{m}{2} D^\frac{1}{2} +\kappa \tau    \left| B^{l+1} (\nabla \pal \rho^{(i+1)},    \nabla  \Delta \pab \rho^{(j-1)})   \right| .
 \end{split}
\end{equation} 
\noindent\emph{The case when $D_t$ in the last term of \eqref{e:time_Bpre2H} acts on $\rho^{({i+1})}$.} Use \eqref{e:time_Dt} with $k=i+1$
 \begin{equation}\label{e:time_Dtplus}
    \begin{split}
    &  \kappa \tau   \left| \int_0^T \int  G^{l}_\tau \nabla \pab\rho^{(j)} \cdot   D_t \nabla \pal \rho^{(i+1)}   \right| = \\
&\kappa \tau   \left| \int_0^T \int  G^{l}_\tau \nabla \pab \rho^{(j)} \cdot \left(  \kappa \Delta \nabla \pal  \rho^{(i+1)} + [u \cdot \nabla, \nabla \pal] \rho^{(i+1)} +  \kappa \nabla \div  \pal  \left(g_\tau \nabla \rho^{(i)} \right) \right)  \right| = \\
&\kappa \tau   \left| \int_0^T \int  G^{l}_\tau \nabla \pab \rho^{(j)} \cdot \left(  \kappa \Delta \nabla \pal  \rho^{(i+1)} + [u \cdot \nabla, \nabla \pal] \rho^{(i+1)} +  \kappa  g_\tau  \Delta \pal  \nabla \rho^{(i)}+ \kappa  [ \nabla \div \pal , g_\tau ] \nabla \rho^{(i)} \right)  \right|.
 \end{split}
\end{equation}
For the first right-hand side term of \eqref{e:time_Dtplus} we shift the laplacian:
 \[
   \begin{split}
\kappa^2  \tau \int  G^{l}_\tau  \Delta \nabla \pal  \rho^{(i+1)}  \cdot \nabla \pab\rho^{(j)} =& \kappa^2  \tau   \int   G^{l}_\tau \nabla \pal  \rho^{(i+1)} \cdot  \Delta  \nabla \pab\rho^{(j)} +  [G^{l}_\tau,   \Delta ] \nabla \pal  \rho^{(i+1)}  \cdot \nabla \pab \rho^{(j)} \\
=& \kappa^2  \tau   \int   G^{l}_\tau \nabla \pal  \rho^{(i+1)}  \cdot \nabla \Delta   \pab\rho^{(j)} +  \nabla \pal  \rho^{(i+1)}  \cdot [   \Delta ,G^{l}_\tau]  \nabla \pab \rho^{(j)}, 
 \end{split}
  \]
so that via the commutator estimate \eqref{Atime_comm2} we have the bound
  \[
  \begin{split}
 &\kappa \tau   \left|  \int_0^T \int  G^{l}_\tau   \kappa \Delta \nabla \pal  \rho^{(i+1)}   \cdot \nabla \pab\rho^{(j)} \right| \lesssim \\
 & \kappa^\frac{1}{2} \|\nabla^{n+1}\rho^{(i+1)} \|_{L_{xt}^2} \kappa \tau \left( \kappa^\frac{1}{2} \|  \nabla^{m+3} \rho^{(j)} \|_{L_{xt}^2} +  \|\nabla G^{l}_\tau   \|_{L^\infty}  \kappa^\frac{1}{2} \|  \nabla^{m+2} \rho^{(j)} \|_{L_{xt}^2} +  \|\nabla^2 G^{l}_\tau   \|_{L^\infty}  \kappa^\frac{1}{2} \|  \nabla \pab \rho^{(j)} \|_{L_{xt}^2} \right)  \\
 &\lesssim \VE \rho^{(i+1)} \VE_n  (\tau \mu)^{j+1}  \left(\frac{\mu}{\kappa}\right)^\frac{m}{2} D^\frac{1}{2},
 \end{split}
    \]
using for the last inequality the inductive assumption and \eqref{Atime_Gspace}.

The second right-hand side term of \eqref{e:time_Dtplus} is bounded by the commutator estimate \eqref{Atime_comm1} as follows
   \[
     \begin{split}
\kappa \tau   \left| \int_0^T \int  G^{l}_\tau \nabla \pab \rho^{(j)} \cdot  [u \cdot \nabla, \nabla \pal] \rho^{(i+1)}   \right| &\lesssim  \tau  \VE \rho^{(j)} \VE_{m} (\|  \nabla^{n+1} u  \|_{L^\infty}   \VE \rho^{(i+1)} \VE_0  +  \|  \nabla u  \|_{L^\infty}  \VE \rho^{(i+1)} \VE_n ) \\
&\lesssim (\tau\|  \nabla u  \|_{L^\infty} )  \VE \rho^{(j)} \VE_{m} \left( \left( \frac{\mu}{\kappa} \right)^\frac{n}{2}   \VE \rho^{(i+1)} \VE_0  +   \VE \rho^{(i+1)} \VE_n \right) \\
&\lesssim (\tau \mu)^{j+1}  \left(\frac{\mu}{\kappa}\right)^\frac{m}{2} D^\frac{1}{2} \left( \left( \frac{\mu}{\kappa} \right)^\frac{n}{2}   \VE \rho^{(i+1)} \VE_0  +   \VE \rho^{(i+1)} \VE_n \right),
 \end{split}
  \]
where the second inequality holds by \eqref{Atime_scales}, and the last one by the inductive assumption.
  
  The third right-hand side term of \eqref{e:time_Dtplus} is, integrating by parts once ($\partial_\iota$ below is a standard partial derivative, $\iota=1,2,3$, and we sum over repeated $\iota$'s)
     \[
          \begin{split}
& \kappa^2 \tau   \left| \int_0^T \int  G^{l}_\tau  g_\tau \nabla \pab \rho^{(j)} \cdot    \nabla \pal  \Delta \rho^{(i)}   \right| \\
 &=  \kappa \tau   \left| B^{l+1} (\nabla \pab \partial_\iota \rho^{(j)},  \nabla \pal  \partial_\iota  \rho^{(i)})    \right| +   \kappa^2 \tau   \left| \int_0^T \int  \partial_\iota G^{l+1}_\tau \nabla \pab \rho^{(j)}  \cdot  \nabla \pal  \partial_\iota \rho^{(i)}   \right|\\
& \lesssim \kappa \tau  \left| B^{l+1} (\nabla \pab \partial_\iota \rho^{(j)},  \nabla \pal  \partial_\iota  \rho^{(i)})    \right|+  (\tau \mu)^{i+j+2}   \left(\frac{\mu}{\kappa}\right)^{\frac{n+m}{2}}  D,
\end{split}
 \]
 using the inductive assumption and \eqref{Atime_Gspace}.
  
 The fourth right-hand side term of \eqref{e:time_Dtplus}, via the commutator estimate \eqref{Atime_comm2}, is estimated as follows
        \[
          \begin{split}
    \kappa^2 \tau   \left| \int_0^T \int  G^{l}_\tau \nabla \pab \rho^{(j)}   [  \nabla \div \pal , g_\tau ] \nabla \rho^{(i)}   \right| &\lesssim     \kappa \tau   \VE \rho^{(j)} \VE_{m}  (\|\nabla^{n+2} g_\tau  \|_{L^\infty} \VE \rho^{(i)}\VE_0   +   \|\nabla g_\tau  \|_{L^\infty} \VE \rho^{(i)} \VE_{n+1})   \\
  & \lesssim   (\tau \mu)^{i+j+2}   \left(\frac{\mu}{\kappa}\right)^{\frac{n+m}{2}}  D,
 \end{split}
     \]
using again inductive assumptions and  \eqref{Atime_Gspace}.

Altogether these estimates in \eqref{e:time_Dtplus} give
\begin{equation}\label{e:time_B2}    
\begin{split}   
&\kappa \tau   \left| \int_0^T \int  G^{l}_\tau \nabla \pab\rho^{(j)} \cdot   D_t \nabla \pal \rho^{(i+1)}   \right| \lesssim \\
       &(\tau \mu)^{j+1}  \left(\frac{\mu}{\kappa}\right)^\frac{m}{2} D^\frac{1}{2} \left(   \VE \rho^{(i+1)} \VE_n
             +  \left(\frac{\mu}{\kappa}\right)^\frac{n}{2}  \VE \rho^{(i+1)} \VE_0  +  (\tau \mu)^{i+1}   \left(\frac{\mu}{\kappa}\right)^{\frac{n}{2}}  D^\frac{1}{2} \right) \\
            &+  \kappa \tau  \left| B^{l+1} (\nabla \pab \partial_\iota \rho^{(j)}, \nabla \pal  \partial_\iota  \rho^{(i)}) \right|.
\end{split}
\end{equation}

\noindent\emph{Estimate of \eqref{e:time_Bpre2H}.} Together, \eqref{e:time_B1},  \eqref{e:time_B2}  used for the right-hand side of \eqref{e:time_Bpre2H} give
\begin{equation}\label{e:time_B3tog}
  \begin{split}
 & \left| B^{l} (\nabla \pal \rho^{(i+1)}, \nabla \pab\rho^{(j)})   \right|
\lesssim_{i,m,n,l} \\
& \kappa \tau    \left| B^{l+1} (\nabla \pal \rho^{(i+1)},    \nabla  \Delta \pab \rho^{(j-1)})   \right| +  \kappa \tau  \left| B^{l+1} (\nabla \pal  \partial_\iota  \rho^{(i)}, \nabla \pab \partial_\iota \rho^{(j)})    \right| + Q (j,m),
\end{split}
    \end{equation}
where we have denoted the parts containing norms of right-hand side by $Q (j,m)$, i.e.
\[
  \begin{split} 
  Q (j,m) := (\tau \mu)^{j+1}  \left(\frac{\mu}{\kappa}\right)^\frac{m}{2} D^\frac{1}{2} \left(   \VE \rho^{(i+1)} \VE_n
             +  
          \left(\frac{\mu}{\kappa}\right)^\frac{n}{2}  \VE \rho^{(i+1)} \VE_0  +  (\tau \mu)^{i+1}   \left(\frac{\mu}{\kappa}\right)^{\frac{n}{2}}  D^\frac{1}{2} \right).
\end{split}
\]
Using the inductive assumption \eqref{l:time_ind} part (b) for the first right-hand side term of  \eqref{e:time_B3tog} we get
\begin{equation} \label{e:time_B3}
  \left| B^{l} (\nabla \pal \rho^{(i+1)}, \nabla \pab\rho^{(j)})   \right|
\lesssim \kappa \tau    \left| B^{l+1} (\nabla \pal \rho^{(i+1)},    \nabla  \Delta \pab \rho^{(j-1)})   \right|   +   Q (j,m),
\end{equation}
which holds for any $j \le i$ and any multiindex  $|\boldsymbol{\alpha}|=n$, $|\boldsymbol{\beta}|=m$. Thus
denoting by 
$$ \mathcal{B} (l,m,j) :=   \sup_{|\boldsymbol{\alpha}|=n, |\boldsymbol{\beta}|=m}  \left| B^{l} (\nabla \pal \rho^{(i+1)}, \nabla \pab \rho^{(j)})    \right| $$ (with $(i+1)$ and $n$ fixed) we rewrite \eqref{e:time_B3} and iterate it as follows
\begin{equation}\label{e:iter_back}
  \begin{split}
 \mathcal{B} (l,m,j) &\lesssim  \kappa \tau  \mathcal{B} (l+1,m+2,j-1)+ Q (j,m)\\
 &\lesssim (\kappa \tau)^2  \mathcal{B} (l+2,m+4,j-2)+ \kappa \tau Q (j-1,m+2) + Q (j,m)\\
 &\dots \lesssim  (\kappa \tau)^j  \mathcal{B} (l+j,m+2j,0)  + \sum_{i=0}^{j-1} (\kappa \tau)^i  Q (j-i,m+2i),
\end{split}
\end{equation}
which holds for any $j \le i$. Observe that $\sum_{i=0}^{j} (\kappa \tau)^i  Q (j-i,m+2i) \lesssim Q (j,m)$. This in \eqref{e:iter_back} yields
\begin{equation}\label{e:iter_back_nw}
 \mathcal{B} (l,m,j) \lesssim (\kappa \tau)^j  \mathcal{B} (l+j,m+2j,0)  + Q (j,m).
\end{equation}
Importantly, we can perform one last step\footnote{
This iteration ``all the way back'' allows us to obtain ``optimal improvement'' mentioned in the introduction.}: recall that we have $ \rho^{(-1)}\equiv 0$ and observe that \eqref{e:time_B3tog} is valid for the choice $j=0$, where it yields 
\[
  \begin{split}
 (\kappa \tau)^j \mathcal{B} (l+j,m+2j,0) &\lesssim  (\kappa \tau)^{j+1}  \sup_{|\boldsymbol{\alpha}|=n, |\boldsymbol{\beta}|=m+2j}   \left| B^{l+j+1} (\nabla \pal  \partial_\iota  \rho^{(i)}, \nabla \pab \partial_\iota \rho^{(0)})    \right| +  (\kappa \tau)^j Q (0,m+2j) \\
&\lesssim  (\kappa \tau)^{j+1}  (\tau \mu)^{i+1}   \left(\frac{\mu}{\kappa}\right)^{\frac{n+m+2j+2}{2}}  D   + Q (j,m) \lesssim  Q (j,m),
\end{split}
\]
where the second inequality is given by the inductive assumption \eqref{l:time_ind} (b) and via the already established $(\kappa \tau)^j  Q (0,m+2j)\lesssim Q (j,m)$, an the last inequality follows from the definition of $Q (j,m)$. This estimate used in \eqref{e:iter_back_nw} with $j=i$ yields via 
definition of $B$ 
\begin{equation}\label{e:time_B4}
\sup_{|\boldsymbol{\alpha}|=n, |\boldsymbol{\beta}|=m}   \left| B^{l} (\nabla \pal \rho^{(i+1)}, \nabla \pab\rho^{(i)})   \right| 
\lesssim_{i,m,n,l} Q (i,m).
\end{equation}

\noindent\emph{Substep 5.2: Closing the induction argument.} To close the argument, we need to return to estimates on $\rho^{(i+1)}$, since it appears in $Q (i,m)$. Note that \eqref{e:ener_n2} for $i+1$ gives for $n=0$
\begin{equation}\label{e:ener_n2H0}
\begin{split}
\VE \rho^{(i+1)} \VE^2_0 \lesssim \left| B^{0} (\nabla \rho^{(i)}, \nabla\rho^{(i+1)})\right|    &\lesssim   (\tau \mu)^{2i+2}  D + Q (i,0) \lesssim (\tau \mu)^{2i+2}  D + (\tau \mu)^{i+1}  \VE \rho^{(i+1)} \VE_0,
\end{split}
 \end{equation}
with the second inequality following from \eqref{e:time_B4} with $m=n=0$ and the last one from the definition of the quantity $Q$. Splitting by Young's inequality gives the inductive hypothesis \eqref{l:time_indH} (a') for $\VE \rho^{(i+1)}\VE_0$.

Similarly for $n \ge 1$ via  \eqref{e:ener_n2}
\begin{equation}\label{e:ener_almost}
\begin{split}
\VE \rho^{(i+1)} \VE^2_n \lesssim& \left(\frac{\mu}{\kappa}\right)^n \left| B^{0} ( \nabla\rho^{(i+1)}, \nabla \rho^{(i)})    \right| +   \sum_{|\boldsymbol{\alpha}|=n}  \left| B^{0} (\nabla \pal \rho^{(i+1)}, \nabla \pal \rho^{(i)})    \right| +  (\tau \mu)^{2i+2}   \left(\frac{\mu}{\kappa}\right)^n  D,
\end{split}
 \end{equation}
where we have used the inductive hypothesis \eqref{l:time_ind} part (a). Observe that at the right-hand sides of \eqref{e:ener_n2H0} and \eqref{e:ener_almost} the bilinear form $B^l$ appears only with $l=0$, hence the constants in the \eqref{l:time_ind} part (a) are $l$-independent. (Since we lower the index of $\rho$ at the cost of raising $l$, compare \eqref{e:time_B3},\eqref{e:iter_back_nw}, closing the argument requires using  $B^l$, $l>0$.)

Using now for the right-hand side of \eqref{e:ener_almost} the estimate \eqref{e:time_B4} with $m=n$, the known estimate for $\VE \rho^{(i+1)}\VE_0$, and the definition of $Q(i,m)$ yields the entire inductive hypothesis \eqref{l:time_indH} (a'), i.e. the one for $\rho^{(i+1)}$. Now, we use \eqref{l:time_indH} (a') for $\rho^{(i+1)}$ in \eqref{e:time_B4} to close the induction. Lemma \ref{l:improved-energy} is proven.
\end{proof}

\section{Construction of the vector field - Proof of Proposition \ref{p:Onsager}}
\label{s:Onsager}

This section is devoted to the proof of Proposition \ref{p:Onsager}. We point out that this proposition is a mild variation of \cite[Proposition 2.1]{BDSV}, and we follow the proof very closely. On the other hand, for our purposes in this paper, we need to adjust and refine certain parameters used in the construction \cite{BDSV}, and for this reason we will need to repeat all the steps.

First, we choose the parameters $\gamma_T, \gamma_R, \gamma_E, \gamma_L, \bar{N}$ and $\alpha_0$. The universal constants $M,\bar{e}$ will be defined below in Definitions \ref{d:defM} and \ref{d:defebar}. Actually, for the proof below we will require the following inequalities relating these parameters:
\begin{align}
\gamma_L&<(b-1)\beta\,,\label{e:comparison}\\
4\alpha(1+\gamma_L)+2\gamma_L&<2(b-1)\beta+\gamma_T+\gamma_R\,,\label{e:gluing2}\\
\alpha\gamma_L&<\gamma_T\,,\label{e:taucondition1}\\
4\alpha(1+\gamma_L)&<\gamma_R\,,\label{e:gammaRalphacondition}\\
2\beta(b-1)+1+\gamma_R&<\bar{N}\gamma_L\,,\label{e:gluing1}\\
b\alpha+\gamma_T+b\gamma_R&<(b-1)(1-(2b+1)\beta)\,,\label{e:transportcondition}\\
b\gamma_E&<(b-1)(1-(2b+1)\beta)\,.\label{e:energycondition}
\end{align}
We first claim that \eqref{e:Onsager_Conditions} allow us to choose $\gamma_L, \bar{N}$ and $\alpha>0$ so that \eqref{e:comparison}-\eqref{e:energycondition} are valid. To see this, we can first choose $\gamma_L>0$ sufficiently small so that \eqref{e:comparison} holds and 
\begin{equation*}
2\gamma_L<2(b-1)\beta+\gamma_T+\gamma_R
\end{equation*}
and $\gamma_T$, $\gamma_R$ such that 
\[
\gamma_T+b\gamma_R <(b-1)(1-(2b+1)\beta).
\]
Then we choose $\bar{N}$ sufficiently large (depending on $\gamma_L$) so that \eqref{e:gluing1} is valid. Finally, we observe that then \eqref{e:gluing2}, \eqref{e:taucondition1}, \eqref{e:gammaRalphacondition} and \eqref{e:transportcondition} hold for any sufficiently small $\alpha>0$. 

Let us make the following important remark:
\begin{remark}\label{r:remarkona}
The basic principle in the subsequent estimates in this section will be, as it has been used in \cite{BDSV}, that implicit constants may depend on these parameters but do not depend on $\lambda_q$ and in particular do not depend on the large constant $a\gg 1$ in the definition of $\lambda_q$ (cf.~\eqref{e:lambdadelta}). Consequently, for sufficiently large $a\gg 1$ the implicit constants can be absorbed, so that ultimately we obtain the estimates \eqref{e:R_q_inductive_est}-\eqref{e:energy_inductive_assumption} for $q+1$.	
\end{remark}

\subsection{Mollification step}
\label{s:mollify}

Following \cite[Section 2.4]{BDSV} we define
\begin{equation*}
u_{\ell}:= u_q*\psi_{\ell_q},\quad
\mathring{R}_{\ell}:= \mathring R_q*\psi_{\ell_q}  -(u_q\mathring\otimes u_q)*\psi_{\ell_q}  + u_{\ell_q}\mathring\otimes u_{\ell_q} 	
\end{equation*}
so that $(u_\ell,\mathring{R}_\ell)$ is another solution to \eqref{e:EulerReynolds}. However, at variance with \cite[Section 2.4]{BDSV}, we fix a mollifying kernel $\psi\in C^\infty_c(\R^3)$ which, in addition to the usual requirement $\int_{\R^3}\psi\,dx=1$, also satisfies 
\begin{equation}\label{e:deepmollifier}
\int_{\R^3}\psi(x)x^\theta\,dx=0\quad\textrm{ for any multiindex $\theta$ with $1\leq|\theta|\leq \bar{N}$}.
\end{equation}
The construction and use of such mollifiers (called ``deep smoothing operators of depth $\bar{N}$'') is standard, see e.g.~\cite[Section 2.3.4]{GromovBook}; the case of infinite depth was introduced by Nash \cite{Nash56}. We point out that if $\bar{N}\geq 2$, then $\psi$ cannot be nonnegative. 

The key point for introducing these ``deep smoothing operators'' is the following lemma, a variant of the usual smoothing estimates.

\begin{lemma}\label{l:mollify}
Let  $\psi\in C^\infty_c(\R^n)$ be a smoothing operator of depth $\bar{N}\geq 1$ and such that $\int_{\R^n}\psi=1$. Let $\psi_\ell (x) := \ell^{-n}\psi (\ell^{-1}x)$. Then for any real $r,s\geq 0$ 
\begin{equation}
	\|f*\psi_\ell\|_{C^{r+s}}\lesssim \ell^{-s}\|f\|_{C^r}\label{e:mollify1}
\end{equation}
and for any $r\geq 0$, $0\leq s\leq \bar{N}$
\begin{equation}
\|f-f*\psi_\ell\|_{C^r} \lesssim \ell^{s}\|f\|_{C^{r+s}}\label{e:mollify2}.
\end{equation}
The implicit constants depend on the choice of $\psi$ as well as on $r,s$.
\end{lemma}

\begin{proof}
Concerning \eqref{e:mollify1} assume first that $r=k$ and $s=l$ are integers and let $\boldsymbol{\alpha},\boldsymbol{\beta}$ be multi-indices with $|\boldsymbol{\alpha}|=k, |\boldsymbol{\beta}|=l$. Then $\partial^{\boldsymbol{\alpha}+\boldsymbol{\beta}}(f*\psi_\ell)=\partial^{\boldsymbol{\alpha}} f*\partial^{\boldsymbol{\beta}} \psi_\ell$, hence
$$
|\partial^{\boldsymbol{\alpha}+\boldsymbol{\beta}}(f*\psi_\ell)|\leq C_l\ell^{-l}\|f\|_k.
$$
If $s=l+\alpha$, $\alpha \in (0,1)$, we write
\begin{align*}
\partial^{\boldsymbol{\alpha}} f*\partial^{\boldsymbol{\beta}}\psi_\ell(x+z)-\partial^{\boldsymbol{\alpha}}f*\partial^{\boldsymbol{\beta}}\psi_\ell(x)&=\int_{\R^n}\partial^{\boldsymbol{\alpha}}f(x-y)\left(\partial^{\boldsymbol{\beta}}\psi_\ell(y+z)-\partial^{\boldsymbol{\beta}}\psi_\ell(y)\right)\,dy\\
	&=\ell^{-k}\int_{\R^n}\partial^{\boldsymbol{\alpha}}f(x-y)\left((\partial^{\boldsymbol{\beta}}\psi)(y+\ell^{-1} z)-(\partial^{\boldsymbol{\beta}}\psi)(y)\right)\,dy,
\end{align*}
from which we obtain
\begin{equation*}
\|\partial^{\boldsymbol{\alpha}+\boldsymbol{\beta}}(f*\psi_\ell)\|_{\alpha}\leq C_{l,\alpha}\ell^{-l-\alpha}\|f\|_k.
\end{equation*}
Finally, if also $r=k+\beta$ for some $\beta\in(0,1)$, we obtain the required estimate from interpolation between $r=k$ and $r=k+1$.
This concludes the proof of \eqref{e:mollify1} for $r,s\geq 0$. 

Next, by considering the Taylor expansion of $f$ at $x$ we can write $f(x-y) = Q_x(y) + R_x(y)$,
where $Q_x(y)$ is a sum of monomials in $y$ of degree $d$ with $1\leq d\leq \bar{N}$ and $|R_x(y)|\lesssim |y|^{s}\|f\|_{C^s}$. Moreover, from \eqref{e:deepmollifier} we deduce that $\int_{\R^n}Q_x(y)\psi_\ell(y)\,dy=0$. Thus,
\begin{equation*}
|f-f*\psi_\ell|= \left|\int \psi_\ell(y)(f(x-y)-f(x))dy\right|\\
\lesssim \|f\|_{C^s}\int \ell^{-n}\left|\psi(\ell^{-1}y)\right||y|^{s}dy\lesssim \ell^{s}\|f\|_s\, .
\end{equation*}
This proves \eqref{e:mollify2} for the case $r=0$. To obtain the estimate for integer $r=k$, repeat the same argument for the partial derivatives
$\partial^a f$ with $|a|=k$. For general real $r\geq 0$ we again proceed by interpolation. 
\end{proof}

With the help of Lemma \ref{l:mollify} we obtain the following bounds.
\begin{proposition}\label{p:est_mollification}
For any $N\geq 0$ we have 
\begin{align}
\norm{u_{\ell}}_{C^{N+1}} &\lesssim \begin{cases} \delta_q^{\sfrac 12}\lambda_q^{N+1}&\textrm{ if }N+1\leq\bar{N}\\ \delta_q^{\sfrac 12}\lambda_q^{\bar{N}}\ell_q^{\bar{N}-N-1}& \textrm{ if }N+1\geq \bar{N}\end{cases}\,,  \label{e:u:ell:k}\\
\norm{\mathring{R}_{\ell}}_{C^{N}}&\lesssim  \mathring{\delta}_{q+1}\ell_q^{-N}+\delta_q\lambda_q^{1+\bar{N}}\ell_q^{\bar{N}-N} \,. \label{e:R:ell}\\
\abs{\int_{\T^3}\abs{u_q}^2-\abs{u_{\ell}}^2\,dx} &\lesssim \mathring{\delta}_{q+1}+\delta_q^{\sfrac12}\lambda_q^{\bar{N}}\ell_q^{\bar{N}}\,.
\label{e:uq_vell_energy_diff}
\end{align}
Moreover, if $z_q=\mathcal{B}u_q$ and $z_\ell=\mathcal{B}u_{\ell}=z_q*\psi_{\ell_q}$ are the vector potentials, we have in addition
\begin{align}
\norm{z_\ell-z_q}_{C^{N+\alpha}}&\lesssim \delta_q^{\sfrac12}\lambda_q^{\bar{N}}\ell_q^{\bar{N}+1-N-\alpha}\,,\label{e:z:ell}
\end{align}
\end{proposition}

\begin{proof}
The bounds \eqref{e:u:ell:k} and \eqref{e:z:ell} follow directly from \eqref{e:mollify1} and \eqref{e:mollify2} together with the classical Schauder estimates on the Calder\'on-Zygmund operator $\nabla\mathcal{B}$. For \eqref{e:R:ell} we use the bound \eqref{e:u_q_inductive_est} and interpolation to obtain 
$$
\|u_q\otimes u_q\|_{C^{\bar{N}}}\lesssim \|u_q\|_{C^0}\|u_q\|_{C^{\bar{N}}}\leq \|u_q\|_{C^1}\|u_q\|_{C^{\bar{N}}}\lesssim\delta_q\lambda_q^{\bar{N}+1}\,.
$$
Then we apply \eqref{e:mollify2} to the decomposition
\begin{equation*}
\|\mathring{R}_\ell\|_{C^N}\leq \|\mathring{R}_q*\psi_{\ell_q}\|_{C^N}+\|(u_q\mathring\otimes u_q)*\psi_{\ell_q}-u_q\mathring\otimes u_q\|_{C^N}+2\|(u_q-u_{q}*\psi_{\ell_q})\otimes u_q\|_{C^N}\,.
\end{equation*}
\end{proof}

Note that at variance with \cite[Proposition 2.2]{BDSV} we have not used a commutator estimate here.

From \eqref{e:gluing1} we obtain $\delta_q^{\sfrac12}\lambda_q^{\bar{N}}\ell_q^{\bar{N}}\leq \delta_q\lambda_q^{1+\bar{N}}\ell_q^{\bar{N}}\leq \mathring{\delta}_{q+1}$. Consequently we have
\begin{corollary}\label{c:mollification}
For any $N\geq 0$ we have the estimates
\begin{align*}
\norm{u_{\ell}}_{C^{N+1}} &\lesssim \delta_q^{\sfrac 12}\lambda_q\ell_q^{-N}\,, \\
\norm{\mathring{R}_{\ell}}_{C^{N}}&\lesssim  \mathring{\delta}_{q+1}\ell_q^{-N}\,, \\
\abs{\int_{\T^3}\abs{u_q}^2-\abs{u_{\ell}}^2\,dx} &\lesssim \mathring{\delta}_{q+1}\,,\\
\norm{z_\ell-z_q}_{C^{N+\alpha}}&\lesssim \mathring{\delta}_{q+1}\ell_q^{1-\alpha-N}\,.
\end{align*}

\end{corollary}

\subsection{Gluing step}

The gluing step, introduced in \cite{Isett2018}, proceeds as follows. For each $i\in\N$ we set $t_i=i\tau_q$ and let $u_i$ be the (classical) solution of the Euler equations
\begin{equation}\label{e:gluingEuler}
	\begin{split}
	\partial_tu_i+\div(u_i\otimes u_i)+\nabla p_i&=0\,,\\
	\div u_i&=0\,,\\
	u_i(\cdot,t_i)&=u_\ell(\cdot,t_i)\,.
	\end{split}
\end{equation}  
It is well-known (see for instance \cite[Proposition 3.1]{BDSV}) that there exists a constant $c(\alpha)>0$ such that, for each $i\in\N$ the solution $u_i$ is smooth, uniquely defined, and satisfies for any $N\geq 1$ the estimates
\begin{equation*}
\|u_i(t)\|_{C^{N+\alpha}}\lesssim \|u_\ell(t_i)\|_{C^{N+\alpha}}\quad \textrm{ for all }t\in(t_i-T,t_i+T)
\end{equation*}
for $T\leq c\|u_\ell(t_i)\|_{C^{1,\alpha}}^{-1}$, where the implicit constant depends on $N$ and $\alpha\in(0,1)$. In particular, from our choice of $\tau_q$ in \eqref{e:elltau} we obtain for any $N\geq 1$ (cf. \cite[Corollary 3.2]{BDSV})
\begin{equation*}
	\|u_i(t)\|_{C^{N+\alpha}}\lesssim \delta_q^{\sfrac12}\lambda_q\ell_q^{1-N-\alpha}\,,
\end{equation*}
provided 
\begin{equation}\label{e:taucondition}
	\tau_q\|u_\ell\|_{C^{1,\alpha}}\leq c.
\end{equation} 
Taking into account Remark \ref{r:remarkona} and \eqref{e:u:ell:k}, this is ensured by \eqref{e:taucondition1} and by choosing $a\gg 1$ sufficiently large. 
Following the derivation of the stability estimates in \cite[Proposition 3.3]{BDSV} and \cite[Proposition 3.4]{BDSV}, we deduce

\begin{proposition}\label{p:stability}
For $|t-t_i|\leq \tau_q$ and $N\geq 0$ we have
\begin{align}
\|u_i-u_\ell\|_{C^{N+\alpha}}&\lesssim \tau_q\mathring{\delta}_{q+1}\ell_q^{-N-1-2\alpha}	\,,\label{e:stabilityu}\\
\|z_i-z_\ell\|_{C^{N+\alpha}}&\lesssim \tau_q\mathring{\delta}_{q+1}\ell_q^{-N-2\alpha}	\,,\label{e:stabilityz}\\
\|(\partial_t+u_\ell\cdot\nabla)(z_i-z_\ell)\|_{C^{N+\alpha}}&\lesssim \mathring{\delta}_{q+1}\ell_q^{-N-2\alpha}.\label{e:stabilityDz}
\end{align}
\end{proposition}

\begin{proof}
Using the equations satisfied by $u_q$ and $u_i$, we obtain the equation for the pressure difference
\begin{equation*}
\Delta (p_{\ell} - p_i) = \div\bigl(\nabla u_\ell(u_i-u_\ell)\bigr)+\div\bigl(\nabla u_i{(u_i-u_\ell)}\bigr)+\div\div\mathring{R}_{\ell},
\end{equation*}
and deduce 
\begin{equation*}
\norm{p_{\ell} - p_i }_{C^{1+\alpha}} \lesssim  \norm{u_\ell}_{C^{1+\alpha}}\norm{u_{i} - u_\ell}_{C^\alpha}+\|\mathring{R}_\ell\|_{C^{1+\alpha}}\,.
\end{equation*}
Using the equation for $u_q$ and $u_i$ we then obtain
\begin{equation*}
\|(\partial_t+u_\ell\cdot\nabla)(u_\ell-u_i)\|_{C^\alpha}\lesssim \norm{u_\ell}_{C^{1+\alpha}}\norm{u_{i} - u_\ell}_{C^\alpha}+\|\mathring{R}_\ell\|_{C^{1+\alpha}}\,.
\end{equation*}
Applying Corollary \ref{c:mollification}, \eqref{e:taucondition} and Grönwall's inequality (for details, compare \cite[Appendix B]{BDSV}) we then conclude
\begin{equation*}
	\|u_i-u_\ell\|_{C^\alpha}\lesssim \tau_q\mathring{\delta}_{q+1}\ell_q^{-1-2\alpha}\,,
\end{equation*}
which is \eqref{e:stabilityu} with $N=0$. The case $N\geq 1$ follows analogously, following the computations in the proof of \cite[Proposition 3.3]{BDSV}. Furthermore, the estimates \eqref{e:stabilityz}-\eqref{e:stabilityDz} can be deduced in the same manner, following the computations in the proof of \cite[Proposition 3.4]{BDSV}.
\end{proof}

Next, as in \cite[Section 4]{BDSV}, we partition time using a partition of unity $\{\chi_i\}_i$, with $\chi_i\in C^{\infty}_c(\R)$ and $0\leq \chi_i\leq 1$, such that 
\begin{itemize}
\item $\sum_i\chi_i\equiv 1$ in $[0,T]$,
\item $\supp\chi_i\subset (t_i-\tfrac{2}{3}\tau_q,t_i+\tfrac{2}{3}\tau_q)$, in particular $\supp\chi_i\cap\supp\chi_{i+2}=\emptyset$,
\item $\chi_i=1$ on $(t_i-\tfrac{1}{3}\tau_q,t_i+\tfrac{1}{3}\tau_q)$ and $\chi_i+\chi_{i+1}=1$ on $(t_i+\tfrac{1}{3}\tau_q,t_i+\tfrac{2}{3}\tau_q)$,
\item $\|\partial_t^N\chi_i\|_{C^0}\lesssim \tau_q^{-N}$,	
\end{itemize}
and define
\begin{equation*}
\bar u_q=\sum_i \chi_i u_i\,,\quad 
\bar p_q^{(1)}=\sum_i \chi_i p_i.
\end{equation*}
Further, we define 
\begin{align*}
\mathring{\bar{R}}_q&=\partial_t\chi_i\mathcal{R}(u_i-u_{i+1})-\chi_i(1-\chi_i)(u_i-u_{i+1})\mathring{\otimes} (u_i-u_{i+1}),\\
\bar{p}_q^{(2)}&=-\chi_i(1-\chi_i)\left(|u_i-u_{i+1}|^2-\int_{\T^3}|u_i-u_{i+1}|^2\,dx\right),
\end{align*}
for $t\in (t_i+\tfrac{1}{3}\tau_q,t_i+\tfrac{2}{3}\tau_q)$ and $\mathring{\bar{R}}_q=0$, $\bar{p}_q^{(2)}=0$ for $t\notin\bigcup_{i}(t_i+\tfrac{1}{3}\tau_q,t_i+\tfrac{2}{3}\tau_q)$, where $\mathcal{R}$ is the ``inverse divergence'' operator for symmetric tracefree 2-tensors, defined as
\begin{equation}
\label{e:R:def}
\begin{split}
({\mathcal R} f)^{ij} &= {\mathcal R}^{ijk} f^k \\
{\mathcal R}^{ijk} &= - \frac 12 \Delta^{-2} \partial_i \partial_j \partial_k - \frac 12 \Delta^{-1} \partial_k \delta_{ij} +  \Delta^{-1} \partial_i \delta_{jk} +  \Delta^{-1} \partial_j \delta_{ik}.
\end{split}
\end{equation}
when acting on vectors $f$ with zero mean on $\T^3$. See \cite{DSz13}, 
\cite[Definition 4.2 and Lemma 4.3]{DaSz2017} and \cite[Proposition 4.1]{BDSV}. 
 
Finally, we set 
\begin{equation*}
	\bar{p}_q=\bar{p}_q^{(1)}+\bar{p}_q^{(2)}.
\end{equation*}
As in \cite[Section 4.2]{BDSV}, one can easily verify that 
\begin{itemize}
\item $\mathring{\bar{R}}_q$ is a smooth symmetric and traceless 2-tensor;
\item For all $(x,t)\in \T^3\times [0,T]$
\begin{equation*}
\left\{\begin{array}{l}
\partial_t\bar{u}_q+\div(\bar{u}_q\otimes\bar{u}_q)+\nabla \bar{p}_q =\div \mathring{\bar{R}}_q\\ \\
\div \bar{u}_q =0,
\end{array}\right.
\end{equation*}
\item The support of $\mathring{\bar{R}}_q$ satisfies 
\begin{equation}\label{e:gluedsupport}
\supp\mathring{\bar{R}}_q\subset \T^3\times \bigcup_i(t_i+\tfrac{1}{3}\tau_q,t_i+\tfrac{2}{3}\tau_q).
\end{equation}
\end{itemize}

With our choice of parameters $\tau_q,\ell_q$ the estimates in \cite[Section 4.3 and Section 4.4]{BDSV} are modified as follows. 
\begin{proposition}\label{p:p_gluing}
The velocity field $\bar{u}_q$ and its vector potential $\bar{z}_q=\mathcal{B}\bar{u}_q$ satisfy the following estimates
\begin{align}
\norm{\bar{u}_q-u_\ell}_{C^{N+\alpha}} &\lesssim \tau_q\mathring{\delta}_{q+1}\ell_q^{-1-N-2\alpha} \label{e:uq:vell:additional}\,,\\
\norm{\bar{z}_q-z_\ell}_{C^{\alpha}}&\lesssim \tau_q\mathring{\delta}_{q+1}\ell_q^{-\alpha}\,.	\label{e:zq:vell:additional}
\end{align}
for all $N \geq 0$. The new Reynolds stress $\mathring{\bar{R}}_q$ satisfies the estimates:
\begin{align}
\norm{\mathring{\bar R}_q}_{N+\alpha} &\lesssim \mathring{\delta}_{q+1}\ell_q^{-N-2\alpha}+\tau_q^2\mathring{\delta}_{q+1}^2\ell_q^{-N-2-4\alpha}\,, \label{e:Rq:1}\\
\norm{(\partial_t + \bar v_q\cdot \nabla) \mathring{\bar R}_q}_{N+\alpha} &\lesssim \tau_q^{-1}\mathring{\delta}_{q+1}\ell_q^{-N-3\alpha}+\tau_q\mathring{\delta}_{q+1}^2\ell_q^{-N-2-4\alpha}\,. \label{e:Rq:Dt}
\end{align}
Furthermore, we have the estimate
\begin{equation}\label{e:gluingenergy}
	\left|\int_{\T^3}|\bar{u}_q|^2-|u_\ell|^2\,dx\right|\lesssim \tau_q\mathring{\delta}_{q+1}\delta_{q}^{\sfrac12}\lambda_q+\tau_q^2\mathring{\delta}_{q+1}^2\ell_q^{-2-4\alpha}\,.
\end{equation}
\end{proposition}

\begin{proof}
Using the identity $	\bar{u}_q-u_\ell=\sum_i\chi_i(u_i-u_\ell)$, the bounds \eqref{e:uq:vell:additional} and \eqref{e:zq:vell:additional} follow directly from \eqref{e:stabilityu} and \eqref{e:stabilityz} in Proposition \ref{p:stability}. 

As in the proof of \cite[Proposition 4.4]{BDSV} we write the new Reynolds stress as
\begin{equation*}
\mathring{\bar{R}}_q=\partial_t\chi_i(\mathcal{R}\curl)(z_i-z_{i+1})-\chi_i(1-\chi_i)(u_i-u_{i+1})\mathring{\otimes} (u_i-u_{i+1})
\end{equation*}
and note that $\mathcal{R}\curl$ is a zero-order operator of Calderon-Zygmund type, for which Schauder estimates are valid. Therefore we obtain, again applying Proposition \ref{p:stability},
\begin{align*}
\|\mathring{\bar{R}}_q\|_{C^{N+\alpha}}&\lesssim \tau_q^{-1}\|z_i-z_{i+1}\|_{C^{N+\alpha}}+\|u_i-u_{i+1}\|_{C^{N+\alpha}}\|u_i-u_{i+1}\|_{C^\alpha}\,\\
	&\lesssim \mathring{\delta}_{q+1}\ell_q^{-N-2\alpha}+\tau_q^2\mathring{\delta}_{q+1}^2\ell_q^{-2-N-4\alpha}\,.
\end{align*}
Next, differentiating the expression for $\mathring{\bar{R}}_q$ as in the proof of \cite[Proposition 4.4]{BDSV}, we obtain
\begin{align*}
\|(\partial_t+u_\ell\cdot\nabla)\mathring{\bar{R}}_q\|_{C^{N+\alpha}}&\lesssim \tau_q^{-2}\|z_i-z_{i+1}\|_{C^{N+\alpha}}+\tau_q^{-1}\|(\partial_t+u_\ell\cdot\nabla)(z_i-z_{i+1})\|_{C^{N+\alpha}}\\
&+\tau_q^{-1}\|u_\ell\|_{C^{1+\alpha}}\|z_i-z_{i+1}\|_{C^{N+\alpha}}+\tau_q^{-1}\|u_\ell\|_{C^{1+N+\alpha}}\|z_i-z_{i+1}\|_{C^{\alpha}}\\
&+\tau_q^{-1}\|u_i-u_{i+1}\|_{C^{N+\alpha}}\|u_i-u_{i+1}\|_{C^{\alpha}}\\
&+\|(\partial_t+u_\ell\cdot\nabla)u_i-u_{i+1}\|_{C^{N+\alpha}}\|u_i-u_{i+1}\|_{C^{\alpha}}\\
&+\|(\partial_t+u_\ell\cdot\nabla)u_i-u_{i+1}\|_{C^{\alpha}}\|u_i-u_{i+1}\|_{C^{N+\alpha}}\,.
\end{align*}
Using again Proposition \ref{p:stability} we deduce
\begin{equation*}
	\|(\partial_t+u_\ell\cdot\nabla)\mathring{\bar{R}}_q\|_{C^{N+\alpha}}\lesssim \tau_q^{-1}\mathring{\delta}_{q+1}\ell_q^{-N-3\alpha}+\tau_q\mathring{\delta}_{q+1}^2\ell_q^{-2-N-4\alpha}\,.
\end{equation*}
Finally, following the proof of \cite[Proposition 4.5]{BDSV} we have
\begin{equation*}
\left|\frac{d}{dt}\int_{\T^3}|u_i|^2-|u_\ell|^2\,dx\right|\lesssim \|u_\ell\|_{C^1}\|\mathring{R}_\ell\|_{C^0}\lesssim \mathring{\delta}_{q+1}\delta_q^{\sfrac12}\lambda_q,
\end{equation*}
so that 
\begin{equation*}
\left|\int_{\T^3}|u_i|^2-|u_\ell|^2\,dx\right|\lesssim \tau_q\mathring{\delta}_{q+1}\delta_q^{\sfrac12}\lambda_q.
\end{equation*}
On the other hand
\begin{equation*}
\int_{\T^3}|u_{i}-u_{i+1}|^2\,dx\lesssim \|u_i-u_{i+1}\|^2_{C^\alpha}.	
\end{equation*}
Using the identity from the proof of \cite[Proposition 4.5]{BDSV} 
\begin{equation*}
|\bar u_q|^2-|u_\ell|^2=\chi_i(|u_i|^2-|u_\ell|^2)+(1-\chi_i)(|u_{i+1}|^2-|u_\ell|^2)-\chi_i(1-\chi_i)|u_i-u_{i+1}|^2\,,
\end{equation*}
and Proposition \ref{p:stability}, we deduce \eqref{e:gluingenergy}. 
\end{proof}

We conclude this section with the following summary of the mollification/gluing steps:

\begin{corollary}\label{c:gluing}
Let $(u_q,\mathring{R}_q)$ be a smooth solution of \eqref{e:EulerReynolds}  satisfying the inductive assumptions \eqref{e:R_q_inductive_est}-\eqref{e:energy_inductive_assumption}. Then there exists another smooth solution $(\bar{u}_q,\mathring{\bar{R}}_q)$ of \eqref{e:EulerReynolds} with the support condition \eqref{e:gluedsupport} such that the following estimates hold:
\begin{align}
\|\bar{u}_q\|_{C^{N+1}}&\lesssim \delta_q^{\sfrac12}\lambda_q\ell_q^{-N}\,\label{e:c_gluingu}\\
\|\mathring{\bar{R}}_{q}\|_{C^{N+\alpha}}&\lesssim \mathring{\delta}_{q+1}\ell_q^{-N-2\alpha}\,,\label{e:c_gluingR}\\
\|(\partial_t+\bar{u}_q\cdot\nabla)\mathring{\bar{R}}_{q}\|_{C^{N+\alpha}}&\lesssim \tau_q^{-1}\mathring{\delta}_{q+1}\ell_q^{-N-2\alpha}\,,\label{e:c_gluingDR}\\
\left|\int_{\T^3}|\bar{u}_q|^2-|u_q|^2\,dx\right|&\lesssim\mathring{\delta}_{q+1}\,,\label{e:c_gluingE}
\end{align}
and moreover the vector potentials satisfy
\begin{equation}\label{e:c_gluingz}
\|\bar{z}_{q}-z_q\|_{C^\alpha}\lesssim \tau_q\mathring{\delta}_{q+1}\ell_q^{-\alpha}\,.
\end{equation}
\end{corollary}

\begin{remark}\label{r:comparison}
It is useful to compare these estimates with the corresponding bounds obtained in the mollification/gluing steps in \cite{BDSV}, namely the bounds in 
\cite[(4.7), (4.10), (4.11), (4.12)]{BDSV}. For this comparison let us denote the respective parameters in \cite{BDSV} (defined in our case by the exponents $\gamma_R, \gamma_L, \gamma_t$) by $\mathring{\delta}_{q+1}^{old},\,\ell_q^{old},\,\tau_q^{old}$, so that, comparing with \cite[(2.7), (2.19), (2.26)]{BDSV}, we have 
\begin{equation*}
	\mathring{\delta}_{q+1}^{old}=3\alpha,\quad \ell_q^{old}=(b-1)\beta+\tfrac{3}{2}\alpha,\quad \tau_q^{old}=2\alpha(1+\gamma_L^{old}).
\end{equation*}
It is not difficult to see that \eqref{e:c_gluingu}, \eqref{e:c_gluingR} and \eqref{e:c_gluingE} are sharper bounds than the corresponding bounds \cite[(4.7), (4.10), (4.12)]{BDSV}, provided $\gamma_L<\gamma_L^{old}$ and $\gamma_R>3\alpha(1+\gamma_L^{old})$, in particular if
\begin{equation*}
	\gamma_L<(b-1)\beta\,\textrm{ and $\alpha>0$ is sufficiently small,}
\end{equation*}
in agreement with \eqref{e:comparison}. 
In contrast, estimate \eqref{e:c_gluingDR} would only be sharper than \cite[(4.11)]{BDSV} if $\gamma_R<\gamma_R^{(old)}$, a condition which we will not assume, because we will need a better bound from \eqref{e:c_gluingz}.
\end{remark}

\subsection{Perturbation step}\label{s:perturbation}

The construction of the new vector field $u_{q+1}=\bar{u}_q+w_{q+1}$ is performed in \cite[Section 5.2]{BDSV} and \cite[Section 5.3]{BDSV}. We start by recalling the main steps.

\subsubsection{}
First we define space-time cutoff functions $\eta_i$, adapted to the temporal support of $\mathring{\bar{R}}_q$ in \eqref{e:gluedsupport}, and supported in ``squiggling stripes'', as done in \cite[Lemma 5.3]{BDSV}. We start with the following construction, which is independent of $q$:
\begin{lemma}\label{l:Onsageretabar}
There exists two constants $c_0,c_1>0$ and a family of smooth nonnegative functions $\bar{\eta}_i\in C^\infty(\T^3\times \R)$ with the following properties:
\begin{enumerate}
\item[(i)] $0\leq \bar\eta_i(x,t)\leq 1$,
\item[(ii)] $\supp\bar\eta_i\cap\supp\bar\eta_j=\emptyset$ for $i\neq j$,
\item[(iii)] $\T^3\times (i+\tfrac13,i+\tfrac23)\subset \{(x,t):\bar\eta_i(x,t)=1\}$,
\item[(iv)] $\supp\bar\eta_i\subset \T^3\times (i-\tfrac13,i+\tfrac43)$.
\end{enumerate}
Moreover, the function $\bar\eta(x,t):=\sum_i\bar\eta_i(x,t)$ is $1$-periodic in $t$ and satisfies 
\begin{enumerate}
\item[(v)] $\dashint_{\T^3}\bar\eta^2(x,t)\,dx=c_0$ for all $t$,
\item[(vi)] $\dashint_0^1\bar\eta^2(x,s)\,ds=c_1$ for all $x$.
\end{enumerate}
\end{lemma}

\begin{proof}
We start by following the proof of \cite[Lemma 5.3]{BDSV} and choose a suitable $h\in C^{\infty}_c(0,1)$ such that, setting
\begin{equation}\label{e:hetai}
h_i(x,t):=h\left(t-\tfrac{1}{6}\sin(2\pi x_1)-i\right)
\end{equation}
the family of functions $\{h_i\}$ satisfies (i)-(iv) above, and there exists a geometric constant $c_0>0$ such that 
\begin{equation}\label{e:heta-xaverage}
\sum_i\dashint_{\T^3}h_i^2(x,t)\,dx\geq c_0\qquad\textrm{ for all }t.
\end{equation}
Then define
\begin{equation*}
\bar\eta_i(x,t)=	\left(\frac{1}{c_0}\sum_i\dashint_{\T^3}h_i^2(y,t)\,dy\right)^{-\sfrac12}h_i(x,t)\,,
\end{equation*}
and
\begin{equation*}
\bar\eta(x,t)=\sum_i\bar\eta_i(x,t).
\end{equation*}
Then $\bar\eta_i$ satisfies (i)-(iv), whereas
$t\mapsto \eta(x,t)$ is 1-periodic and satisfies (v). Finally, from \eqref{e:hetai} we see that  
$\bar\eta_i(x,t)=\bar\eta_i(0,t-\tfrac16\sin(2\pi x_1))$, and consequently 
$$
\bar\eta(x,t)=\bar\eta(0,t-\tfrac16\sin(2\pi x_1)).
$$
But then the $1$-periodicity of $t\mapsto\bar\eta(x,t)$ implies that $\int_0^1\bar\eta^2(x,t)\,dt$ is independent of $x$. This shows (vi). 
\end{proof}

The constants $c_0,c_1>0$ in Lemma \ref{l:Onsageretabar} determine our choice of $\bar{e}>0$:

\begin{definition}\label{d:defebar}
The constant $\bar{e}$ in \eqref{e:energy_inductive_assumption} is defined to be
$\bar{e}=\frac{3c_0}{c_1}$ for the universal constants in conditions (vi) and (vii) of Lemma \ref{l:Onsageretabar}.	
\end{definition}

Now we are ready to define the family of cutoff functions $\eta_i$, in analogy with \cite{BDSV}: let
\begin{equation}\label{e:defetai}
	\eta_i(x,t)=\bar\eta_i(x,\tau_q^{-1}t),\quad \eta(x,t)=\bar\eta(x,\tau_q^{-1}t).
\end{equation}
It is easy to see that $\eta_i$ have the properties:
\begin{itemize}
\item $0\leq \eta_i(x,t)\leq 1$,
\item $\supp\eta_i\cap\supp\eta_j=\emptyset$ for $i\neq j$,
\item $\T^3\times I_i\subset \{(x,t):\eta_i(x,t)=1\}$, where $I_i=(t_i+\tfrac13\tau_q,t_{i}+\tfrac23\tau_q)$,
\item $\supp\eta_i\subset \T^3\times \tilde I_i$, where $\tilde I_i:=(t_i-\tfrac13\tau_q,t_{i+1}+\tfrac43\tau_q)$,
\item for any $m,n\in\N$ we have the estimate
\begin{equation}\label{e:etai_est}
\|\partial_t^m\eta_i\|_{C^n}\lesssim \tau_q^{-m}.	
\end{equation}
\end{itemize}

\subsubsection{} The second step is to introduce a scalar function of time, which acts as the trace of the Reynolds stress tensor. In our case this will be defined as
\begin{equation}\label{e:sigmaq}
	\sigma_q(t):=\frac{1}{3c_0}\left(e(t)-\int_{\T^3}|\bar{u}_q|^2\,dx-\bar{e}\delta_{q+2}\right).
\end{equation}
We remark that the notation for this function in \cite{BDSV} is $\rho_q(t)$, but in this paper we reserve $\rho$ to denote density in subsequent sections. Moreover, in \cite{BDSV} the definition involves $\delta_{q+2}/2$ rather than $\bar{e}\delta_{q+2}$, this difference is related to our sharper inductive estimate \eqref{e:energy_inductive_assumption}. In particular, this leads to the following bound, which follows from \eqref{e:energy_inductive_assumption} and \eqref{e:gluingenergy}:  
\begin{equation}\label{e:sigmabound}
\left|\sigma_q(t)-\frac{\bar{e}}{3c_0}\delta_{q+1}\right|\lesssim \delta_{q+1}(\lambda_q^{-\gamma_E}+\lambda_q^{-\gamma_R}+\lambda_q^{-(b-1)\beta})\,.	
\end{equation}
Next, we introduce, as in \cite[Section 5.2]{BDSV} the localized versions of the Reynolds stress as
\begin{equation}\label{e:defRqi}
R_{q,i}=\eta_i^2(\sigma_{q}\Id-\mathring{\bar{R}}_q),\quad \tilde R_{q,i}=\frac{\nabla\Phi_i R_{q,i}\nabla\Phi_i^T}{\sigma_{q,i}},\quad \sigma_{q,i}=\eta_i^2\sigma_q,
\end{equation}
where $\Phi_i$ is the backward flow map for the velocity field $\bar{u}_q$, defined as the solution of the transport equation
\begin{align*}
(\partial_t + \bar{u}_q  \cdot \nabla) \Phi_i &=0 \\
\Phi_i(x,t_i) &= x.
\end{align*}
By our choice of $\eta_i$ and $\tau_q$ (cf.~\eqref{e:taucondition}) the backward flow $\Phi_i$ is well-defined in the support of $\eta_i$ and satisfies the estimate
\begin{equation}\label{e:backwardflow}
\|\nabla\Phi_i-\Id\|_{C^0}\lesssim \tau_q\|\bar{u}_q\|_{C^1}\lesssim \lambda_q^{-\gamma_T}.	
\end{equation}
In particular, we have the following analogue of \cite[Lemma 5.4]{BDSV}:

\begin{lemma}\label{l:sigma}
For $a\gg 1$ sufficiently large we have
\begin{align}
\|\nabla\Phi_i-\Id\|_{C^0}&\leq 1/2\quad\textrm{ for }t\in\tilde I_i,\label{e:backwardflow1}\\
|\sigma_q(t)-\tfrac13\bar{e}\delta_{q+1}|&\leq \tfrac{1}{9}\bar{e}\delta_{q+1}\quad\textrm{ for all }t, \label{e:sigma_range}
\end{align}
and for any $N\geq 0$ 
\begin{align}
 \norm{\sigma_{q,i}}_{C^N}&\lesssim \delta_{q+1}\,,\label{e:sigma_i_bnd_N}\\
 \norm{\partial_t \sigma_{q,i}}_{C^N} &\lesssim \delta_{q+1}\tau_q^{-1}\,.
 \label{e:sigma_i_bnd_t}
\end{align}
Moreover, we also have, for any $t\in \tilde I_i$
\begin{equation}\label{e:Rqi-I}
\left|\frac{R_{q,i}}{\sigma_{q,i}}-\Id\right|=\left|\sigma_{q}^{-1}\mathring{\bar{R}}_q\right|\lesssim \lambda_q^{-\sfrac{\gamma_R}{2}}.	
\end{equation}
In particular, for $a\gg 1$ sufficiently large and for all $(x,t)$
$$
\tilde R_{q,i}(x,t)\in B_{\sfrac12}(\Id)\subset \mathcal{S}^{3\times 3}_+\,,
$$
where $B_{\sfrac12}(\Id)$ denotes the Euclidean ball of radius $1/2$ around the identity $\Id$ in the space $\mathcal{S}^{3\times 3}$.

\end{lemma}

\begin{proof}
	The proof follows closely the proof of \cite[Lemma 5.4]{BDSV}. In particular \eqref{e:backwardflow1} follows from \eqref{e:backwardflow} and \eqref{e:sigma_range} follows from \eqref{e:sigmabound}. The estimates \eqref{e:sigma_i_bnd_N}-\eqref{e:sigma_i_bnd_t} can be obtained as in  \cite[(5.13)-(5.15)]{BDSV}. Indeed, we start with using equation \eqref{e:EulerReynolds} to estimate
	\begin{equation*}
\abs{\frac{d}{dt} \int \abs{\bar{u}_{q}(x,t)}^2\,dx}= \abs{2\int \nabla \bar{u}_q\cdot \mathring{\bar{R}}_q\,dx }\lesssim \delta_{q+1}\delta_q^{\sfrac 12}\lambda_q,
\end{equation*}
	so that
	\begin{equation*}
	|\tfrac{d}{dt}\sigma_q(t)|	\lesssim \|\tfrac{d}{dt}e\|_{C^0}+\delta_{q+1}\delta_q^{\sfrac 12}\lambda_q\lesssim \delta_{q+1}\tau_q^{-1},
	\end{equation*}
where we assume $a\gg 1$ is sufficiently large to absorb the term $\|\tfrac{d}{dt}e\|_{C^0}$. Then we use \eqref{e:etai_est} to conclude the bounds \eqref{e:sigma_i_bnd_N}-\eqref{e:sigma_i_bnd_t}. The estimate \eqref{e:Rqi-I} follows directly from \eqref{e:gluingR} and \eqref{e:gammaRalphacondition}. Consequently, the bound on the range of $\tilde R_{q,i}$ follows from \eqref{e:backwardflow} and by choosing $a\gg 1$ sufficiently large.
\end{proof}

\subsubsection{} With Lemma \ref{l:sigma} and the definitions in \eqref{e:defRqi} we define the new perturbation $w_{q+1}$, precisely as in \cite[Section 5.3]{BDSV}, as follows:\footnote{here we use the calculus identities $\curl[\nabla\Phi^TU(\Phi)]=\nabla\Phi^{-1}(\curl U)(\Phi)$ and $\curl(\varphi F)=\varphi \curl F+\nabla\varphi\times F$.}
\begin{equation}\label{e:neww}
\begin{split}
w_{q+1}&=\frac{1}{\lambda_{q+1}}\curl\left[\sum_{i}\sum_{\vec{k}\in\Lambda}\sigma_{q,i}^{\sfrac12}a_{\vec{k}}(\tilde R_{q,i})\nabla\Phi_i^TU_{\vec{k}}(\lambda_{q+1}\Phi_i)\right],\\
	&=\underbrace{\sum_{i}\sum_{\vec{k}\in\Lambda}\sigma_{q,i}^{\sfrac12}a_{\vec{k}}(\tilde R_{q,i})\nabla\Phi_i^{-1}W_{\vec{k}}(\lambda_{q+1}\Phi_i)}_{w_o}+\\
	&+\underbrace{\frac{1}{\lambda_{q+1}}\sum_{i}\sum_{\vec{k}\in\Lambda}\nabla (\sigma_{q,i}^{\sfrac12}a_{\vec{k}}(\tilde R_{q,i}))\times \nabla\Phi_i^TU_{\vec{k}}(\lambda_{q+1}\Phi_i)}_{w_c}
\end{split}
\end{equation}
Note that in the formulas above $\vec{k}\in \Lambda$ denotes vectors in $\R^3$ and the corresponding sum is finite. In contrast, the notation introduced in \cite{BDSV} is
\begin{equation}\label{e:wowc-old}
w_o=\sum_i\sum_{k\in\Z^3\setminus\{0\}}(\nabla\Phi_i)^{-1}b_{i,k}e^{\lambda_{q+1}k\cdot\Phi_i},\quad w_c=\sum_i\sum_{k\in\Z^3\setminus\{0\}}c_{i,k}e^{\lambda_{q+1}k\cdot\Phi_i},	
\end{equation}
where
\begin{equation*}
b_{i,k}=\sigma_{q,i}^{\sfrac12}a_{k}(\tilde R_{q,i})A_{k},\quad c_{i,k}=\frac{-i}{\lambda_{q+1}}\curl\left[\sigma_{q,i}^{\sfrac12}\frac{\nabla\Phi_i^T(k\times a_k(\tilde R_{q,i}))}{|k|^2}\right],
\end{equation*}
the index $k\in\Z^3\setminus\{0\}$ denotes the Fourier variable, and $A_k\in\C^3$ are complex vectors arising in the Fourier decomposition of Mikado flows, specifically of the functions $\psi_{\vec{k}}$ in \eqref{e:defUW}. In particular, since $\psi_{\vec{k}}$ is smooth, the Fourier coefficients $a_k$ in the expression for $b_{i,k},c_{i,k}$, together with all derivatives, and bounded and have polynomial decay in $k$ of arbitrary order (cf.~\cite[(5.5)]{BDSV}). At variance with \cite{BDSV} we will make use of this fact in the form
\begin{equation}\label{e:decayofak}
\|a_k(\tilde R_{q,i})\|_0\lesssim |k|^{-\bar{N}-3}\,.	
\end{equation}
The representation \eqref{e:wowc-old} is useful for obtaining estimates for $w_{q+1}$ and for the new Reynolds stress $\mathring{R}_{q+1}$, whereas the representation \eqref{e:neww} will be useful for computing the bulk diffusion coefficient induced by $w_{q+1}$ in Section \ref{s:homogenization}. 

\subsubsection{} As far as the estimates on $w_{q+1}, \mathring{R}_{q+1}$ are concerned, in light of Remark \ref{r:comparison} all estimates in \cite[Section 5.3-5.5]{BDSV} which do not use transport derivatives remain valid. These are (cf.~\cite[Lemma 5.5 and Proposition 5.7]{BDSV}):
\begin{lemma}\label{l:boundsonbc}
There is a geometric constant $\bar{M}$ such that
\begin{equation}\label{e:barM}
\|b_{i,k}\|_0 \leq \bar M\delta_{q+1}^{\sfrac{1}{2}}|k|^{-\bar{N}-3} \, .
\end{equation}
Moreover, for $t\in \tilde I_i$ and any $N\geq 0$
\begin{align}
\norm{ (\nabla\Phi_i)^{-1}}_{C^N} + \norm{\nabla\Phi_i}_{C^N} &\lesssim \ell_q^{-N} \,,\label{e:est_PhiN}\\
\norm{\sigma_{q,i}^{-1}R_{q,i}}_{C^N}+\norm{\tilde R_{q,i}}_{C^N} &\lesssim  \ell_q^{-N}\,,\label{e:est_Rqi}\\
\norm{b_{i,k}}_{C^N} &\lesssim \delta_{q+1}^{\sfrac12}|k|^{-\bar{N}-3}\ell_q^{-N}\,,\label{e:est_b} \\
\norm{c_{i,k}}_{C^N} &\lesssim  \delta_{q+1}^{\sfrac12}\lambda_{q+1}^{-1}|k|^{-\bar{N}-3}\ell_q^{-N-1}\,.\label{e:est_c}
\end{align}
\end{lemma}
\begin{proof}
The estimate \eqref{e:est_PhiN} follows from \eqref{e:backwardflow1} and \eqref{e:gluingu}. Let us denote $D_t=\partial_t+\bar{u}_q\cdot\nabla$. Since $D_t\Phi_i=0$ by definition, for $N\geq 1$ we have
\begin{align*}
\norm{D_t\nabla\Phi_i}_{C^N} &\lesssim \|\nabla\bar{u}_q^T\nabla\Phi_i\|_{C^N}\lesssim \|\nabla\bar{u}_q\|_{C^0}\|\nabla\Phi_i\|_{C^N}+\|\nabla\bar{u}_q\|_{C^N}\|\nabla\Phi_i\|_{C^0}\\
&\lesssim 
\tau_q^{-1}\|\nabla\Phi_i\|_{C^N}+\tau_q^{-1}\ell_q^{-N}\,.
\end{align*}
We deduce \eqref{e:est_PhiN} from here using Grönwall's inequality. 
Next, using \eqref{e:defRqi} we write 
\begin{equation}
	\sigma_{q,i}^{-1}R_{q,i}=\Id-\sigma_q^{-1}\mathring{\bar{R}}_q,\quad \tilde R_{q,i}=\nabla\Phi_i(\Id-\sigma_q^{-1}\mathring{\bar{R}}_q)\nabla\Phi_i^T\,.
\end{equation}
Then, applying \eqref{e:gluingR} and \eqref{e:sigma_range} we obtain
\begin{align*}
	\norm{\sigma_{q,i}^{-1}R_{q,i}}_{C^N}\lesssim \delta_{q+1}^{-1}\norm{R_{q,i}}_{C^{N+\alpha}}\lesssim \frac{\mathring{\delta}_{q+1}}{\delta_{q+1}}\ell_q^{-N-2\alpha}\lesssim \ell_q^{-N},
\end{align*}
where in the last inequality we used \eqref{e:gammaRalphacondition}. Similarly we obtain the estimate for $\tilde R_{q,i}$, leading to \eqref{e:est_Rqi}. 

The estimates \eqref{e:est_b} and \eqref{e:est_c} follow directly from \eqref{e:sigma_i_bnd_N} and \eqref{e:decayofak} as well as the above.
\end{proof}

Because we require inductive estimates on $\bar{N}\geq 1$ derivatives of $u_q$, \cite[Corollary 5.9]{BDSV} is replaced by
\begin{lemma}\label{l:boundsonw}
Under the assumption \eqref{e:comparison} and assuming $a\gg 1$ is sufficiently large, we have for any $N=0,1,\dots,\bar{N}$
	\begin{align*}
	\norm{w_o}_{C^N} &\leq \tilde M\delta_{q+1}^{\sfrac 12}\lambda_{q+1}^N\,,\\
\norm{w_c}_{C^N} &\lesssim \delta_{q+1}^{\sfrac 12}\ell_q^{-1}\lambda_{q+1}^{-1}\lambda_{q+1}^N\,,\\
\norm{w_{q+1}}_{C^N} &\leq  2\tilde M \delta_{q+1}^{\sfrac 12}\lambda_{q+1}^N\,,
	\end{align*}
where the constant $\tilde M$ depends on $\bar{N}$ and $\bar{M}$.
\end{lemma}

\begin{proof}
We use the representation in \eqref{e:wowc-old}. First of all, using the chain rule we obtain
\begin{equation*}
	\|e^{i\lambda_{q+1}k\cdot\Phi_i}\|_{C^m}\leq \lambda_{q+1}^m|k|^m\|\nabla\Phi_i\|_{C^0}^m+\sum_{j<m,\theta}C_{j,m}\lambda_{q+1}^j|k|^j\|\nabla\Phi_i\|_{C^0}^{\theta_1}\cdot\dots\cdot\|\nabla\Phi_i\|_{C^{m-1}}^{\theta_m}\,
\end{equation*}	
for some constants $C_{j,m}$ (binomial coefficients), 
where the sum is over $1\leq j\leq m$ and multi-indices $\theta$ with $m=\theta_1+2\theta_2+\dots+m\theta_m$ and $j=\theta_1+\dots+\theta_m$. Then, using Lemma \ref{l:boundsonbc} we deduce
\begin{equation*}
	\|e^{i\lambda_{q+1}k\cdot\Phi_i}\|_{C^m}\lesssim \lambda_{q+1}^m|k|^m+\lambda_{q+1}|k|\ell_q^{1-m}.
\end{equation*}
However, from \eqref{e:comparison} it follows in particular $\gamma_L<(b-1)$, hence $\ell_q^{-1}<\lambda_{q+1}$, so that we deduce
\begin{equation*}
	\|e^{i\lambda_{q+1}k\cdot\Phi_i}\|_{C^m}\lesssim \lambda_{q+1}^m|k|^m.
\end{equation*}
By applying the product rule and Lemma \ref{l:boundsonbc} we then conclude that there exists $\bar{M}$ such that 
\begin{equation*}
\|w_o\|_{C^m}\leq \tilde{M}\delta_{q+1}^{\sfrac12}\lambda_{q+1}^m\quad\textrm{ for all }m=0,1,\dots,\bar{N}.	
\end{equation*}
The estimate on $w_c$ follows directly from Lemma \ref{l:boundsonbc}. 
\end{proof}

\begin{definition}\label{d:defM}
The constant $M$ in \eqref{e:u_q_inductive_est} is defined as $M:=4\tilde M$, where $\tilde M$ is the constant in Lemma \ref{l:boundsonw}.  	
\end{definition}

Finally, coming to estimates involving time-derivatives, we have the following variant of \cite[Proposition 5.9]{BDSV}:
\begin{lemma}\label{l:boundsonbc1}
    For any $t\in \tilde I_i$ and $N\geq 0$ we have
    \begin{align*}
        \|D_t\nabla\Phi_i\|_{C^N}&\lesssim \delta_q^{1/2}\lambda_q\ell_q^{-N}\,,\\
        \|D_t\sigma_{q,i}\|_{C^N}&\lesssim \delta_{q+1}\tau_q^{-1}\ell_q^{-N}\,,\\
        \|D_t\tilde{R}_{q,i}\|_{C^N}&\lesssim \tau_q^{-1}\ell_q^{-N}\,,\\
        \|D_t c_{i,k}\|_{C^N}&\lesssim \delta_{q+1}^{\sfrac12}\tau_q^{-1}\lambda_{q+1}^{-1}\ell_q^{-N-1}|k|^{-\bar{N}-3} \,,
    \end{align*}
    where $D_t=\partial_t+\bar{u}_q\cdot\nabla$.
\end{lemma}
\begin{proof}
    The proof follows \cite[Proposition 5.9]{BDSV} using this time Lemma \ref{l:sigma} and Lemma \ref{l:boundsonbc}. In particular, using the expressions for the $D_t$ derivatives, we have
  \begin{align*}
      \|D_t\nabla\Phi_i\|_{C^N}&\lesssim \|\nabla\Phi_i\nabla \bar{u}_q\|_{C^N}\lesssim \delta_q^{\sfrac12}\lambda_q\ell_q^{-N}\lesssim \tau_q^{-1}\ell_q^{-N}\,,\\
      \|D_t\sigma_{q,i}\|_{C^N}&\lesssim \|\partial_t\sigma_{q,i}\|_{C^N}+\|\sigma_{q,i}\|_{C^{N+1}}\|\bar{u}_q\|_{C^1}+\|\sigma_{q,i}\|_{C^1}\|\bar{u}_q\|_{C^N}\\
      &\lesssim \delta_{q+1}\tau_q^{-1}+\delta_{q+1}\delta_q^{\sfrac12}\lambda_q+\delta_{q+1}\delta_q^{\sfrac12}\lambda_q\ell_q^{-N+1}\\
      &\lesssim \delta_{q+1}\tau_q^{-1}\ell_q^{-N}\,.
 \end{align*}
 Further, using \eqref{e:defRqi} we write $\sigma_{q,i}^{-1}R_{q,i}=\Id-\sigma_q^{-1}\mathring{\bar{R}}_q$	and compute
 \begin{align*}
 	        \|D_t(\sigma_{q,i}^{-1}R_{q,i})\|_{C^N}&\lesssim \|\sigma_{q}^{-1}\partial_t\sigma_{q}\mathring{\bar{R}}_q\|_{C^{N+\alpha}}+\|\sigma_{q}^{-1}D_t\mathring{\bar{R}}_q\|_{C^{N+\alpha}}\\
      &\lesssim \frac{\mathring{\delta}_{q+1}}{\delta_{q+1}}\tau_q^{-1}\ell_q^{-N-2\alpha}\lesssim \tau_q^{-1}\ell_q^{-N}\,,
 \end{align*}
 where we again used \eqref{e:gammaRalphacondition}.
Then, using Lemma \ref{l:boundsonbc},
 \begin{align*}
      \|D_t\tilde R_{q,i}\|_{N}\lesssim& \|D_t\nabla\Phi_i\|_{C^N}\|\sigma_{q,i}^{-1}R_{q,i}\|_{C^0}+ \|D_t\nabla\Phi_i\|_{C^0}\|\sigma_{q,i}^{-1}R_{q,i}\|_{C^{N}}+\\
      &+ \|D_t\nabla\Phi_i\|_{C^0}\|\sigma_{q,i}^{-1}R_{q,i}\|_{C^0}\|\nabla\Phi_i\|_{C^N} +\\
      &+ \|D_t(\sigma_{q,i}^{-1}R_{q,i})\|_{C^N}+\|D_t(\sigma_{q,i}^{-1}R_{q,i})\|_{C^0}\|\nabla\Phi_i\|_{C^N}\\
      \lesssim& \tau_q^{-1}\ell_q^{-N}\,. 
  \end{align*}  
  The estimate for $D_t c_{i,k}$ follows again from \eqref{e:decayofak}, Lemma \ref{l:sigma}, Lemma \ref{l:boundsonbc} and the above.
\end{proof}

\subsubsection{} Having obtained the analogous estimates for the perturbation $w_{q+1}$, the estimates on the new Reynolds stress $\mathring{R}_{q+1}$ proceed precisely as in \cite[Section 6.1]{BDSV}. 
We set (cf.~\cite[(5.21)]{BDSV})
\begin{equation}\label{e:decompR}
\mathring{R}_{q+1} =  \underbrace{\RR \left( w_{q+1} \cdot \nabla \bar u_q\right)}_{\mbox{Nash error}} + \underbrace{\RR \left( \partial_t  w_{q+1} + \bar u_q \cdot \nabla w_{q+1} \right)}_{\mbox{Transport error}} + \underbrace{\RR \div \left(- {\bar R}_{q} + (w_{q+1} \otimes w_{q+1}) \right)}_{\mbox{Oscillation error}},
\end{equation}
where 
\begin{align*}
\bar R_q = \sum_{i} R_{q,i}\, .
\end{align*}
With this definition one may verify that 
\begin{equation*}
\left\{
\begin{array}{l}
 \partial_t u_{q+1} + \div (u_{q+1} \otimes u_{q+1}) + \nabla p_{q+1} = \div(\mathring{R}_{q+1}) \, ,
\\ \\
 \div v_{q+1} = 0 \, ,
\end{array}\right.
\end{equation*}
where the new pressure is given by
\begin{equation*}
p_{q+1}(x,t) = \bar p_q(x,t)  - \sum_{i} \sigma_{q,i}(x,t)  + \sigma_{q}(t).
\end{equation*}

The analogue of \cite[Proposition 6.1]{BDSV} for estimating the new Reynolds stress is 
\begin{proposition}\label{p:newReynolds}
   The Reynolds stress error $\mathring R_{q+1}$ satisfies the estimate
\begin{equation}\label{e:final_R_est}
\norm{\mathring R_{q+1}}_{0}\lesssim \frac{\delta_{q+1}^{\sfrac12}}{\tau_q\lambda_{q+1} ^{1-\alpha}} \,.
\end{equation}
    
\end{proposition}

\begin{proof}
	We follow the proof of \cite[Proposition 6.1]{BDSV} and estimate each term in \eqref{e:decompR} separately. 
	
	Concerning the \emph{Nash error}, we have, as in \cite{BDSV}, for any $N\in\N$
	\begin{align*}
 \norm{\mathcal R\left(w_{q+1} \cdot \nabla \bar u_q \right)}_{\alpha}&\lesssim \sum_{k\in\Z^3\setminus\{0\}}\frac{\delta_{q+1}^{\sfrac12} \delta_q^{\sfrac12}\lambda_q }{\lambda_{q+1}^{1-\alpha}|k|^{\bar{N}+3}} + \frac{ \delta_{q+1}^{\sfrac12} \delta_q^{\sfrac12}\lambda_{q}}{\lambda_{q+1}^{N-\alpha}\ell_q^{N+\alpha} |k|^{\bar{N}+3}}\\
 &+\sum_{k\in\Z^3\setminus\{0\}}\frac{\delta_{q+1}^{\sfrac12}\delta_{q}^{\sfrac12}\lambda_q}{\ell_q\lambda_{q+1}^{2-\alpha}|k|^{\bar{N}+3}}+\frac{\delta_{q+1}^{\sfrac12}\delta_{q}^{\sfrac12}\lambda_q}{\ell_q^{N+1-\alpha}\lambda_{q+1}^{N+1-\alpha}|k|^{\bar{N}+3}},
\end{align*}
where we have used the representation \eqref{e:wowc-old}, Lemma \ref{l:boundsonbc} and the stationary phase estimate \cite[Proposition C.2]{BDSV}. We claim that it is possible to choose $N\geq 1$ so that
\begin{equation*}
\lambda_{q+1}^{N-1}\ell_q^{N+\alpha}>1.
\end{equation*}
Indeed, using \eqref{e:comparison}, this follows provided
\begin{equation}\label{e:choiceofN}
N(b-1)(1-\beta)>b+\alpha(1+\gamma_L).	
\end{equation}
In turn, with this choice of $N$, using \eqref{e:comparison} again to obtain $\lambda_{q+1}\ell_q>1$, and using $\bar{N}\geq 2$, we deduce
\begin{equation}\label{e:Nash_est}
 \norm{\mathcal R\left(w_{q+1} \cdot \nabla \bar u_q \right)}_{\alpha}\lesssim \frac{\delta_{q+1}^{\sfrac12} \delta_q^{\sfrac12}\lambda_q }{\lambda_{q+1}^{1-\alpha}}\,.
\end{equation}

\smallskip

Concerning the \emph{transport error} we write, for $t\in \tilde I_i$,
\begin{equation}\label{e:wotransport}
\begin{split}
(\partial_t+\bar u_q\cdot \nabla) w_o =&
\sum_{i,k} (\nabla\bar u_q)^T(\nabla\Phi_i)^{-1} b_{i,k} e^{i\lambda_{q+1}k\cdot \Phi_i}\\&\quad +
\sum_{i,k} (\nabla\Phi_i)^{-1} (\partial_t+\bar u_q\cdot \nabla) \left(\sigma_{q,i}^{\sfrac12} a_k(\tilde R_{q,i})\right) e^{i\lambda_{q+1}k\cdot \Phi_i} \,.
\end{split}
\end{equation}
As in \cite{BDSV} we obtain, arguing again as above with a sufficiently large $N$ satisfying \eqref{e:choiceofN}, 
\begin{align*}
\norm{\mathcal R\left( (\nabla\bar u_q)^T(\nabla\Phi_i)^{-1} b_{i,k} e^{i\lambda_{q+1}k\cdot \Phi_i} \right)}_{\alpha} &\lesssim \frac{\delta_{q+1}^{\sfrac12} \delta_q^{\sfrac12}\lambda_q }{\lambda_{q+1}^{1-\alpha}|k|^{\bar{N}+3}},
\end{align*}
whereas, using Lemma \ref{l:boundsonbc1},
\begin{align*}
\norm{\mathcal R\left( (\nabla\Phi_i)^{-1} (\partial_t+\bar u_q\cdot \nabla) (\sigma_{q,i}^{\sfrac12} a_k(\tilde R_{q,i}))e^{i\lambda_{q+1}k\cdot \Phi_i} \right)}_{\alpha}
\lesssim \frac{\delta_{q+1} ^{\sfrac 12}}{\tau_q\lambda_{q+1}^{1-\alpha}|k|^{\bar{N}+3}}.
\end{align*}
Moreover, using \eqref{e:wowc-old}, we have
\begin{align*}
(\partial_t+\bar u_q\cdot \nabla) w_c =&
\sum_{i,k} \left((\partial_t+\bar u_q\cdot \nabla) c_{i,k }\right)e^{i\lambda_{q+1}k\cdot \Phi_i}
\end{align*}
and obtain, again using Lemma \ref{l:boundsonbc1} and arguing as above,
\begin{align*}
\norm{\mathcal R \left(\left((\partial_t+\bar u_q\cdot \nabla) c_{i,k }\right)e^{i\lambda_{q+1}k\cdot \Phi_i}\right)}_{\alpha}\lesssim & \frac{\delta_{q+1} ^{\sfrac 12}}{\tau_q\ell_q \lambda_{q+1}^{2-\alpha}|k|^{\bar{N}+3}}
\lesssim  \frac{\delta_{q+1} ^{\sfrac 12}}{\tau_q\lambda_{q+1}^{1-\alpha}|k|^{\bar{N}+3}}.
\end{align*}
Since $w_{q+1}= w_o + w_c$, we deduce
\begin{equation}\label{e:trans_est}
\|\RR \left( \partial_t  w_{q+1} + \bar u_q \cdot \nabla w_{q+1} \right)\|_\alpha \lesssim
\frac{\delta_{q+1} ^{\sfrac 12}}{\tau_q\lambda_{q+1}^{1-\alpha}}\, .
\end{equation}

\smallskip
Concerning the \emph{oscillation error} we argue precisely as in \cite{BDSV} and obtain
\begin{equation}\label{e:osc_est}
\|\RR \div \left(- {\bar R}_q + w_{q+1}  \otimes  w_{q+1}\right)\|_\alpha \lesssim
\frac{\delta_{q+1}}{\ell_q\lambda_{q+1}^{1-\alpha}}\, .
\end{equation}

\smallskip

From \eqref{e:comparison} we deduce $\delta_{q+1}^{\sfrac12}\ell_q^{-1}<\delta_q^{\sfrac12}\lambda_q$. We also recall $\gamma_T>0$, hence $\tau_q<\delta_q^{\sfrac12}\lambda_q$. Consequently, combining \eqref{e:Nash_est}, \eqref{e:trans_est} and \eqref{e:osc_est} we finally deduce \eqref{e:final_R_est} as required.
\end{proof}

\subsubsection{} Finally, the new energy can be estimated, following \cite[Section 6.2]{BDSV}, as 
\begin{proposition}\label{p:newEnergy}
The energy of $u_{q+1}$ satisfies the following estimate:
\begin{equation}\label{e:final_E_est}
    \abs{e(t)-\dashint_{\T^3}\abs{u_{q+1}}^2\,dx-\bar{e}\delta_{q+2} }\lesssim \frac{\delta_q^{\sfrac12}\delta_{q+1}^{\sfrac12}\lambda_q}{\lambda_{q+1}}\,.
\end{equation}
\end{proposition}
\begin{proof}
We argue as in \cite{BDSV}. More precisely, we write 
\begin{align*}
\dashint_{\T^3}\abs{u_{q+1}}^2\,dx&=\dashint_{\T^3}\abs{\bar{u}_{q}}^2\,dx + 2\dashint_{\T^3} w_{q+1}\cdot \bar{u}_q\,dx+\dashint_{\T^3}\abs{w_{q+1}(x,t)}^2\,dx\\
&=\dashint_{\T^3}\abs{\bar{u}_{q}}^2\,dx +\dashint_{\T^3}\abs{w_o}^2\,dx+\mathcal{E}_1,
\end{align*}
where, arguing as in \cite{BDSV} using stationary phase and Lemma \ref{l:boundsonw},
\begin{align*}
|\mathcal{E}_1|=\left|\dashint 2w_{q+1}\cdot\bar{u}_q+2w_o\cdot w_c+|w_c|^2\,dx\right|\lesssim \frac{\delta_{q+1}^{\sfrac12}\delta_q^{\sfrac12}\lambda_q}{\lambda_{q+1}}.
\end{align*}
Similarly
\begin{equation*}
	\dashint_{\T^3}\abs{w_o}^2\,dx=\sum_i\dashint_{\T^3}\tr R_{q,i}\,dx+\mathcal{E}_2,
\end{equation*}
where, using Lemma \ref{l:sigma}, Lemma \ref{l:boundsonbc} and \eqref{e:comparison},  
\begin{align*}
|\mathcal{E}_2|\lesssim \frac{\delta_{q+1}}{\ell_q\lambda_{q+1}}\lesssim \frac{\delta_{q+1}^{\sfrac12}\delta_q^{\sfrac12}\lambda_q}{\lambda_{q+1}}. 	
\end{align*}
On the other hand, recalling the definition of $R_{q,i}$ in \eqref{e:defRqi} and using property (v) of $\bar\eta_i$ in Lemma \ref{l:Onsageretabar} as well as the definition of $\sigma_q(t)$ in \eqref{e:sigmaq}, we have
\begin{align*}
	\sum_i\dashint_{\T^3}\tr R_{q,i}(x,t)\,dx&=3\sigma_q(t)\sum_i\dashint_{\T^3}\eta_{i}^2\,dx=3c_0\sigma_q(t)=e(t)-\dashint_{\T^3}|\bar{u}_q|^2\,dx-\bar{e}\delta_{q+2}.
\end{align*}
The statement of the proposition follows. 
\end{proof}

We conclude this section with:
\begin{proof}[Proof of Proposition \ref{p:Onsager}]
	We need to verify that $u_{q+1}:=\bar{u}_q+w_{q+1}$, with $\bar{u}_q$ from Corollary \ref{c:gluing}, $w_{q+1}$ defined in \eqref{e:neww}, as well as $\mathring{R}_{q+1}$ defined in \eqref{e:decompR} satisfy the inductive estimates \eqref{e:R_q_inductive_est}-\eqref{e:energy_inductive_assumption} with $q$ replaced by $q+1$. 
	
	First of all note that \eqref{e:u_q_inductive_est} follows from Lemma \ref{l:boundsonw}, choice of $M$ in Definition \ref{d:defM}, and by choosing $a\gg 1$ sufficiently large. 
	
	Secondly, \eqref{e:R_q_inductive_est} follows from \eqref{e:final_R_est} and the inequality
	\begin{equation*}
	C\frac{\delta_{q+1}^{\sfrac12}}{\tau_q\lambda_{q+1}^{1-\alpha}}<\delta_{q+2}\lambda_{q+1}^{-\gamma_R},	
	\end{equation*}
	where $C$ is the implicit constant in \eqref{e:final_R_est}. In light of the inequality \eqref{e:transportcondition}, this is satisfied provided $a\gg 1$ is sufficiently large.
	Similarly, \eqref{e:energy_inductive_assumption} follows from \eqref{e:final_E_est} and the inequality 
    \begin{equation*}
	C\frac{\delta_{q+1}^{\sfrac12}\delta_q^{\sfrac12}\lambda_q}{\lambda_{q+1}}<\delta_{q+2}\lambda_{q+1}^{-\gamma_E},	
	\end{equation*}
	where $C$ is the implicit constant in \eqref{e:final_E_est}. In lilght of the inequality \eqref{e:energycondition} this is satisfied provided $a\gg 1$ is sufficiently large. 
	
	Finally, the estimate \eqref{e:v_diff_prop_est} follows directly from Lemma \ref{l:boundsonw}. This concludes the proof of Proposition \ref{p:Onsager}.
\end{proof}

\bibliographystyle{alpha}
\bibliography{anomalous.bib}

\end{document}